\documentclass{article}

\usepackage{microtype}
\usepackage{graphicx}
\usepackage{subfigure}
\usepackage{booktabs} 

\usepackage{hyperref}
\usepackage{enumitem}

\usepackage[accepted]{icml2024}

\usepackage{amsmath}
\usepackage{amssymb}
\usepackage{mathtools}
\usepackage{amsthm}

\usepackage[capitalize,noabbrev]{cleveref}
\usepackage{threeparttable, booktabs}

\theoremstyle{plain}
\newtheorem{theorem}{Theorem}[section]

\newtheorem{lemma}[theorem]{Lemma}

\theoremstyle{definition}

\newtheorem{assumption}[theorem]{Assumption}
\theoremstyle{remark}

\usepackage[textsize=tiny]{todonotes}
\usepackage{pifont}
\usepackage{bbding}
\usepackage{fontawesome}
\usepackage{wasysym}

\icmltitlerunning{Distributed Bilevel Optimization with Communication Compression}

\newcommand{\R}{\mathbb{R}}
\newcommand{\argmin}{\operatorname*{argmin}}
\newcommand{\cC}{\mathcal{C}}
\newcommand{\cD}{\mathcal{D}}
\newcommand{\E}{\mathbb{E}}
\newcommand{\ie}{\textit{i.e.}}
\newcommand{\EE}[1]{\mathbb{E}\left[\left\|#1\right\|_2^2\right]}
\newcommand{\cP}{\mathcal{P}}
\newcommand{\cO}{\mathcal{O}}
\newcommand{\eg}{\textit{e.g.}}
\newcommand{\MSC}[3]{\mathrm{MSC}\left(#1;#2,#3\right)}
\newcommand{\cF}{\mathcal{F}}
\newcommand{\cX}{\mathcal{X}}
\newcommand{\cY}{\mathcal{Y}}
\newcommand{\cZ}{\mathcal{Z}}
\newcommand{\Var}[1]{\mathrm{Var}\left[#1\right]}
\newcommand{\cB}{\mathcal{B}}
\allowdisplaybreaks

\definecolor{highlight_color}{rgb}{0.69, 0.97, 0.56}
\definecolor{babyblue}{rgb}{0.54, 0.81, 0.94}
\definecolor{mayablue}{rgb}{0.45, 0.76, 0.98}
\definecolor{springgreen}{rgb}{0.0, 1.0, 0.5}
\definecolor{aquamarine}{rgb}{0.5, 1.0, 0.83}
\definecolor{periwinkle}{rgb}{0.68, 0.7, 1.0}
\definecolor{orchid}{rgb}{0.85, 0.44, 0.84}
\definecolor{limegreen}{rgb}{0.2, 0.8, 0.2}

\newcommand{\highlighta}[1]{\colorbox{highlight_color}{$ #1$}}

\begin{document}

\twocolumn[
\icmltitle{Distributed Bilevel Optimization with Communication Compression}



\icmlsetsymbol{equal}{*}

\begin{icmlauthorlist}
\icmlauthor{Yutong He}{equal,pku}
\icmlauthor{Jie Hu}{equal,pku}
\icmlauthor{Xinmeng Huang}{upenn}
\icmlauthor{Songtao Lu}{ibm}
\icmlauthor{Bin Wang}{zju}
\icmlauthor{Kun Yuan}{pku,bigdata,aisi}
\end{icmlauthorlist}

\icmlaffiliation{pku}{Peking University}
\icmlaffiliation{bigdata}{National Engineering Labratory for Big Data Analytics and Applications}
\icmlaffiliation{aisi}{AI for Science Institute, Beijing, China}
\icmlaffiliation{upenn}{University of Pennsylvania}
\icmlaffiliation{ibm}{IBM Research}
\icmlaffiliation{zju}{Zhejiang University}
\icmlcorrespondingauthor{Kun Yuan}{kunyuan@pku.edu.cn}

\icmlkeywords{Bilevel Optimization, Distributed Optimization, Communication Compression}

\vskip 0.3in
]



\printAffiliationsAndNotice{\icmlEqualContribution} 

\begin{abstract}
Stochastic bilevel optimization tackles challenges involving nested optimization structures. Its fast-growing scale nowadays necessitates efficient distributed algorithms. In conventional distributed bilevel methods, each worker must transmit full-dimensional stochastic gradients to the server every iteration, leading to significant communication overhead and thus hindering efficiency and scalability. 
To resolve this issue, we introduce the {\em first} family of distributed bilevel algorithms with communication compression. The primary challenge in algorithmic development is mitigating bias in hypergradient estimation caused by the nested structure. We first propose C-SOBA, a simple yet effective approach with unbiased compression and provable linear speedup convergence. However, it relies on strong assumptions on bounded gradients. 
To address this limitation, we explore the use of moving average, error feedback, and multi-step compression in bilevel optimization, resulting in a series of advanced algorithms with relaxed assumptions and improved convergence properties. Numerical experiments show that our compressed bilevel algorithms can achieve $10\times$ reduction in communication overhead without severe performance degradation.
\end{abstract}

\section{Introduction}
\label{sec:intro}
Large-scale optimization and learning have emerged as indispensable tools in numerous applications. Solving such large and intricate problems poses a formidable challenge, usually demanding hours or days to complete. Consequently, it is imperative to expedite large-scale optimization and learning with distributed algorithms. In distributed learning, multiple workers collaborate to solve a global problem through inter-worker communications. In most current implementations \cite{smola2010architecture,li2014scaling,strom2015scalable,gibiansky2017bringing}, each worker 
transmits full-dimensional gradients to a central server for updating model parameters. Since the size of full-dimensional gradients is massive, communicating them per iteration incurs substantial overhead, which impedes algorithmic efficiency and scalability \cite{Seide20141bitSG,chilimbi2014project}. 

To mitigate this issue, communication compression  \cite{Alistarh2017QSGDCS, Bernstein2018signSGDCO, Stich2018SparsifiedSW, Richtrik2021EF21AN,huang2022lower} has been developed to reduce overhead. Instead of transmitting full gradient/model tensors, these strategies communicate compressed tensors with substantially smaller sizes per iteration. Two prevalent approaches of compression are quantization and sparsification. Quantization \cite{Alistarh2017QSGDCS, Horvath2019NaturalCF, Seide20141bitSG} involves mapping input tensors from a large, potentially infinite, set to a smaller set of discrete values, such as 1-bit quantization \cite{Seide20141bitSG} or natural compression \cite{Horvath2019NaturalCF}. In contrast, sparsification \cite{Wangni2018GradientSF, Stich2018SparsifiedSW, Safaryan2020UncertaintyPF} entails dropping a certain number of entries to obtain sparse tensors for communication, such as rand-$K$ or top-$K$ compressor \cite{Stich2018SparsifiedSW}. Both approaches have demonstrated strong empirical performance in communication savings.

Communication compression has widespread application in single-level stochastic optimization. However, many machine learning tasks, including adversarial learning \cite{madry2018towards}, meta-learning \cite{bertinetto2019meta}, hyperparameter optimization \cite{franceschi2018bilevel}, reinforcement learning \cite{hong2023two}, neural architecture search \cite{liu2018darts}, 
and imitation learning \cite{arora2020provable} involve upper- and lower-level optimization formulations that go beyond the conventional single-level paradigm. Addressing such nested problems has prompted substantial attention towards stochastic bilevel optimization \cite{ghadimi2018approximation,ji2021bilevel}. While tremendous efforts have been made~\cite{yang2022decentralized,chen2022decentralized,tarzanagh2022fednest,yang2023simfbo} to solve distributed stochastic bilevel optimization, {\em no existing algorithms, to the best of our knowledge, have been developed under communication compression}. To fill this gap, this paper provides the {\em first} comprehensive study on distributed stochastic bilevel optimization with communication compression.

\subsection{Distributed Bilevel Optimization}
We consider distributed stochastic bilevel problems with the  following 
nested upper- and lower-level structure: 
\begin{subequations}
\label{eq:prob}
\vspace{-5pt}
\begin{align}
\hspace{-1mm}\min_{x\in\R^{d_x}}\quad&\Phi(x)\triangleq \frac{1}{n}\sum_{i=1}^nf_i(x,y^\star(x)), \label{eq:prob-upper}\\
\hspace{-1mm}\mathrm{s.t.}\quad&y^\star(x)=\argmin_{y\in\R^{d_y}}\ g(x,y)\triangleq \frac{1}{n}\sum_{i=1}^ng_i(x,y).\label{eq:prob-lower}
\end{align}
\vspace{-15pt}
\end{subequations}

Here, $n$ denotes the number of workers, with each worker $i$ privately owns its 
upper-level cost function $f_i: \mathbb{R}^{d_x} \times \mathbb{R}^{d_y}  \rightarrow \mathbb{R}$, lower-level cost function $g_i: \mathbb{R}^{d_x} \times \mathbb{R}^{d_y} \rightarrow \mathbb{R}$, and local data distribution $\cD_{f_i}$, $\cD_{g_i}$ such that 
\begin{align*}\vspace{-2mm}
    f_i(x,y)\triangleq &\ \E_{\phi\sim\cD_{f_i}}[F(x,y;\phi)],\\
    g_i(x,y)\triangleq &\ \E_{\xi\sim\cD_{g_i}}[G(x,y;\xi)].\vspace{-3mm}
\end{align*}
The objective for all workers is to find a global solution to bilevel problem \eqref{eq:prob}. 
Typical applications of problem \eqref{eq:prob} can be found in \cite{yang2021provably,tarzanagh2022fednest}. 

\begin{table*}[ht]
	\footnotesize
	\centering
 \vspace{-5mm}
	\caption{
Comparison between distributed bilevel algorithms with communication compression. 
For simplicity, we unify the compression variance and heterogeneity bounds in both upper and lower levels.
Notation $n$ is the number of workers, $\epsilon$ is the target precision such that $\EE{\nabla\Phi(\hat{x})}\le\epsilon$, $\omega$ is compression-related parameter (see \cref{asp:unbia}), $\sigma^2$ is the variance of stochastic oracles, $b^2$ bounds the gradient dissimilarity. We also list the best-known single-level compression algorithm in the bottom line for reference.
 }
\vspace{1mm}
\begin{threeparttable}
\begin{tabular}{lcccccc}
\toprule
\textbf{Algorithms} & \textbf{\#A. Comp.}$^\dagger$ & \textbf{\#A. Comm.}$^\Diamond$ & \textbf{Single Loop} & \textbf{Mechanism}$^\ddagger$ & \textbf{Heter. Asp.}$^\ast$ & \textbf{Implement}$^\triangleleft$\\
\midrule
 C-SOBA (Alg.~\ref{alg:C-SOBA} green) &  {\large$\frac{(1+\omega)\sigma^2+\omega b^2}{n\epsilon^2}$} & { \large$\frac{\omega b^2}{n\epsilon^2}+\frac{1+\omega/n}{\epsilon}^\triangleright$} & \color{green}\ding{51} & --- & BG + BGD & {\color{green}{\Large \smiley{}}}\vspace{1mm}\\
    CM-SOBA (Alg.~\ref{alg:C-SOBA} pink) & {\large$\frac{(1+\omega)\sigma^2+\omega b^2}{n\epsilon^2}$} & {\large$\frac{\omega b^2}{n\epsilon^2}+\frac{1+\omega/n}{\epsilon}$} & \color{green}\ding{51} & MA & BGD & \hspace{0.05mm} {\color{green}{\Large \smiley{}}} \vspace{1mm}\\
    \hspace{3mm} {\sc {\small +MSC}} (Alg.~\ref{alg:CM-SOBA-MSC}) & {\large$\frac{\sigma^2}{n\epsilon^2}$} & {\large$\frac{1+\omega}{\epsilon}$} & \color{red}\ding{56} & MA + MSC & BGD & \hspace{0.05mm} {\color{green}{\Large \smiley{}}} \vspace{1mm}\\
    EF-SOBA (Alg.~\ref{alg:EF-SOBA}) & {\large$\frac{(1+\omega)^7\sigma^2}{n\epsilon^2}$} & {\large$\frac{1+\omega+\omega^5/n}{\epsilon}$} & \color{green}\ding{51} & EF + MA & {None} & {\color{red}\Large \frownie{}}\vspace{1mm}\\
    \hspace{3mm} 
    {\sc {\small +MSC}} 
 (Alg.~\ref{alg:EF-SOBA-MSC})   & {\large$\frac{\sigma^2}{n\epsilon^2}$} & {\large$\frac{1+\omega}{\epsilon}$} & \color{red}\ding{56} & EF + MSC & {None} & {\color{red}\Large \frownie{}} \\
    \midrule 
    \multicolumn{1}{l}{\begin{tabular}[c]{@{}l@{}}Single-level EF21-SGDM \\ \cite{fatkhullin2023momentum}\end{tabular}}
     & {\large$\frac{\sigma^2}{n\epsilon^2}$} & {\large$\frac{1+\omega}{\epsilon}$} & \color{green}\ding{51} & EF + MA & {None} & \hspace{0.05mm}{\color{red}{\Large \frownie{}}} \\
\bottomrule    
\end{tabular}
\begin{tablenotes}
	\item[$\Diamond$]Asymptotic communication complexity: number of  communication rounds when $\sigma\rightarrow0$ (smaller is better).
 \item[$\dagger$]Asymptotic computational complexity: number of gradient/Jacobian evaluations per worker when $\epsilon\rightarrow 0$ (smaller is better).
 \item[$\ddagger$]Compression mechanisms: ``MA'', ``EF'', and ``MSC'' refer to as moving average, error feedback, and multi-step compression. 
 \item[$\ast$]Data heterogeneity assumptions (fewer/milder is better). ``BG'' and ``BGD'' denote bounded gradients (\cref{asp:bound-updat}) and bounded gradient dissimilarity (Assumptions \ref{asp:bound-heter} and \ref{asp:parti}), respectively. ``BG'' is much more restrictive than ``BGD''.
 \item[$\triangleleft$]Easy to implement or not. {\color{green}{\smiley{}}} indicates ``easy to implement'' and {\color{red}{\frownie{}}} indicates ``relatively harder to implement'' due to the EF mechanism. 
 \item[$\triangleright$]With gradient upper bound $B_x$ (\cref{asp:bound-updat}), a more precise complexity is $\frac{\omega b^2}{n\epsilon^2}+\frac{\sqrt[3]{\omega(n+\omega)^2b^2B_x^4}}{n\epsilon^{4/3}}+\frac{(1+\omega/n)(1+B_x)}{\epsilon}$.
\end{tablenotes}
\end{threeparttable}
\label{tab:unbiased}
\vspace{-5mm}
\end{table*}

\vspace{-2mm}
\subsection{Challenges in Compressed Bilevel Optimization}


Conceptually, if each worker $i$ could access the accurate oracle function ${f}_i^\star(x) \triangleq  f_i(x,y^\star(x))$ without any sampling noise, a straightforward framework to solve \eqref{eq:prob} under compressed communication (with compressors $\{\cC_i\}_{i=1}^n$) is  
\vspace{-2mm}
\begin{align}
\label{eq:compression-framework}
    x^{k+1} = x^k - \frac{1}{n}\sum_{i=1}^n \cC_i\big(\nabla {f}_i^\star(x^k)\big),
\end{align}
where each worker transmits the compressed hypergradient to the server to update model parameters. However, update \eqref{eq:compression-framework} demands an accurate estimate of the hypergradient $\nabla {f}_i^\star(x)$, which can be written as \cite{ghadimi2018approximation}
\vspace{-2mm}
\begin{align}
\label{eq:hypergrad}
\nabla {f}_i^\star(x) =& \  \nabla_x f_i(x,y^\star(x)) - \Big( \nabla^2_{xy} g(x,y^\star(x))\ \cdot \nonumber \\
 &  \left[\nabla^2_{yy} g(x,y^\star(x))\right]^{-1} \cdot \nabla_y f_i(x,y^\star(x)) \Big) 
\end{align}
with
\vspace{-2mm}
\begin{subequations}
\label{eq:second-order-derivatives}
\begin{align}
\nabla^2_{xy} g(x,y^\star(x)) &= \frac{1}{n}\sum_{i=1}^n \nabla^2_{xy}~ g_i(x,y^\star(x)), \\
\nabla^2_{yy} g(x,y^\star(x)) &= \frac{1}{n}\sum_{i=1}^n \nabla^2_{yy}~ g_i(x,y^\star(x)). \label{eq:g22}
\end{align}
\end{subequations}
It is challenging to precisely evaluate $\nabla {f}_i^\star(x)$ through \eqref{eq:hypergrad} in distributed bilevel optimization, particularly under compressed communication, for the following reasons: 
\begin{itemize}[leftmargin=10pt]
\vspace{-10pt}
    \item \textbf{Unavailable $y^\star(x)$.} 
    The solution $y^{\star}(x)$ to problem \eqref{eq:prob-lower} is not directly accessible. Existing literature \cite{ghadimi2018approximation,ji2021bilevel} often introduces iterative loops to approximately solve problem \eqref{eq:prob-lower}, leading to expensive computation costs and biased estimates of $y^\star(x)$.

\vspace{-5pt}
    \item \textbf{Inexact Hessian inversion.} Even provided with the accurate $y^\star(x)$, it is cumbersome to evaluate the global $\nabla^2_{xy} g(x,y^\star(x))$ and $[\nabla^2_{yy} g(x,y^\star(x))]^{-1}$  through \eqref{eq:second-order-derivatives} as it incurs expensive matrix communication. Recent works \cite{tarzanagh2022fednest,xiao2023communication} propose to communicate imprecise Hessian/Jacobian-vector products achieved by approximate implicit differentiation, which inevitably introduces bias in estimating $\nabla {f}_i^\star(x)$.

    \vspace{-5pt}
    \item \textbf{Compression-incurred distortion.} 
    As indicated by \eqref{eq:hypergrad} and \eqref{eq:second-order-derivatives}, 
    specific Jacobean matrices shall be communicated to tackle sub-problem \eqref{eq:prob-lower}.
    One may consider using the compressed proxies $\cC_i(\nabla^2_{yy}~g_i(x,y^\star(x)))$ and $\cC_i(\nabla^2_{xy}~g_i(x,y^\star(x)))$ to  replace Jacobians matrices in \eqref{eq:second-order-derivatives}. However, the compression incurs information distortion, which brings additional bias when evaluating $\nabla {f}_i^\star(x)$.


    \vspace{-10pt}
\end{itemize}

To summarize, practical bilevel algorithms with communication compression essentially perform
\begin{align}
\label{eq:compression-framework-bias}
    x^{k+1} = x^k - \frac{1}{n}\sum_{i=1}^n \cC_i\big(\nabla {f}_i^\star(x^k) + \textbf{bias}\big)
\end{align}
rather than \eqref{eq:compression-framework}, where the \textbf{bias} originates from the nested bilevel structure of \eqref{eq:prob}, as opposed to data sampling or communication compression. This bias term poses substantial challenges in developing distributed bilevel algorithms with communication compression. Most existing single-level compression techniques, including error feedback \cite{Stich2018SparsifiedSW,Richtrik2021EF21AN} and multi-step compression \cite{huang2022lower,he2023lower}, require unbiased estimates of gradients\footnote{While error feedback and multi-step compression accommodate biased compressors, they need accurate or unbiased gradients.} (\ie, $\nabla {f}_i^\star(x)$) every  iteration, and are thus not directly applicable to  bilevel problem \eqref{eq:prob}. This calls for the urgent need to develop new algorithms that can effectively mitigate the bias incurred by the nested bilevel structure, as well as new analyses to clarify how this bias impacts the convergence of compressed bilevel algorithms.

\subsection{Contributions and Main Results}
\textbf{Contributions.} 
This paper develops the first set of bilevel algorithms with communication compression. 

\begin{itemize}[leftmargin=10pt]
\vspace{-10pt}
    \item SOBA \cite{dagreou2022framework} is a newly introduced single-loop bilevel algorithm with lightweight communication and computation. While SOBA still suffers from biased hypergradient estimates, we surprisingly find applying unbiased compression directly to SOBA yields a simple yet effective compressed bilevel algorithm, which is denoted as distributed \underline{\bf SOBA} with communication \underline{\bf C}ompression, or \underline{\bf C-SOBA} for brevity. Under the strong assumption of bounded gradients, we establish its convergence as well as computational and communication complexities\footnote{Throughout the paper, the computational and communication complexities are referred to in an asymptotic sense, see Table~\ref{tab:unbiased}.}. 

\vspace{-2pt}
    \item While commonly adopted in literature, the bounded-gradient assumption of C-SOBA is restrictive. To address this limitation, we leverage \underline{\bf M}oving average to enhance the theoretical performance of C-SOBA, proposing the refined \underline{\bf CM-SOBA} method. CM-SOBA converges under the more relaxed assumption of bounded heterogeneity with improved complexities, compared to C-SOBA. 

    \vspace{-2pt}
    \item 
    The convergence of CM-SOBA  still relies on the magnitude of data heterogeneity.
    When local data distributions $\cD_{f_i}$ and $\cD_{g_i}$ differ drastically across workers,  the performance of C-SOBA and CM-SOBA substantially degrade. To mitigate this issue, we further incorporate \underline{\bf E}rror \underline{\bf F}eedback into CM-SOBA, leading to the \underline{\bf EF-SOBA} algorithm. EF-SOBA does not rely on any assumptions regarding data heterogeneity, making it suitable for applications with severe data heterogeneity.

    \vspace{-2pt}
    \item Finally, owing to the bias in \eqref{eq:compression-framework-bias} incurred by the nested structure, the established communication and computation complexities of the aforementioned compressed bilevel algorithms are less favorable compared to the best-known single-level compressed algorithms \cite{huang2022lower,fatkhullin2023momentum}. Consequently, we utilize multi-step compression to enhance the convergence of C-SOBA and EF-SOBA, attaining the same complexities as the best-known single-level compressed algorithms. 
\vspace{-10pt}
\end{itemize}

\textbf{Results in Table \ref{tab:unbiased}.} All established algorithms, along with their assumptions and complexities are listed in Table \ref{tab:unbiased}. It is noteworthy that the utilization of more advanced mechanisms relaxes assumptions and substantially improves complexities, albeit introducing increased intricacy in algorithmic structures and implementations. Furthermore, our
algorithms can achieve the same theoretical complexities as the best-known single-level compression algorithm \cite{fatkhullin2023momentum}, demonstrating its efficacy in overcoming the bias incurred by the nested structure in bilevel problems.

\textbf{Experiments.} Our numerical experiments demonstrate that the proposed algorithms can achieve  $10\times$ reduction in communicated bits, compared to non-compressed distributed bilevel algorithms, see Fig.~\ref{fig:hr_mnist_id-intr} and Sec.~\ref{sec:exp}. 

\begin{figure}[t!]
\centering
\hspace{-1mm}
\includegraphics[width=4.2 cm]{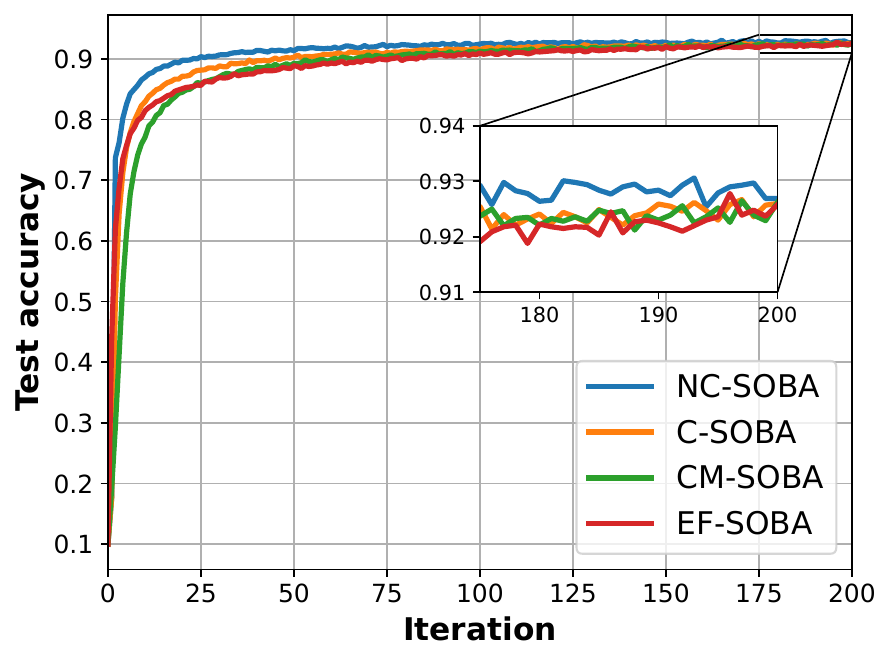}
\hspace{-3mm}
\includegraphics[width=4.2 cm]{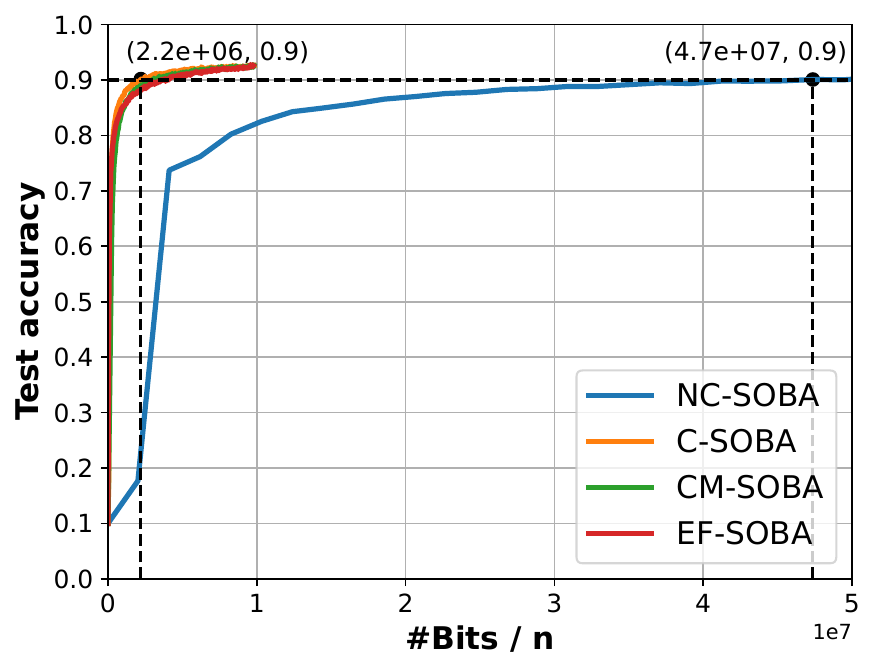}
\hspace{-1mm}
\vspace{-9mm}
\caption{\small Hyper-representation on MNIST under homogeneous data distributions. {{\em NC-SOBA indicates non-compressed SOBA.}}}
\label{fig:hr_mnist_id-intr}
\vspace{-8mm}
\end{figure}

\textbf{Analysis.} Our analysis also contributes new insights. They furnish convergence guarantees even when utilizing biased gradient estimates in compressors as shown in \eqref{eq:compression-framework-bias}. Additionally, they elucidate how upper- and lower-level compression exerts distinct influences on convergence, enlightening the compressor selection for 
upper- and lower-level problems. 


\vspace{-2mm}
\subsection{Related Work}
\textbf{Bilevel optimization.} 
A key challenge in bilevel optimization lies in accurately estimating the hypergradient $\nabla \Phi(x)$. Various algorithms have emerged to tackle this challenge, employing techniques such as approximate implicit differentiation \citep{domke2012generic,ghadimi2018approximation,grazzi2020iteration,ji2021bilevel}, iterative differentiation \citep{franceschi2018bilevel,maclaurin2015gradient,domke2012generic,grazzi2020iteration,ji2021bilevel}, and Neumann series \citep{chen2021closing,hong2023two}. However, these methods introduce additional inner loops 
that lead to increased computational overhead and deteriorated computation complexity. A recent work \cite{dagreou2022framework} introduces SOBA, a novel single-loop framework, 
to enable simultaneous updates of the lower- and upper-level variables. 
Recent efforts have been made to develop distributed bilevel algorithms within the federated learning setup~\cite{tarzanagh2022fednest,yang2021provably,huang2023achieving} and decentralized scenarios~\cite{yang2022decentralized,chen2022decentralized,chen2023decentralized,lu2022decentralized,gao2023convergence}. 
However, bilevel optimization with communication compression has not been studied in existing literature to our knowledge.



\vspace{-1mm}
\textbf{Communication compression.}
Communication compression shows notable success in single-level distributed optimization \cite{Alistarh2017QSGDCS, Bernstein2018signSGDCO, Stich2018SparsifiedSW}. Two main approaches are quantization and sparsification. Quantization strategies include Sign-SGD \cite{Seide20141bitSG,Bernstein2018signSGDCO}, TurnGrad \cite{Wen2017TernGradTG}, and natural compression \cite{Horvath2019NaturalCF}. On the other hand, classical sparsification strategies involve rand-$K$ and top-$K$ \cite{stich2019local,Wangni2018GradientSF}. 
Compression introduces information distortion,
which slows down convergence and incurs more communication rounds to achieve
desired solutions. Various advanced techniques such as error feedback \cite{Richtrik2021EF21AN,Stich2018SparsifiedSW},  multi-step compression \cite{huang2022lower,he2023lower}, and momentum \cite{fatkhullin2023momentum,huang2023stochastic} are developed to effectively mitigate the impact of compression-incurred errors. Furthermore, the optimal convergence for single-level distributed optimization with communication compression is established in \cite{huang2022lower,he2023lower}. However, none of these results have been established for bilevel stochastic optimization. 

\vspace{-2mm}
\section{Preliminaries}\label{sec:preli}
\textbf{Notations.} For a second-order continuously differentiable function $f:\mathbb{R}^{d_x}\times\mathbb{R}^{d_y}\to \mathbb{R}$, we denote $\nabla_x f(x, y)$ and $\nabla_y f(x, y)$ as the partial gradients in terms of $x$ and $y$, respectively. Correspondingly, $\nabla_{xy}^2 f(x, y) \in \mathbb{R}^{d_x \times d_y}$ and $\nabla_{yy}^2 f(x, y)\in \R^{d_y \times d_y}$ represent its Jacobian matrices. The full gradient of $f$ is represented as $\nabla f(x,y)\triangleq \left(\nabla_x f(x,y)^\top,\nabla_yf(x,y)^\top\right)^\top$.



\textbf{Basic assumptions.} 
We now introduce some basic assumptions needed throughout theoretical analysis.

\begin{assumption}[\sc Continuity]\label{asp:conti}
    For any $i$ $(1\le i\le n)$,
    \begin{itemize}[leftmargin=10pt]
    \vspace{-10pt}
        \item function $f_i$ is $C_f$-Lipschitz continuous with respect to $y$;
        \vspace{-5pt}
        \item functions $\nabla f_i$, $\nabla g_i$, $\nabla^2_{xy} g_i$, $\nabla^2_{yy} g_i$ are Lipschitz continuous with  constants $L_f$, $L_g$, $L_{g_{xy}}$, $L_{g_{yy}}$, respectively.
    \end{itemize}
\end{assumption}
\begin{assumption}[\sc Strong Convexity]\label{asp:stron-conve}
    For any $i$ $(1\le i\le n)$, $g_i$ is $\mu_g$-strongly convex with respect to $y$. \vspace{-2mm}
\end{assumption}
In the above assumptions, we allow data heterogeneity to exist across different workers, \ie, every $(f_i,g_i)$ differs from each other. Furthermore, we do not assume the Lipschitz continuity of $f_i$ with respect to $x$, which relaxes the assumptions used in prior works \cite{lu2022stochastic,yang2023achieving,huang2023achieving,chen2022decentralized,lu2022decentralized,yang2022decentralized}.

\begin{assumption}[\sc Stochastic noise]\label{asp:stoch}
    There exists $\sigma\ge0$ such that for any $i$ $(1\le i\le n)$, and any $x\in\R^{d_x}$, $y\in\R^{d_y}$, 
    \begin{itemize}[leftmargin=10pt]
    \vspace{-10pt}
        \item the gradient oracles satisfy:
        \begin{align*}
        \vspace{-5pt}
            &\E_{\phi\sim\cD_{f_i}}[\nabla F(x,y;\phi)]=\nabla f_i(x,y),\\
            &\E_{\xi\sim\cD_{g_i}}[\nabla_y G(x,y;\xi)]=\nabla_y g_i(x,y),\\
            &\E_{\phi\sim\cD_{f_i}}\left[\left\|\nabla F(x,y;\phi)-\nabla f_i(x,y)\right\|_2^2\right]\le\sigma^2,\\
            &\E_{\xi\sim\cD_{g_i}}\left[\left\|\nabla_y G(x,y;\xi)-\nabla_y g_i(x,y)\right\|_2^2\right]\le\sigma^2;
            \vspace{-5pt}
        \end{align*}
        \item the Jacobian oracles satisfy:
        \begin{align*}
            \vspace{-5pt}&\E_{\xi\sim\cD_{g_i}}\left[\nabla_{xy}^2G(x,y;\xi)\right]=\nabla_{xy}^2g_i(x,y),\\
            &\E_{\xi\sim\cD_{g_i}}\left[\nabla_{yy}^2G(x,y;\xi)\right]=\nabla_{yy}^2g_i(x,y),\\
            &\E_{\xi\sim\cD_{g_i}}\left[\left\|\nabla_{xy}^2G(x,y;\xi)-\nabla_{xy}^2g_i(x,y)\right\|_2^2\right]\le\sigma^2,\\
            &\E_{\xi\sim\cD_{g_i}}\left[\left\|\nabla_{yy}^2G(x,y;\xi)-\nabla_{yy}^2g_i(x,y)\right\|_2^2\right]\le\sigma^2.\vspace{-5pt}
        \end{align*}
    \end{itemize}
\end{assumption}
The following notion is for compressors.
\begin{assumption}[\sc Unbiased Compression]\label{asp:unbia}
    A compressor $\cC(\cdot):\R^{d_{\cC}}\rightarrow\R^{d_{\cC}}$ is $\omega$-unbiased $(\omega\ge0)$, if for any input $x\in\R^{d_{\cC}}$, we have
    \begin{align*}
    \E[\cC(x)]=x,\quad\mathrm{and}\quad\EE{\cC(x)-x}\le\omega\|x\|_2^2.       \vspace{-2mm}
    \end{align*}
\end{assumption}
Different compressors yield different values for $\omega$. Generally speaking, a large $\omega$ indicates more aggressive compression and, consequently, induces more information distortion. Below, we also assume the conditional independence among all local compressors, \ie, the outputs of local compressors are mutually independent, conditioned on the inputs.


\vspace{-2mm}
\section{Compressed SOBA }\label{sec:VCC-S}
SOBA \cite{dagreou2022framework} is a single-loop bilevel algorithm with lightweight communication and computational costs, originally devised for single-node optimization. In this section, we extend SOBA to address distributed bilevel optimization \eqref{eq:prob} and then incorporate communication compression, resulting in our first compressed bilevel algorithm.

\textbf{Non-compressed SOBA.} To address the Hessian-inversion issue when evaluating the hypergradient $\nabla \Phi(x)$ for problem \eqref{eq:prob}, SOBA  introduces $z^\star \triangleq  -[\nabla_{yy}^2g(x,y^\star(x))]^{-1}\cdot\nabla_yf(x,y^\star(x))$, which can be viewed as the solution to the following distributed optimization problem: 
$$
\frac{1}{n}\hspace{-0.5mm}\sum_{i=1}^n\hspace{-0.5mm}\left\{\hspace{-0.5mm}\frac{1}{2} z^\top\hspace{-0.5mm}\nabla_{yy}^2 g_i\hspace{-0.5mm}\left(x, y^\star(x)\right)z\hspace{-0.5mm}+\hspace{-0.5mm}z^{\top}\hspace{-0.5mm}\nabla_{y}f_i\hspace{-0.5mm}\left(x, y^\star(x)\right) \hspace{-0.5mm}\right\}.
$$
By simultaneously solving the lower-level problem, estimating the Hessian-inverse-vector product, and minimizing the upper-level problem, we achieve distributed recursions:
\begin{align}
    x^{k+1}=&\ x^k-\frac{\alpha}{n}\sum_{i=1}^n\left(\nabla_{xy}^2g_i(x^k,y^k)z^k+\nabla_xf_i(x^k,y^k)\right),\nonumber\\
    y^{k+1}=&\ y^k-\frac{\beta}{n}\sum_{i=1}^n\nabla_yg_i(x^k,y^k),
    \nonumber \\
    z^{k+1}=&\ z^k-\frac{\gamma}{n}\sum_{i=1}^n\left(\nabla_{yy}^2g_i(x^k,y^k)z^k+\nabla_yf_i(x^k,y^k)\right),\nonumber
\end{align}
where a central server collects local information to update global variables. We call this algorithm non-compressed SOBA (NC-SOBA). NC-SOBA accommodates stochastic variants by introducing noisy gradient/Jacobian oracles.  

\textbf{Compressed SOBA.} 
When each worker compresses its information before communicating with the central server, we obtain  \underline{\bf C}ompressed \underline{\bf SOBA}, or C-SOBA for short. To detail the algorithm, we let each worker $i$ independently sample data $\phi_i^k\sim\cD_{f_i}$ and $\xi_i^k\sim\cD_{g_i}$, and calculate
\begin{subequations}
\label{eq:local-Dx-Dy-Dz}
\begin{align}
D^k_{x,i} &\triangleq  \nabla_{xy}^2G(x^k,y^k;\xi_i^k) z^k + \nabla_x F(x^k,y^k;\phi_i^k), \label{eq:local-Dx} \\
D^k_{y,i} &\triangleq  \nabla_yG(x^k,y^k;\xi_i^k), \label{eq:local-Dy}\\
D^k_{z,i} &\triangleq  \nabla_{yy}^2G(x^k,y^k;\xi_i^k) z^k + \nabla_y F(x^k,y^k;\phi_i^k). \label{eq:local-Dz}
\end{align}
\end{subequations}
Next, worker $i$ transmits $\cC_i^u(D^k_{x,i})$, $\cC_i^\ell(D^k_{y,i})$, and $\cC_i^\ell(D^k_{z,i})$ to the central server where $\cC_i^u$ and $\cC_i^\ell$ are the upper-level $\omega_u$- and lower-level $\omega_\ell $-unbiased compressors utilized by worker $i$. The implementation of C-SOBA is listed in Algorithm~\ref{alg:C-SOBA} where the update of $x^{k+1}$ follows the green line. A Clip operation is conducted before updating $z$ to boost the algorithmic performance, in which 
$$
    \mathrm{Clip}(\tilde{z}^{k+1}; \rho)\triangleq \min\left\{1,\rho/\|\tilde{z}^{k+1}\|_2\right\}\cdot\tilde{z}^{k+1}.
$$



{\textbf{Intuition behind clipping operation.} The variable $z$ is intended to estimate $z^\star(x)$, which, under Assumptions \ref{asp:conti} and \ref{asp:stron-conve}, should not exceed a magnitude of $C_f/\mu_g$ (refer to \cref{lm:const}). However, during gradient descent steps, as in our algorithms, the update of $z$ may surpass this limit. Therefore, it is natural to opt for the nearest neighbor of $z$ with a magnitude no greater than $C_f/\mu_g$ instead. A rough estimation $\rho \geq C_f/\mu_g$ suffices as an upper bound. Alternatively, we can directly confine $z$ to the domain $\mathcal{B}(0,\rho)$, the closed ball of dimension $d_y$ centered at $0$ with a radius of $\rho$, using projected gradient descent. Both approaches result in the projection operation $z^{k+1} = \mathcal{P}_{\mathcal{B}(0,\rho)}(\tilde{z}^{k+1})$, equivalent to the clipping operation $z^{k+1} = \mathrm{Clip}(\tilde{z}^{k+1},\rho)$. Here, $\mathcal{P}_\Omega(\cdot)$ denotes projection onto a closed convex set $\Omega$.

The justification for this operation in theory is straightforward; $z^{k+1}$ is always closer than (or at an equal distance with) $\tilde{z}^{k+1}$ to $z^\star(x^{k+1})$.

In particular, assuming $\rho \geq C_f/\mu_g$, the non-expansiveness property of projection operators allows us to deduce:}
\vspace{5pt}
\begin{align*}
    &\left\|z^{k+1}-z^\star\left(x^{k+1}\right)\right\|_2\\
    =&\left\|\cP_{\cB(0,\rho)}\left(\tilde{z}^{k+1}\right)-\cP_{\cB(0,\rho)}\left(z^\star\left(x^{k+1}\right)\right)\right\|_2\\
    \le&\left\|\tilde{z}^{k+1}-z^\star\left(x^{k+1}\right)\right\|_2.
\end{align*}

\begin{algorithm}[tb]
    \caption{\highlighta{\hspace{-1mm}\mbox{C-SOBA}\hspace{-1mm}} and {\colorbox{pink}{\hspace{-1mm}CM-SOBA\hspace{-1mm}}}}
    \label{alg:C-SOBA}
    \begin{algorithmic}
        \STATE {\bfseries Input:} $\alpha$, $\beta$, $\gamma$, $\rho$, $x^0$, $y^0$, $z^0(\|z^0\|_2\le\rho), {\colorbox{pink}{$h_x^0$}}$;
        \FOR{$k=0,1,\cdots,K-1$}   
        \STATE \textbf{on each worker:}
        \STATE \quad Compute $D_{x,i}^k$, $D_{y,i}^k$, $D_{z,i}^k$ as in \eqref{eq:local-Dx-Dy-Dz};
        \STATE \quad Send $\cC_i^u(D_{x,i}^{k}),\cC_i^\ell(D_{y,i}^{k}),\cC_i^\ell(D_{z,i}^{k})$ to the server;
        \STATE \textbf{on server:}
        \STATE \quad $\highlighta{x^{k+1}=x^{k}-(\alpha/n)\sum_{i=1}^n\cC_i^u(D_{x,i}^{k})\ \  \mbox{(C-SOBA)}}$;
        \STATE \quad {\colorbox{pink}{$h_x^{k+1}=(1-\theta)h_x^k+(\theta/n)\sum_{i=1}^n\cC_i^u(D_{x,i}^{k})$}};
        \STATE \quad {\colorbox{pink}{$x^{k+1}=x^{k}-\alpha \cdot h_x^k \hspace{15mm} \mbox{(CM-SOBA)}$}};
        \STATE \quad $y^{k+1}=y^{k}-(\beta/n)\sum_{i=1}^n\cC_i^\ell(D_{y,i}^{k})$;
        \STATE \quad $\tilde{z}^{k+1}=z^{k}-(\gamma/n)\sum_{i=1}^n\cC_i^\ell(D_{z,i}^{k})$;
        \STATE \quad $z^{k+1}=\mathrm{Clip}(\tilde{z}^{k+1};\rho)$;
        \STATE \quad Broadcast $x^{k+1},y^{k+1},z^{k+1}$ to all workers;
        \ENDFOR
    \end{algorithmic}
\end{algorithm}

\textbf{Convergence.} To establish the convergence for C-SOBA, we need more assumptions beyond those discussed in Sec.~\ref{sec:preli}.
\begin{assumption}[\sc Bounded Heterogeneity]\label{asp:bound-heter}
    There exist constants $b_f\ge0$, $b_g\ge0$, such that for any $x\in\R^{d_x}$, $y\in\R^{d_y}$, it holds that
    \begin{align}
    \begin{split}
        \|\nabla f_i(x,y)-\nabla f(x,y)\|_2^2&\le b_f^2,\\
        \|\nabla_yg_i(x,y)-\nabla_yg(x,y)\|_2^2&\le b_g^2,\\
        \left\|\nabla_{xy}^2g_i(x,y)-\nabla_{xy}^2g(x,y)\right\|_2^2&\le b_g^2,\\
        \left\|\nabla_{yy}^2g_i(x,y)-\nabla_{yy}^2g(x,y)\right\|_2^2&\le b_g^2.
    \end{split}\label{eq:asp-bgd}
    \end{align}
\end{assumption}
For conciseness, we present results with the notation $b^2\triangleq \max\{b_f^2,b_g^2\}$ in the main text and defer the detailed counterparts associated with $b_f^2$ and $b_g^2$ to \cref{app:conve}.
\begin{assumption}[\sc Bounded Gradients]\label{asp:bound-updat}
    There exists constant $B_x\ge0$, such that for any $(x^k,y^k,z^k)$ {generated by C-SOBA(Alg.~\ref{alg:C-SOBA})}, we have
    \begin{align}
        \left\|\nabla_{xy}^2g(x^k,y^k)z^k+\nabla_xf(x^k,y^k)\right\|_2^2\le B_x^2.\label{eq:asp-bound-updat}
    \end{align}
\end{assumption}
It is noteworthy that the above assumption is milder than the $L_{f_x}$-Lipschitz continuity of $f$ with respect to $x$ \cite{lu2022stochastic,yang2023achieving,huang2023achieving,chen2022decentralized,lu2022decentralized,yang2022decentralized}, which implies $B_x\triangleq \sqrt{2}L_gC_f/\mu_g+\sqrt{2}L_{f_x}$.

\begin{assumption}[\sc 2nd-Order Smoothness]\label{asp:stron-conti}
    Jacobian matrices $ \nabla_{xy}^2g$, $\nabla_{yy}^2g$ are $L_{g_{xy\cdot}}$- and $L_{g_{yy\cdot}}$-smooth, {$\nabla^2f$ is $L_{ff}$-Lipschitz continuous.}
\end{assumption}

\begin{theorem}\label{thm:C-SOBA}
    Under Assumptions \ref{asp:conti}--\ref{asp:unbia} and \ref{asp:bound-heter}--\ref{asp:stron-conti}, if we set the hyperparameters as in Appendix \ref{app:C-SOBA}, C-SOBA converges as 
    \begin{align}
        &\frac{1}{K}\sum_{k=0}^{K-1}\EE{\nabla\Phi(x^k)}\nonumber\\
        =&\mathcal{O}\left(\frac{\sqrt{(1+\omega_\ell +\omega_u)\Delta}\sigma+\sqrt{(\omega_\ell +\omega_u)\Delta}b}{\sqrt{nK}}\right.\nonumber\\
        +&\left.\frac{\Delta^{\frac{3}{4}}((1+\omega_\ell )\sigma^2+\omega_\ell b^2)^{\frac14}\sqrt{1+\omega_u/n}B_x}{n^{1/4}K^{\frac{3}{4}}}\right.\nonumber\\
        +&\left.\frac{\Delta^{\frac34}\left((1+\omega_\ell )\sigma^2+\omega_\ell b^2\right)^{\frac14}\left((1+\omega_u)\sigma^2+\omega_ub^2\right)^{\frac12}}{(nK)^{3/4}}\right.\nonumber\\   
        +&\left.\frac{\sqrt{(1+\omega_u)(1+\omega_\ell /n)}\Delta\sigma+\sqrt{\omega_u(1+\omega_\ell /n)}b}{\sqrt{n}K}\right.\nonumber\\
        +&\left.\frac{\sqrt{(1+\omega_\ell /n)(1+\omega_u/n)}\Delta B_x}{K}\right.\nonumber\\
        +&\left.\frac{(1+\omega_\ell /n+\omega_u/n)\Delta}{K}\right),\label{eq:thm-VCC}
    \end{align}
    where $\Delta\triangleq \max\{\Phi(x^0),\|y^0-y^\star(x^0)\|_2^2,\|z^0-z^\star(x^0)\|_2^2\}$.
\end{theorem}

\textbf{Asymptotic complexities.} C-SOBA achieves asymptotic linear speedup with a rate of $\cO(1/\sqrt{nK})$ as the number of iterations $K \to \infty$. This corresponds to an asymptotic sampling/computational complexity of $\cO(1/(n\epsilon^2))$ as $\epsilon \to 0$. Furthermore, C-SOBA asymptotically requires $\cO(\omega/(n\epsilon^2))$ communication rounds to achieve an $\epsilon$-accurate solution when  $\epsilon \to 0$, see more details in Table \ref{tab:unbiased}. 

\textbf{Consistency with non-compression methods.} The leading term in \eqref{eq:thm-VCC} reduces to $\cO({\sqrt{\Delta}\sigma}/{\sqrt{nK}})$ when $\omega_\ell =\omega_u=0$, which is consistent with non-compression algorithms. 


{\textbf{A recommended choice of $\omega_u$ and $\omega_l$.}} According to \cite{he2023unbiased}, an $\omega$-unbiased compressor $\cC(x)$ with input dimension $d$ will transmit at least $\cO(d/(1+\omega))$ bits per communication round. If $\omega_u+\omega_\ell$ is bounded away from $0$, C-SOBA will asymptotically transmit  
$$
\cO\Bigg(\underbrace{\frac{(\omega_\ell +\omega_u)\Delta(\sigma^2+ b^2)}{n\epsilon^2}}_{\mbox{asym. comm. rounds} } \cdot \underbrace{\Big(\frac{d_x}{1+\omega_u}+\frac{d_y}{1+\omega_\ell }\Big)}_{\mbox{bits per round}}\Bigg) 
$$
bits to achieve an $\epsilon$-accurate solution. To minimize the communicated bits, {a recommended choice  for }$\omega_\ell $ and $\omega_u$ satisfies
\begin{align}\label{eq:optimal-omega}
(1+\omega_u)/(1+\omega_\ell )=\Theta\left(\sqrt{d_x/d_y}\ \right).
\end{align}
When using rand-$K$ compressors, we can set different values of $K$ for lower- and upper-level compression to achieve the { recommended} relation of $\omega_\ell $ and $\omega_u$ that satisfy \eqref{eq:optimal-omega}. { We refer the readers to the ablation experiments on different choices in Appendix~\ref{subapp:ablation}.}


\section{CM-SOBA Algorithm}\label{sec:VCM-S}
 Although simple and effective, C-SOBA relies on strong assumptions, particularly Assumption~\ref{asp:bound-updat} on bounded gradients. Moreover, the typically large upper bound $B_x$ significantly hampers the convergence performance. These strong assumptions and inferior convergence complexities are attributed to the bias introduced by the nested bilevel structure, as elucidated in \eqref{eq:compression-framework-bias}.

\textbf{CM-SOBA.} To enhance the convergence properties, we introduce a momentum procedure $h_x^{k+1}=(1-\theta)h_x^k+(\theta/n)\sum_{i=1}^n\cC_i^u(D_{x,i}^{k})$ to tweak the descent direction of $x$, \ie, $x^{k+1}=x^k-\alpha  h_x^k$. We refer to this new algorithm as \underline{\bf C}ompressed \underline{\bf SOBA} with \underline{\bf M}omentum, abbreviated as CM-SOBA. The implementation is outlined in Algorithm \ref{alg:C-SOBA}, where the update of $x^{k+1}$ is indicated by the pink color. {CM-SOBA is inspired by the momentum-based algorithm \cite{chen2023optimal} which eliminates the dependence on Assumption \ref{asp:bound-updat} in the single-node scenario. }

\textbf{Convergence.}
With an additional momentum step, CM-SOBA converges with more relaxed assumptions. In particular, it removes the strong assumption on bounded gradients. 

\begin{assumption}[\sc Point-wise Bounded Heterogeneity]\label{asp:parti}
There exist constants $b_f\ge0$, $b_g\ge0$, such that for any $x\in\R^{d_x}$, $y=y^\star(x)$, \eqref{eq:asp-bgd} holds.
\end{assumption}
Assumption \ref{asp:parti} is weaker than Assumption \ref{asp:bound-heter} since it only assumes bounds at $(x, y^\star(x))$, as opposed to arbitrary $(x,y)$. By writing $b^2\triangleq \max\{b_f^2,b_g^2\}$, we have the follows.
\begin{theorem}
\label{thm:CM-SOBA}
Under Assumptions \ref{asp:conti}--\ref{asp:unbia} and \ref{asp:parti}, if we set the hyperparameters as in Appendix \ref{app:CM-SOBA}, CM-SOBA converges as
    \begin{align}
        &\frac{1}{K}\sum_{k=0}^{K-1}\EE{\nabla\Phi(x^k)}\nonumber\\
=&\mathcal{O}\left(\frac{\sqrt{(1+\omega_\ell +\omega_u)\Delta}\sigma+\sqrt{(\omega_\ell +\omega_u)\Delta}b}{\sqrt{nK}}\right.\nonumber\\       +&\left.\frac{(1+\omega_\ell /n+\omega_u/n)\Delta}{K}\right),\label{eq:thm-VCM}
    \end{align}
    in which
    $\Delta\triangleq \max\{\Phi(x^0), \|h_x^0-\nabla\Phi(x^0)\|_2^2, \|y^0-y^\star(x^0)\|_2^2, \|z^0-z^\star(x^0)\|_2^2\}$.
\end{theorem}

\textbf{Improved complexities.} CM-SOBA achieves an asymptotic linear speedup rate under more relaxed assumptions, compared to C-SOBA. Furthermore, by eliminating the influence of the gradient upper bound $B_x$ on the convergence, we observe from \eqref{eq:thm-VCM} that CM-SOBA enjoys a faster rate, especially when $B_x$ is large.

\section{EF-SOBA Algorithm}\label{sec:MED-S}
Despite the simplicity, C-SOBA and CM-SOBA rely on the restrictive Assumptions \ref{asp:bound-heter} and \ref{asp:parti} concerning bounded data heterogeneity. When local data distributions $\cD_{f_i}$ and $\cD_{g_i}$ differ drastically across workers, the bounded-data-heterogeneity assumption can be violated, significantly deteriorating the convergence performance of C-SOBA and CM-SOBA. This section is devoted to developing compressed bilevel algorithms that are robust to data heterogeneity.

\textbf{Error feedback to upper-level compressors.} When updating $x$ in the upper-level optimization, each worker $i$ in C-SOBA or CM-SOBA transmits $\cC_i^u(D^k_{x,i})$ to the central server at each iteration $k$. Since $D^k_{x,i}$ does not approach zero due to the sampling randomness and data heterogeneity, $\cC_i^u(D^k_{x,i})$ does not converge to $D^k_{x,i}$ either. This reveals that compression-incurred distortion persists even when $k\to \infty$, explaining why C-SOBA or CM-SOBA necessitates Assumption \ref{asp:bound-heter} or \ref{asp:parti} to bound compression distortions. 

Inspired by \cite{fatkhullin2023momentum}, we employ error feedback to alleviate the impact of data heterogeneity when solving the upper-level optimization. Consider recursions 
\begin{subequations}
\label{eq:upper-ef}
\begin{align}
h_{x,i}^{k+1} & =(1-\theta)h_{x,i}^k+\theta\cdot D_{x,i}^k \label{eq:upper-ef-h}\\
m_{x,i}^{k+1}&= m_{x,i}^k+\delta_u\cdot\cC_i^u(h_{x,i}^{k+1}-m_{x,i}^k), \label{eq:upper-ef-m}\\
\hat{h}_x^{k+1}&=\hat{h}_x^k+\frac{\delta_u}{n}\sum_{i=1}^n\cC_i^u(h_{x,i}^{k+1}-m_{x,i}^k), \label{eq:upper-ef-global-h} \\
x^{k+1}&=x^{k}-\alpha\cdot \hat{h}_x^{k}, \label{eq:x-update}
\end{align}
\end{subequations}
in which $\delta_u$ is a positive scaling coefficient, $m^k_{x,i}$ is an auxiliary variable to track $h^k_{x,i}$, and $\hat{h}^k_x = (1/n)\sum_{i=1}^n m^k_{x,i}$ holds for any $k \ge 0$. In each iteration $k$, it is the difference $h^k_{x,i} - m^k_{x,i}$ that is compressed and transmitted instead of $h^k_{x,i}$ itself. When $m^k_{x,i}$ converges to a fixed point as $k\to \infty$, we have $m^k_{x,i} \to h^k_{x,i}$ from \eqref{eq:upper-ef-m} and hence $\hat{h}^k_x \to (1/n)\sum_{i=1}^n h^k_{x,i}$.  
In other words, error feedback \eqref{eq:upper-ef} removes the compression-incurred distortion asymptotically, making $x$ update along the exact momentum direction even when data heterogeneity exists.

\textbf{Error feedback to lower-level compressors.} One may naturally wonder whether the same error feedback technique \eqref{eq:upper-ef} can be used for the lower-level compressors, \ie, 
\begin{subequations}
\label{eq:lower-ef}
\begin{align}
m_{y,i}^{k+1}&=m_{y,i}^k+\delta_\ell \cdot\cC_i^\ell (D_{y,i}^k-m_{y,i}^k), \label{eq:lower-ef-my}\\
m_{z,i}^{k+1}&=m_{z,i}^k+\delta_\ell \cdot\cC_i^\ell (D_{z,i}^k-m_{z,i}^k),\label{eq:lower-ef-mz}\\
m_y^{k+1}&=m_y^k+\frac{\delta_\ell }{n}\sum_{i=1}^n\cC_i^\ell (D_{y,i}^k-m_{y,i}^k),\label{eq:lower-ef-my-global}\\
    m_z^{k+1}&=m_z^k+\frac{\delta_\ell }{n}\sum_{i=1}^n\cC_i^\ell (D_{z,i}^k-m_{z,i}^k),\label{eq:lower-ef-mz-global}
\end{align}
\end{subequations}
and let $y$ and $z$ update along the direction $m_y$ and $m_z$. The answer, however, is {\em negative}. Since the hypergradient $D_{x,i}$ used in $x$-update relies heavily on the accurate values of $y$ and $z$ as shown in \eqref{eq:local-Dx}, 
a more refined error feedback is needed for lower-level compressors. As illustrated in Appendix \ref{app:lower-unbia}, $y$ and $z$ must be updated following an unbiased estimate of their gradient direction. However, $m_y$ and $m_z$ provide biased estimates under the presence of $\delta_u$. To address this issue, we propose using
\begin{subequations}
\begin{align}
    \hat{D}_y^k=&m_y^k+\frac{1}{n}\sum_{i=1}^n\cC_i^\ell (D_{y,i}^k-m_{y,i}^k) \label{eq:Dy-hat}\\
    \hat{D}_z^k=&m_z^k+\frac{1}{n}\sum_{i=1}^n\cC_i^\ell (D_{z,i}^k)-m_{z,i}^k).\label{eq:Dz-hat}
\end{align}
\end{subequations}
to update $y$ and $z$ for the unbiasedness (\ie, $\E[\hat{D}_y^k]=\E[{D}_y^k]=\nabla_y G(x^k,y^k)$ and the similar applies to $z$). The resulting algorithm, termed as compressed \underline{\bf SOBA} with \underline{\bf E}rror \underline{\bf F}eedback, or EF-SOBA for short, is listed in Algorithm \ref{alg:EF-SOBA}. 

\begin{theorem}\label{thm:EF-SOBA}
    Under Assumptions \ref{asp:conti}--\ref{asp:unbia}, if hyperparameters are set as in Appendix \ref{app:EF-SOBA}, EF-SOBA converges as
    \begin{align}
        &\frac{1}{K}\sum_{k=0}^{K-1}\EE{\nabla\Phi(x^k)}\label{eq:thm-MED}\\
        &=\cO\left(\frac{(1+\omega_u)^2(1+\omega_\ell )^{3/2}\sqrt{\Delta}\sigma}{\sqrt{nK}}\right.\nonumber\\
        &+\frac{\omega_u^{1/3}(1+\omega_u)^{1/3}\Delta^{2/3}\sigma^{2/3}}{K^{2/3}}+\frac{(1+\omega_u)\Delta}{K}\nonumber\\
        &+\frac{(1+\omega_u)^2\sqrt{\omega_\ell (1+\omega_\ell )}\Delta}{\sqrt{n}K}+\left.\frac{(1+\omega_u)^2\omega_\ell ^3\Delta}{nK}\right)\nonumber 
    \end{align}
    where $\Delta$ represents algorithmic initialization constants detailed in Appendix \ref{app:EF-SOBA}. 
\end{theorem}
According to the above theorem, EF-SOBA achieves linear speedup convergence without relying on Assumption \ref{asp:bound-heter} or \ref{asp:parti}. Furthermore, the convergence of EF-SOBA is unaffected by data heterogeneity $b^2$.

\begin{algorithm}[tb]
    \caption{EF-SOBA}
    \label{alg:EF-SOBA}
    \begin{algorithmic}
        \STATE {\bfseries Input:} $\alpha$, $\beta$, $\gamma$, $\theta$, $\rho$, $\delta_u$, $\delta_\ell $, $x^0$, $y^0$, $z^0(\|z^0\|_2\le\rho)$, $\{m_{x,i}^0\}$, $\{m_{y,i}^0\}$, $\{m_{z,i}^0\}$, $\{h_{x,i}^0\}$, $\hat{h}_x^0=\frac{1}{n}\sum_{i=1}^nm_{x,i}^0$, $m_y^0=\frac{1}{n}\sum_{i=1}^nm_{y,i}^0$, $m_z^0=\frac{1}{n}\sum_{i=1}^nm_{z,i}^0$;
        \FOR{$k=0,1,\cdots,K-1$}   
        \STATE \textbf{on worker:}
        \STATE Compute $D_{x,i}^k$, $D_{y,i}^k$, $D_{z,i}^k$ as in \eqref{eq:local-Dx-Dy-Dz};
        \STATE Update $h_{x,i}^{k+1}$ and $m_{x,i}^{k+1}$ as in \eqref{eq:upper-ef-h} and \eqref{eq:upper-ef-m};
        \STATE Update $m_{y,i}^{k+1}$ and $m_{z,i}^{k+1}$ as in \eqref{eq:lower-ef-my} and \eqref{eq:lower-ef-mz};
        \STATE Send $\cC_i^u(h_{x,i}^{k+1}-m_{x,i}^k),\cC_i^\ell (D_{y,i}^{k}-m_{y,i}^k),\cC_i^\ell (D_{z,i}^{k}-m_{z,i}^k)$ to the server;
        \STATE \textbf{on server:}
        \STATE Update $\hat{D}_y^k$ and $\hat{D}_z^k$ as in \eqref{eq:Dy-hat} and \eqref{eq:Dz-hat};
        \STATE  $x^{k+1}=x^{k}-\alpha\cdot \hat{h}_x^{k};\quad y^{k+1}=y^{k}-\beta\cdot \hat{D}_y^k$;
        \STATE  $\tilde{z}^{k+1}=z^{k}-\gamma\cdot \hat{D}_z^k; \quad z^{k+1}=\mathrm{Clip}(\tilde{z}^{k+1},\rho)$;
        \STATE Update $\hat{h}_x^{k+1}$, $m_y^{k+1}$, $m_z^{k+1}$ as in \eqref{eq:upper-ef-global-h}, \eqref{eq:lower-ef-my-global}, \eqref{eq:lower-ef-mz-global}; 
        \STATE Broadcast $x^{k+1},y^{k+1},z^{k+1}$ to all workers;
        \ENDFOR
    \end{algorithmic}
\end{algorithm}

\begin{figure*}[t!]
\centering
\includegraphics[width=4.2 cm]{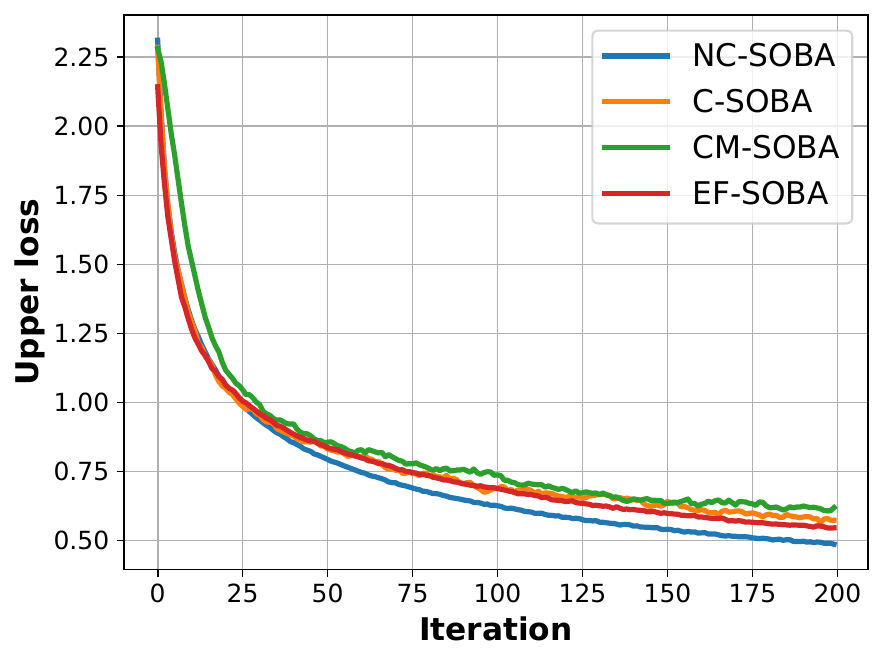}
\includegraphics[width=4.2 cm]{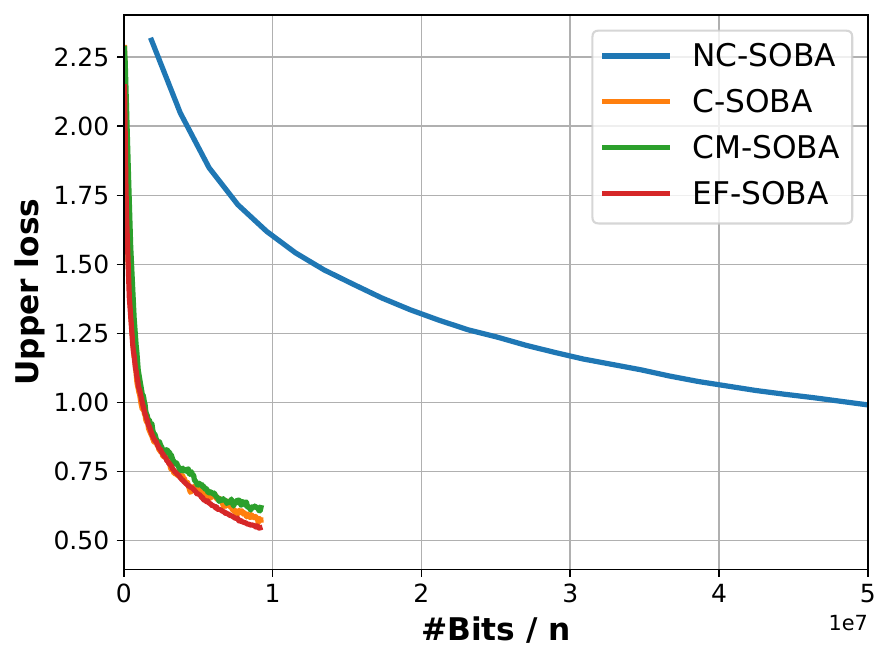}
\includegraphics[width=4.2 cm]{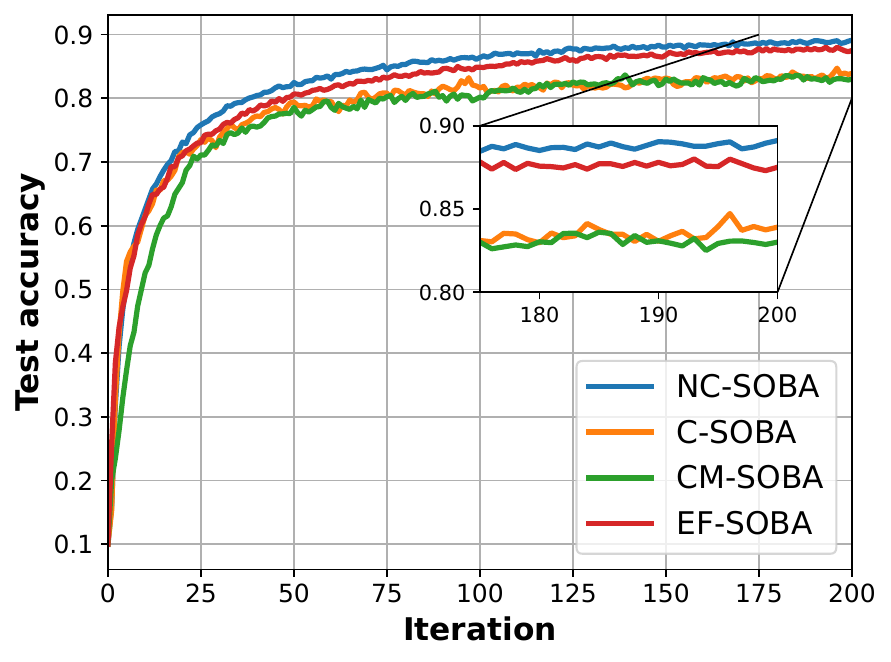}
\includegraphics[width=4.2 cm]{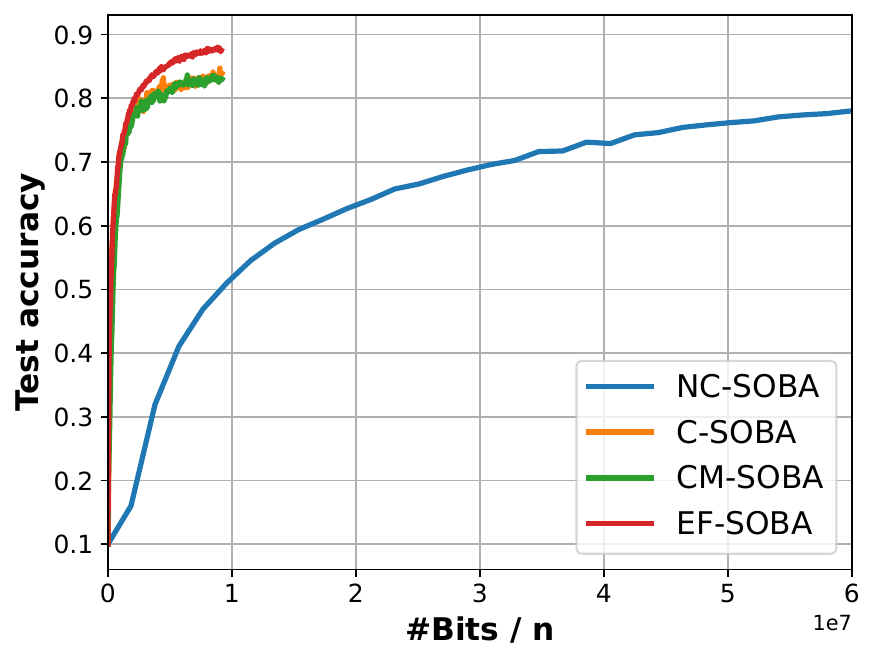}
\vspace{-4mm}
\caption{\small Hyper-representation on MNIST under heterogeneous data distributions.}
\label{fig:hr_mnist_nd}
\vspace{-3mm}
\end{figure*}
\begin{figure*}
\centering
\includegraphics[width=4.2 cm]{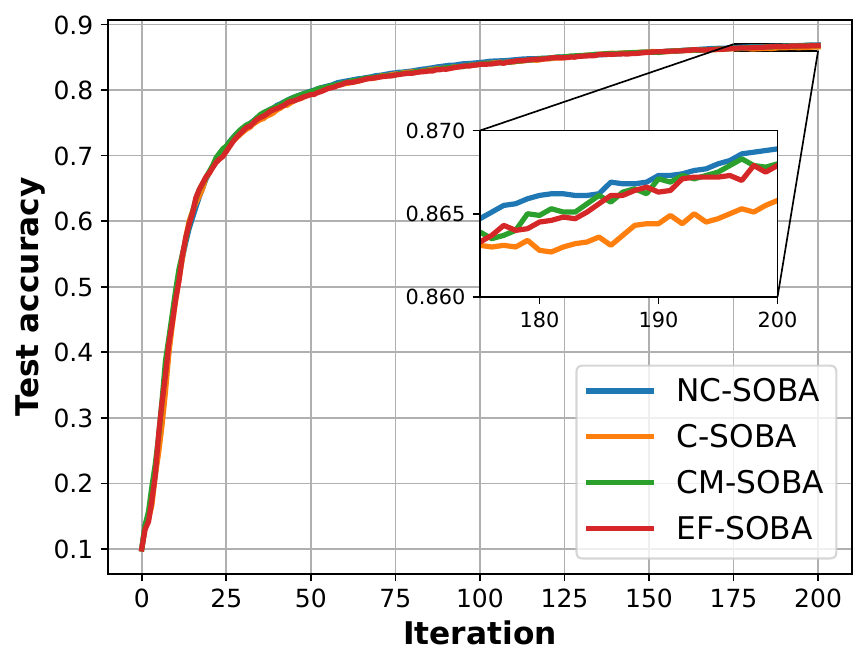}
\includegraphics[width=4.2 cm]{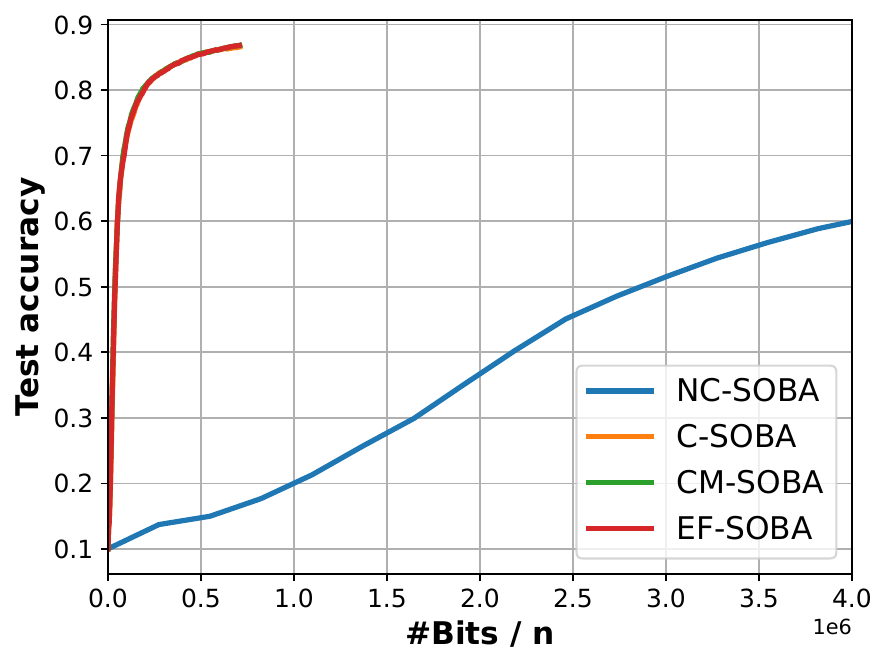}
\includegraphics[width=4.2 cm]{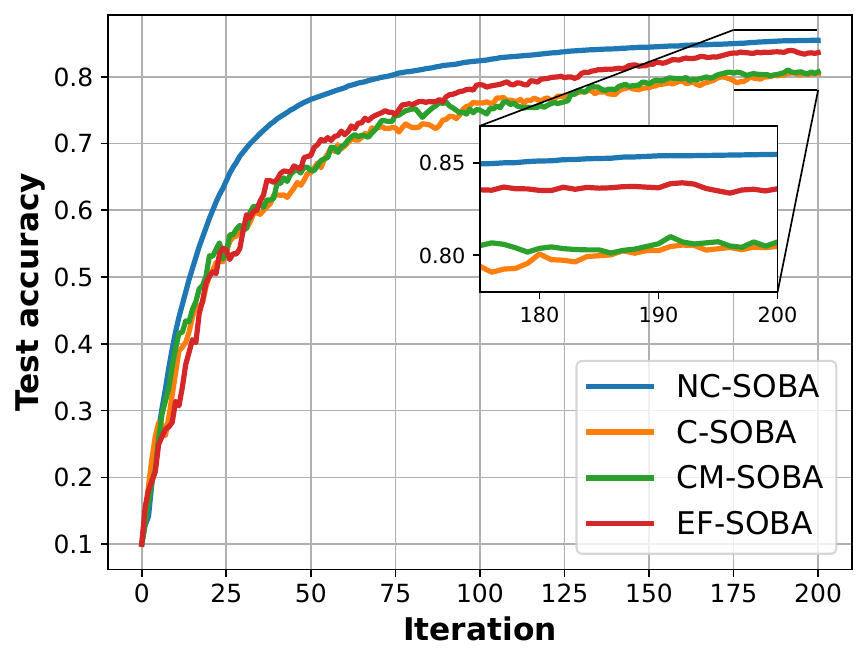}
\includegraphics[width=4.2 cm]{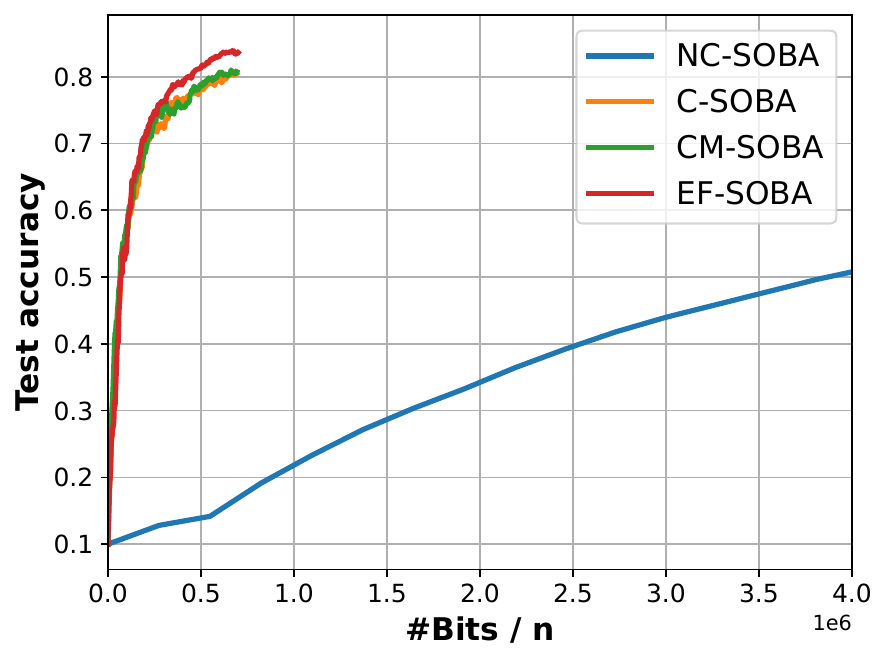}
        \label{fig:ho_mnist_acc_nd}
\vspace{-8mm}
\caption{\small Hyperparameter optimization on MNIST under homogeneous (left two figures) and heterogeneous (right two figures) data distributions.}
\label{fig:ho_mnist_id}
\vspace{-3mm}
\end{figure*}

\begin{figure*}[t!]
    \centering
    \begin{minipage}{0.4\textwidth}
        \centering
        \includegraphics[width=6 cm]{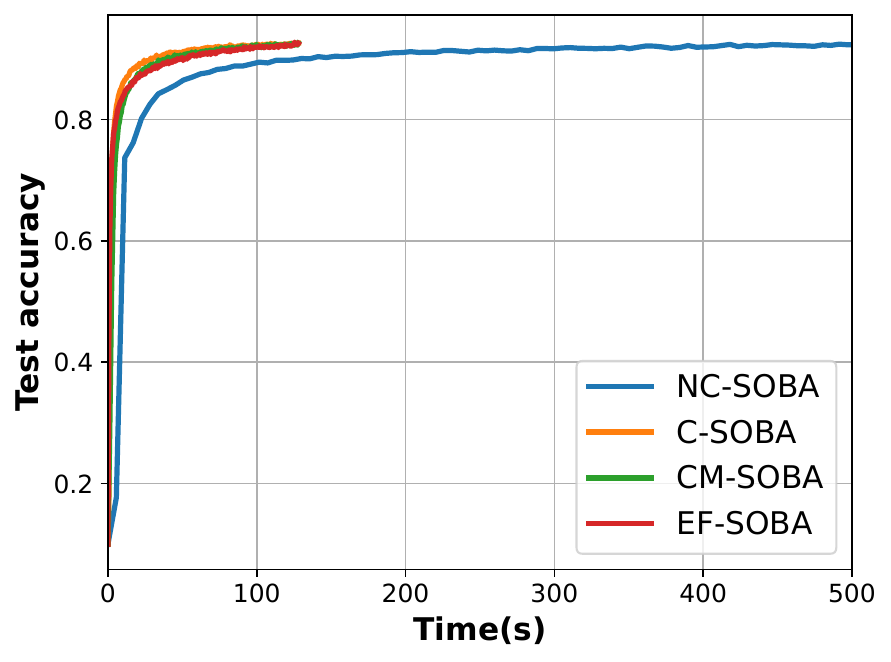}
        \label{fig:id_time}
    \end{minipage}
    \begin{minipage}{0.4\textwidth}
        \centering
        \includegraphics[width=6 cm]{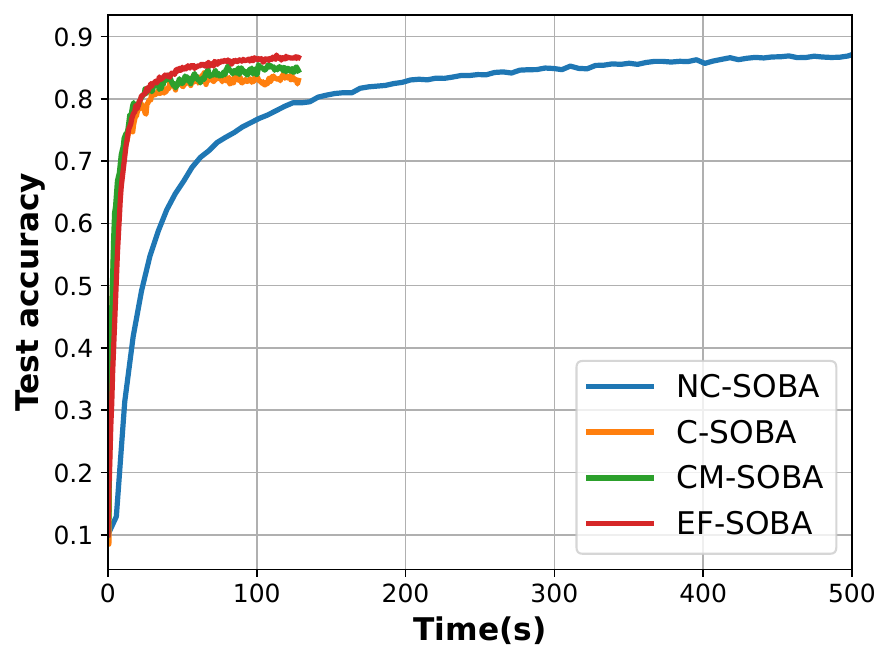}
        \label{fig:nd_time}
    \end{minipage}
    \caption{Running time comparison under homogeneous(left) and heterogeneous(right) data distributions.}
    \label{fig:runtime}
\end{figure*}


\section{Convergence Acceleration}\label{sec:varia}
While C-SOBA, CM-SOBA, and EF-SOBA can converge, their computational and communication complexities are inferior to those of the best-known single-level compression algorithms, such as EF21-SGDM \cite{loizou2020momentum} and NEOLITHIC \cite{huang2022lower}. In Appendix \ref{app:MSC}, we leverage \underline{\bf M}ulti-\underline{\bf S}tep \underline{\bf C}ompression  \cite{huang2022lower} to expedite CM-SOBA and EF-SOBA, leading to CM-SOBA-MSC and EF-SOBA-MSC, respectively, to achieve the same complexities as EF21-SGDM and NEOLITHIC, see Table \ref{tab:unbiased}. 
EF-SOBA-MSC converges as follows:
\begin{theorem}\label{thm:CM-SOBA-MSC}
    Under Assumptions \ref{asp:conti}--\ref{asp:unbia}, with hyperparameters in Appendix \ref{app:MSC}, EF-SOBA-MSC converges as
    \begin{align}
    \cO\left(\frac{\sqrt{\Delta}\sigma}{\sqrt{nT}}+\frac{(1+\omega_\ell +\omega_u)\Delta\tilde{\Theta}(1)}{T}\right),\label{eq:thm-MCM}
    \end{align}
    {where $T$ denotes the total number of iterations (i.e., \#outer-loop recursion $\times$ \#inner rounds in MSC) of EF-SOBA-MSC, }and $\tilde{\Theta}$ hides logarithm terms independent of $T$.
\end{theorem}
Similarly, CM-SOBA-MSC can converge at the same rate as given in \eqref{eq:thm-MCM}. For details on algorithmic development and convergence properties, please refer to Appendix \ref{app:CA}.

\vspace{-3mm}
\section{Experiments}
\label{sec:exp}
In this section, we evaluate the performance of the proposed compressed bilevel algorithms on two problems: hyper-representation and hyperparameter optimization. To demonstrate the impact of data heterogeneity on the compared algorithms, we follow the approach in \cite{hsu2019measuring} by partitioning the dataset using a Dirichlet distribution with a concentration parameter $\alpha=0.1$. 
We use the unbiased compressor, scaled rand-$K$, to compress communication, and $K$ is specified differently in various experiments.

\textbf{Hyper-representation.} Hyper-representation can be formulated as bilevel optimization in which the upper-level problem optimizes the intermediate representation parameter to obtain better feature representation on validation data, while the lower-level optimizes the weights of the downstream tasks on training data. We conduct the experiments on MNIST dataset with MLP and CIFAR-10 dataset with CNN. The detailed problem formulation and experimental setup can be found in Appendix \ref{app:exp_form}. 

Figure \ref{fig:hr_mnist_id-intr} compares C-SOBA, CM-SOBA, and EF-SOBA with {\em non-compressed} distributed SOBA (NC-SOBA) under homogeneous data distributions. It is observed that compressed algorithms can achieve a $10\times$ reduction in communicated bits without substantial performance degradation.  Figure \ref{fig:hr_mnist_nd} illustrates the performance under heterogeneous data distributions. It is observed that C-SOBA and CM-SOBA deteriorate in this scenario. However, error feedback significantly benefits convergence, and its convergence (in terms of iterations) and test accuracy of EF-SOBA are close to those of NC-SOBA. This is consistent with the theory that EF-SOBA is more robust to data heterogeneity. To justify its efficiency on various datasets and models, additional results of CIFAR-10 with CNN are provided in Appendix \ref{app:exp_cifar}.  { Furthermore, for more detailed discussions regarding the adjustment of the momentum parameter, please consult Appendix \ref{app:exp_cifar}.}

\textbf{Hyperparameter optimization.} Hyperparameter optimization aims to enhance the performance of learning models by optimizing hyperparameters. In our experiments on the MNIST dataset using an MLP model, the left two figures in Fig.~\ref{fig:ho_mnist_id} depict the test accuracy performance under homogeneous data distributions. It is observed that all compressed bilevel algorithms perform on par with the non-compressed algorithm but achieve a $10\times$ reduction in communicated bits. The right two figures illustrate the test accuracy performance under heterogeneous data distributions, further corroborating the superiority of EF-SOBA in resisting data heterogeneity. The problem formulation, experimental setup, and additional numerical results can be found in Appendix \ref{app:exp_syn}. 


{\textbf{Runtime comparison. }Understanding the importance of evaluating the practical trade-offs between computation and communication, we performed a runtime comparison to complement our theoretical findings with empirical evidence. The results of our hyper-representation experiment on the MNIST dataset with the MLP backbone are presented in Fig.~\ref{fig:runtime}, conducted under both homogeneous and heterogeneous data distributions. It is evident that the convergence with respect to running time closely aligns with those observed for communication bits. Additionally, we observed that the computation time across all compared algorithms closely matches. This result highlights the computational efficiency of our proposed compression techniques, which introduce minimal computational overhead. Our experiments unequivocally demonstrate that communication time is the dominant factor affecting total runtime in a distributed scenario.}

\vspace{-2mm}
\section{Conclusion and Limitation} 
This paper introduces the first set of distributed bilevel algorithms with communication compression and establishes their convergence guarantees. In experiments, these algorithms achieve a $10\times$ reduction in communication overhead.

However, our developed algorithms are only compatible with unbiased compressors, excluding biased but contractive compressors such as Top-$K$. In future work, we will explore bilevel algorithms with contractive compressors and investigate their convergence properties.

\section*{Impact Statement}
This paper centers on the theoretical analysis of machine learning algorithm convergence. We do not foresee any significant societal consequences arising from our work, none of which we believe need to be explicitly emphasized.

\section*{Acknowlegement}
The work of Kun Yuan is supported by Natural Science Foundation of China under Grants 12301392 and 92370121.



\bibliography{ref}
\bibliographystyle{icml2024}

\newpage
\appendix
\onecolumn

\section{Necessity of Unbiased Lower Level Gradients}\label{app:lower-unbia}
In this section, we demonstrate why the lower-level variables, $y$ and $z$ are supposed to be updated following an unbiased estimate of their gradient direction. 

First of all, it's noteworthy that the lower level in a bilevel optimization problem is not equivalent to a single level one. Recall the final goal of our bilevel algorithm is to optimize $\Phi(x)$, and $y$, $z$ are actually auxiliary variables to help improve the estimate precision of the hypergradient $\nabla\Phi(x)$. Consequently, by following \cref{lm:gradi}, $y^k$, $z^k$ are expected to be good estimation of $y^\star(x^k)$ and $z^\star(x^k)$, respectively. Thus, in single-loop algorithms, where lower/upper level variables are updated alternatively, the lower level cannot be regarded as a single-level optimization problem since $x^k$ varies in each iteration.

Even if we assume the stability of upper-level solution $x^k$, upper-level algorithms, \eg, EF21-SGDM which we use in EF-SOBA, cannot work well in the strongly-convex lower-level scenario. When we modify the non-convex objectives into a strongly-convex one, it's natural to believe that the algorithm can guarantee at least the same convergence rate as before, since the assumption has become stronger. However, it's worth noting that the convergence metric also varies in different settings. Specifically, the metric for a non-convex objective $f(x)$ is usually $\|\nabla f(\bar{x}^{K})\|_2^2$ with $\bar{x}^K=\frac{1}{K+1}\sum_{k=0}^{K}x^k$, while that for strongly-convex $f(x)$ should be $\|x^K-x^\star\|_2^2$ or $f(x^K)-f(x^\star)$, where $x^\star$ is the minimum of $f(x)$. In the typical SOBA structure, $\|y^k-y^\star(x^k)\|_2^2$ is the actually concerned metric, which negates applying EF21-SGDM for lower-level design.

Next, we go through some technical issues to demonstrate the role of unbiased gradient estimates in this \textit{different} optimization problem. Consider the following recursion for optimizing a strongly-convex  objective $f(x)$:
\begin{align*}
    x^{k+1}=x^k-\alpha g^k,
\end{align*}
where $g^k$ is a stochastic estimator of $\nabla f(x^k)$.
When trying to establish descent inequalities with respect to $\EE{x^k-x^\star}$, we consider 
\begin{align}
    \EE{x^{k+1}-x^\star}=&\EE{x^k-\alpha g^k-x^\star}=\EE{x^k-x^\star}-2\alpha\E\left[\left\langle x^k-x^\star,g^k\right\rangle\right]+\alpha^2\EE{g^k}.\label{eq:NULLG-1}
\end{align}
If $g^k$ is an unbiased estimate of $\nabla f(x^k)$, $\E\left[\left\langle x^k-x^\star,g^k\right\rangle\right]=\E\left[\left\langle x^k-x^\star,\nabla f(x^k)\right\rangle\right]$ can be directly used in proving descent inequalities, leaving $\alpha^2\EE{g^k}$ the only term including noise. Otherwise, if $g^k$ is a biased one, an additional term $2\alpha\E\left[\left\langle x^k-x^\star,g^k-\nabla f(x^k)\right\rangle\right]$ which is upper bounded by $\alpha^{2-t}\EE{x^k-x^\star}+\alpha^t\EE{g^k-\nabla f(x^k)}$ will either include bigger noise (if $t<2$) or deteriorate contraction of $\EE{x^k-x^\star}$ (if $t\ge2$).

\textbf{Remark.} Based on the above discussions, it is possible that non-convex algorithms like EF21-SGDM without unbiasedness can be applied to the lower level by using $\bar{y}^k$, $\bar{z}^k$ for hypergradient estimation. Exploration of this type of algorithms is left to future work.

\section{Convergence Analysis}\label{app:conve}
In this section, we provide proofs for the theoretical results in Sections \ref{sec:VCC-S}, \ref{sec:VCM-S}, \ref{sec:MED-S} and \ref{sec:varia}. Throughout this section, we have the following notations.

\textbf{Notation. }\vspace{-5pt}
\begin{itemize}
    \item $y_\star^k\triangleq y^\star(x^k)$, $z_\star^k\triangleq z^\star(x^k)$;
    \item $\cF^k$: the $\sigma$-field of random variables already generated at the beginning of iteration $k$;
    \item $\E_k[\cdot]\triangleq \E[\cdot\mid\cF^k]$;
    \item $F_i^k\triangleq f_i(x^k,y^k)$, $F^k\triangleq f(x^k,y^k)$, $F_{\star,i}^k\triangleq f(x^k,y_\star^k)$, $F_\star^k\triangleq f(x^k,y_\star^k)$, $G_i^k\triangleq g_i(x^k,y^k)$, $G^k\triangleq g(x^k,y^k)$, $G_{\star,i}^k\triangleq g_i(x^k,y_\star^k)$, $G_\star^k\triangleq g(x^k,y_\star^k)$;
    \item $\cX_+^k\triangleq \EE{x^{k+1}-x^k}$, $\cY_+^k\triangleq \EE{y^{k+1}-y^k}$, $\cZ_+^k\triangleq \EE{z^{k+1}-z^k}$, $\cY^k\triangleq \EE{y^k-y_\star^k}$, $\cZ^k\triangleq \EE{z^k-z_\star^k}$;
    \item $D_x^k\triangleq\frac{1}{n}\sum_{i=1}^nD_{x,i}^k$, $D_y^k\triangleq\frac{1}{n}\sum_{i=1}^nD_{y,i}^k$, $D_z^k\triangleq\frac{1}{n}\sum_{i=1}^nD_{z,i}^k$, $\hat{D}_x^k\triangleq\frac{1}{n}\sum_{i=1}^n\cC_i^u(D_{x,i}^k)$;
    \item $\hat{D}_{y,i}^k$ and $\hat{D}_{z,i}^k$ denote compressed local update estimations, which depend on different lower-level compression mechanisms:
    \begin{align*}
        \hat{D}_{y,i}^k\triangleq \begin{cases}
          \cC_i^\ell(D_{y,i}^k), &\mbox{in Alg.~\ref{alg:C-SOBA} and \ref{alg:CM-SOBA-MSC}},\\
            m_{y,i}^k+\cC_i^\ell (D_{y,i}^k-m_{y,i}^k), &\mbox{in Alg.~\ref{alg:EF-SOBA} and \ref{alg:EF-SOBA-MSC}},
        \end{cases}\quad\hat{D}_{z,i}^k\triangleq \begin{cases}
          \cC_i^\ell (D_{z,i}^k), &\mbox{in Alg.~\ref{alg:C-SOBA}  and \ref{alg:CM-SOBA-MSC}},\\
            m_{z,i}^k+\cC_i^\ell (D_{z,i}^k-m_{z,i}^k), &\mbox{in Alg.~\ref{alg:EF-SOBA} and \ref{alg:EF-SOBA-MSC}}.
        \end{cases}
    \end{align*}
\end{itemize}

The following lemmas are frequently used in our analysis.

\begin{lemma}[\cite{chen2023optimal}, Lemma B.1]\label{lm:desce}
    Under Assumptions \ref{asp:conti}, \ref{asp:stron-conve}, if $\beta<\frac{2}{L_g+\mu_g}$, the following inequality holds:
    \begin{align*}
        \|y^k-\beta \nabla_yG^k-y_\star^k\|_2\le(1-\beta\mu_g)\|y^k-y_\star^k\|_2.
    \end{align*}
\end{lemma}

\begin{lemma}[\cite{chen2023optimal}, Lemma B.2]\label{lm:const}
Under Assumptions \ref{asp:conti}, \ref{asp:stron-conve}, there exist positive constants $L_{\nabla\Phi}$, $L_{y^\star}$, $L_{z^\star}$, such that $\nabla\Phi(x)$, $y^\star(x)$, $z^\star(x)$ are $L_{\nabla\Phi}$, $L_{y^\star}$, $L_{z^\star}$-Lipschitz continuous, respectively. Moreover, we have $\|z^\star(x)\|_2\le\frac{C_f}{\mu_g}$ for all $x\in\R^{d_x}$.
\end{lemma}

\textbf{Notation. }For convenience, we define the following constants:
\begin{align*}
    &L_1^2\triangleq L_f^2+L_{g_{yy}}^2\rho^2,\quad L_2^2\triangleq L_f^2+L_{g_{xy}}^2\rho^2,\quad L_3^2\triangleq L_g^2(1+3L_{y^\star}^2),\quad L_4^2\triangleq L_1^2(1+3L_{y^\star}^2)+3L_g^2L_{z^\star}^2,\\
    &L_5^2\triangleq L_2^2(1+3L_{y^\star}^2)+3L_g^2L_{z^\star}^2,\quad \sigma_1^2\triangleq \sigma^2(1+\rho^2),\quad \tilde{\sigma}^2\triangleq \sigma^2/R,\quad \tilde{\sigma}_1^2\triangleq \sigma_1^2/R,\quad \omega_1\triangleq 1+6\omega_\ell (1+\omega_\ell ),\\
    &\omega_2\triangleq 1+36\omega_u(1+\omega_u),\quad \tilde{\omega}_\ell \triangleq \omega_\ell \left(\omega_\ell /(1+\omega_\ell )\right)^R,\quad \tilde{\omega}_u\triangleq \omega_u\left(\omega_u/(1+\omega_u)\right)^R.
\end{align*}

\begin{lemma}[Variance Bounds]\label{lm:varia}
    Under \cref{asp:stoch}, we have the following variance bounds for Alg.~\ref{alg:C-SOBA} and \ref{alg:EF-SOBA}:
    \begin{align*}
        &\Var{D_{y,i}^k\mid\cF^k}\le\sigma^2,\quad \Var{D_y^k\mid\cF^k}\le\sigma^2/n;\\
        &\Var{D_{x,i}^k\mid\cF^k}\le\sigma_1^2,\quad \Var{D_x^k\mid\cF^k}\le\sigma_1^2/n;\\
        &\Var{D_{z,i}^k\mid\cF^k}\le\sigma_1^2,\quad \Var{D_z^k\mid\cF^k}\le\sigma_1^2/n.
    \end{align*}
\end{lemma}
\begin{proof}
    Note that $\|z^k\|_2\le\rho$, \cref{lm:varia} is a direct consequence of \cref{asp:stoch} and the definition of $\sigma_1$.
\end{proof}

\begin{lemma}[Gradient Bias]\label{lm:gradi}
    Under Assumptions \ref{asp:conti}, \ref{asp:stron-conve}, \ref{asp:stoch}, if $\rho\ge C_f/\mu_g$, the following inequality holds for C-SOBA, CM-SOBA (Alg.~\ref{alg:C-SOBA}) and EF-SOBA (Alg.~\ref{alg:EF-SOBA}):
    \begin{align}
        \sum_{k=0}^{K-1}\EE{\E_k(D_x^k)-\nabla\Phi(x^k)}\le&3L_2^2\sum_{k=0}^{K}\cY^k+3L_g^2\sum_{k=0}^{K}\cZ^k.\label{eq:lm-gradi-1}
    \end{align}
\end{lemma}
\begin{proof}
By \cref{asp:stoch}, we have $\E_k(D_x^k)=\nabla_{xy}^2G^kz^k+\nabla_xF^k$. Since $\nabla\Phi(x^k)=\nabla_{xy}^2G_\star^kz_\star^k+\nabla_xF_\star^k$, we have
\begin{align}
    &\EE{\E_k(D_x^k)-\nabla\Phi(x^k)}\nonumber\\
    =&\EE{(\nabla_xF^k-\nabla_xF_\star^k)+(\nabla_{xy}^2G^k-\nabla_{xy}^2G_\star^k)z_\star^k+\nabla_{xy}^2G^k(z^k-z_\star^k)}\nonumber\\
    \le&3\EE{\nabla_xF^k-\nabla_xF_\star^k}+3\EE{(\nabla_{xy}^2G^k-\nabla_{xy}^2G_\star^k)z_\star^k}+3\EE{\nabla_{xy}^2G^k(z^k-z_\star^k)}\nonumber\\
    \le&3L_f^2\EE{y^k-y_\star^k}+3L_{g_{xy}}^2\rho^2\EE{y^k-y_\star^k}+3L_g^2\EE{z^k-z_\star^k},\label{eq:pflm-gradi-1}
\end{align}
where the first inequality uses Cauchy-Schwarz inequality and the second inequality uses \cref{asp:conti}, \cref{lm:const} and $\rho\ge C_f/\mu_g$.
Summing \eqref{eq:pflm-gradi-1} from $k=0$ to $K-1$ we achieve \eqref{eq:lm-gradi-1}. 
\end{proof}

\begin{lemma}[Local Vanilla Compression Error in Lower Level]\label{lm:LVCEL}
    Under Assumptions \ref{asp:conti}, \ref{asp:stron-conve}, \ref{asp:stoch}, \ref{asp:unbia}, and \ref{asp:parti}(or \ref{asp:bound-heter}), the following inequality holds for C-SOBA and CM-SOBA (Alg.~\ref{alg:C-SOBA}):
    \begin{align}
        \sum_{k=0}^{K-1}\frac{1}{n}\sum_{i=1}^n\EE{\hat{D}_{y,i}^k-D_{y,i}^k}\le&2\omega_\ell  L_g^2\sum_{k=0}^K\cY^k+K\omega_\ell \sigma^2+2K\omega_\ell b_g^2,\label{eq:lm-LVCEL-1}\\
        \sum_{k=0}^{k-1}\frac{1}{n}\sum_{i=1}^n\EE{\hat{D}_{z,i}^k-D_{z,i}^k}\le&6\omega_\ell  L_1^2\sum_{k=0}^K\cY^k+6\omega_\ell  L_g^2\sum_{k=0}^{K}\cZ^k+K\omega_\ell \sigma_1^2+4K\omega_\ell b_f^2+\frac{4K\omega_\ell C_f^2}{\mu_g^2}b_g^2.\label{eq:lm-LVCEL-2}
    \end{align}
\end{lemma}
\begin{proof}
By definition, we have $\hat{D}_{y,i}^k=\cC_i^\ell (D_{y,i}^k)$, thus
    \begin{align}
        \EE{\hat{D}_{y,i}^k-D_{y,i}^k}\le&\omega_\ell \EE{D_{y,i}^k}\le\omega_\ell \EE{\nabla_yG_i^k}+\omega_\ell \sigma^2\nonumber\\
        \le&2\omega_\ell \EE{\nabla_yG_{\star,i}^k}+2\omega_\ell L_g^2\EE{y^k-y_\star^k}+\omega_\ell \sigma^2,\label{eq:pflm-LVCEL-1}
    \end{align}
    where the first inequality uses \cref{asp:unbia}, the second inequality uses \cref{lm:varia}, the third inequality uses $\nabla_yG_i^k=\nabla_yG_{\star,i}^k+(\nabla_yG_i^k-\nabla_yG_{\star,i}^k)$, Cauchy-Schwarz inequality and \cref{asp:conti}.
    Averaging \eqref{eq:pflm-LVCEL-1} from $i=1$ to $n$, we obtain
    \begin{align}
        \frac{1}{n}\sum_{i=1}^n\EE{\hat{D}_{y,i}^k-D_{y,i}^k}\le&\frac{2\omega_\ell }{n}\sum_{i=1}^n\EE{\nabla_yG_{\star,i}^k-\nabla_yG_\star^k}+2\omega_\ell L_g^2\EE{y^k-y_\star^k}+\omega_\ell \sigma^2\nonumber\\
        \le&2\omega_\ell b_g^2+2\omega_\ell L_g^2\EE{y^k-y_\star^k}+\omega_\ell \sigma^2,\label{eq:pflm-LVCEL-2}
    \end{align}
    where the first inequality uses $\nabla_yG_\star^k=0$, and the second inequality uses \cref{asp:parti} (or \cref{asp:bound-heter}). Summing \eqref{eq:pflm-LVCEL-2} from $k=0$ to $K-1$ we obtain \eqref{eq:lm-LVCEL-1}. Similarly, we have
    \begin{align}
        &\EE{\hat{D}_{z,i}^k-D_{z,i}^k}\nonumber\\
        \le&\omega_\ell \EE{\nabla_{yy}^2G_i^kz^k+\nabla_yF_i^k}+\omega_\ell \sigma_1^2\nonumber\\
        \le& 6\omega_\ell \EE{(\nabla_{yy}^2G_i^k-\nabla_{yy}^2G_{\star,i}^k)z^k}+6\omega_\ell \EE{\nabla_{yy}^2G_{\star,i}^k(z^k-z_\star^k)}+6\omega_\ell \EE{\nabla_yF_i^k-\nabla_yF_{\star,i}^k}\nonumber\\
        &+4\omega_\ell \EE{(\nabla_{yy}^2G_{\star,i}^k-\nabla_{yy}^2G_{\star}^k)z_\star^k}+4\omega_\ell \EE{\nabla_yF_{\star,i}^k-\nabla_yF_\star^k}+\omega_\ell \sigma_1^2\nonumber\\
        \le&6\omega_\ell \left(L_f^2+L_{g_{yy}}^2\rho^2\right)\EE{y^k-y_\star^k}+6\omega_\ell L_g^2\EE{z^k-z_\star^k}+\frac{4\omega_\ell C_f^2}{\mu_g^2}\EE{\nabla_{yy}^2G_{\star,i}^k-\nabla_{yy}^2G_\star^k}\nonumber\\
        &+4\omega_\ell \EE{\nabla_yF_{\star,i}^k-\nabla_yF_\star^k}+\omega_\ell \sigma_1^2,\label{eq:pflm-LVCEL-3}
    \end{align}
    where the first inequality uses  \cref{asp:unbia} and \cref{lm:varia}, the second inequality uses $\nabla_{yy}^2G_\star^kz_\star^k+\nabla_yF_\star^k=0$ and Cauchy-Schwarz inequality, the third inequality uses \cref{asp:conti} and \cref{lm:const}. Averaging \eqref{eq:pflm-LVCEL-3} from $i=1$ to $n$, summing from $k=0$ to $K-1$, we achieve \eqref{eq:lm-LVCEL-2}.
\end{proof}

\begin{lemma}[Global Vanilla Compression Error in Lower Level]\label{lm:GVCEL}
    Under Assumptions \ref{asp:conti}, \ref{asp:stron-conve}, \ref{asp:stoch}, \ref{asp:unbia} and \ref{asp:parti}(or \ref{asp:bound-heter}), the following inequalities hold for C-SOBA and CM-SOBA (Alg.~\ref{alg:C-SOBA}):
    \begin{align}
        \sum_{k=0}^{K-1}\EE{\nabla_yG^k-\hat{D}_y^k}\le&\frac{2\omega_\ell  L_g^2}{n}\sum_{k=0}^K\cY^k+\frac{(1+\omega_\ell )K\sigma^2}{n}+\frac{2\omega_\ell Kb_g^2}{n},\label{eq:lm-GVCEL-1}\\
        \sum_{k=0}^{K-1}\EE{\E_k(D_z^k)-\hat{D}_z^k}\le&\frac{6\omega_\ell  L_1^2}{n}\sum_{k=0}^K\cY^k+\frac{6\omega_\ell  L_g^2}{n}\sum_{k=0}^K\cZ^k+\frac{(1+\omega_\ell )K\sigma_1^2}{n}+\frac{4\omega_\ell Kb_f^2}{n}+\frac{4\omega_\ell C_f^2Kb_g^2}{n\mu_g^2}.\label{eq:lm-GVCEL-2}
    \end{align}
\end{lemma}
\begin{proof}
By \cref{asp:stoch}, we have
\begin{align*}
    \EE{\nabla_yG^k-\hat{D}_y^k}=&\EE{(\hat{D}_y^k-D_y^k)+(D_y^k-\nabla_yG^k)}=\EE{\hat{D}_y^k-D_y^k}+\EE{D_y^k-\nabla_yG^k}.
\end{align*}
Thus, by \cref{asp:unbia} and conditional independence of the compressors, we have
    \begin{align}
        \EE{\nabla_yG^k-\hat{D}_y^k}=&\EE{\frac{1}{n}\sum_{i=1}^n(\hat{D}_{y,i}^k-D_{y,i}^k)}+\EE{D_y^k-\nabla_yG^k}\nonumber\\
        =&\frac{1}{n^2}\sum_{i=1}^n\EE{\hat{D}_{y,i}^k-D_{y,i}^k}+\EE{D_y^k-\nabla_yG^k}\nonumber\\
        \le&\frac{1}{n^2}\sum_{i=1}^n\EE{\hat{D}_{y,i}^k-D_{y,i}^k}+\frac{1}{n}\sigma^2,\label{eq:pflm-GVCEL-1}
    \end{align}
    where the inequality uses \cref{lm:varia}.
    Summing \eqref{eq:pflm-GVCEL-1} from $k=0$ to $K-1$ and applying \eqref{eq:lm-LVCEL-1}, we obtain \eqref{eq:lm-GVCEL-1}. Similarly, we have
    \begin{align}
        \EE{\E_k(D_z^k)-\hat{D}_z^k}\le\frac{1}{n^2}\sum_{i=1}^n\EE{\hat{D}_{z,i}^k-D_{z,i}^k}+\frac{1}{n}\sigma_1^2.\label{eq:pfVC2-2}
    \end{align}
    Summing \eqref{eq:pfVC2-2} from $k=0$ to $K-1$ and applying \eqref{eq:lm-LVCEL-2}, we obtain \eqref{eq:lm-GVCEL-2}.
\end{proof}

\subsection{Proof of \cref{thm:C-SOBA}}\label{app:C-SOBA}
Before proving \cref{thm:C-SOBA}, we need a few additional lemmas.

\begin{lemma}[\cite{dagreou2022framework}, Lemma C.1 and C.2]\label{lm:smoot}
    Under Assumptions \ref{asp:conti}, \ref{asp:stron-conve}, \ref{asp:stron-conti}, there exists positive constants $L_{yx}$ and $L_{zx}$, such that $y^\star(x)$ and $z^\star(x)$ are $L_{yx}$ and $L_{zx}$-smooth, respectively.
\end{lemma}

\begin{lemma}[Bounded Second Moment of Vanilla Compressed Update Directions]\label{lm:BSMVC}
    Under Assumptions \ref{asp:conti}, \ref{asp:stron-conve}, \ref{asp:stoch}, \ref{asp:unbia} and \ref{asp:bound-heter}, the following inequalities hold for C-SOBA (Alg.~\ref{alg:C-SOBA}):
    \begin{align}
        \sum_{k=0}^{K-1}\EE{\hat{D}_x^k}\le&\left(1+\frac{\omega_u}{n}\right)\sum_{k=0}^{K-1}\EE{\E_k(D_x^k)}+\frac{(1+\omega_u)K\sigma_1^2}{n}+\frac{2\omega_uKb_f^2}{n}+\frac{2\rho^2\omega_uKb_g^2}{n},\label{eq:lm-BSMVC-1}\\
        \sum_{k=0}^{K-1}\EE{\hat{D}_y^k}\le&\left(1+\frac{2\omega_\ell }{n}\right)L_g^2\sum_{k=0}^K\cY^k+\frac{(1+\omega_\ell )K\sigma^2}{n}+\frac{2\omega_\ell Kb_g^2}{n},\label{eq:lm-BSMVC-2}\\
        \sum_{k=0}^{K-1}\EE{\hat{D}_z^k}\le&\left(3+\frac{6\omega_\ell }{n}\right)L_1^2\sum_{k=0}^{K}\cY^k+\left(3+\frac{6\omega_\ell }{n}\right)L_g^2\sum_{k=0}^{K}\cZ^k+\frac{(1+\omega_\ell )K\sigma_1^2}{n}+\frac{4\omega_1Kb_f^2}{n}+\frac{4C_f^2\omega_\ell Kb_g^2}{n\mu_g^2}.\label{eq:lm-BSMVC-3}
    \end{align}
\end{lemma}
\begin{proof}
    Let $\cF_x^k$ denote the $\sigma$-field of $\{D_{x,i}^k\mid 1\le i\le n\}$ and random variables already generated at the beginning of iteration $k$. 
    By \cref{asp:stoch} and \cref{lm:varia}, we have
    \begin{align}
        \EE{D_x^k}=\EE{\E_k(D_x^k)}+\Var{D_x^k\mid\cF^k}\le\EE{\E_k(D_x^k)}+\frac{\sigma_1^2}{n}.\label{eq:pflm-BSMVC-1}
    \end{align}
    Similarly, 
    \begin{align}
        \EE{D_{x,i}^k}\le&\EE{\E_k(D_{x,i}^k)}+\sigma_1^2.\label{eq:pflm-BSMVC-2}
    \end{align}
    Averaging \eqref{eq:pflm-BSMVC-2} from $i=1$ to $n$, we obtain
    \begin{align}
        \frac{1}{n}\sum_{i=1}^n\EE{D_{x,i}^k}\le&\frac{1}{n}\sum_{i=1}^n\EE{\E_k(D_{x,i}^k)}+\sigma_1^2\nonumber\\
        =&\EE{\E_k(D_x^k)}+\frac{1}{n}\sum_{i=1}^n\EE{\E_k(D_{x,i}^k)-\E_k(D_x^k)}+\sigma_1^2\nonumber\\
        \le&\EE{\E_k(D_x^k)}+\frac{2}{n}\sum_{i=1}^n\EE{(\nabla_{xy}^2G_i^k-\nabla_{xy}^2G^k)z^k}+\frac{2}{n}\sum_{i=1}^n\EE{\nabla_xF_i^k-\nabla_xF^k}+\sigma_1^2\nonumber\\
        \le&\EE{\E_k(D_x^k)}+2\rho^2b_g^2+2b_f^2+\sigma_1^2,\label{eq:pflm-BSMVC-3}
    \end{align}
    where the second inequality uses Cauchy-Schwarz inequality, and the third inequality uses \cref{asp:bound-heter}. For $\hat{D}_x^k$, we have
    \begin{align}
        \EE{\hat{D}_x^k}=&\E\left[\E\left[\left.\|\hat{D}_x^k\|^2\right| \cF_x^k\right]\right]\nonumber\\
        =&\EE{D_x^k}+\E\left[\E\left[\left.\left\|\hat{D}_x^k-D_x^k\right\|^2\right| \cF_x^k\right]\right]\nonumber\\
        =&\EE{D_x^k}+\frac{1}{n^2}\sum_{i=1}^n\E\left[\E\left[\left.\left\|\hat{D}_{x,i}^k-D_{x,i}^k\right\|^2\right|\cF_x^k\right]\right]\nonumber\\
        \le&\EE{D_x^k}+\frac{\omega_u}{n^2}\sum_{i=1}^n\EE{D_{x,i}^k},\label{eq:pflm-BSMVC-4}
    \end{align}
    where the inequality uses \cref{asp:unbia}. Applying \eqref{eq:pflm-BSMVC-1}\eqref{eq:pflm-BSMVC-3} to \eqref{eq:pflm-BSMVC-4} and summing from $k=0$ to $K-1$, we obtain \eqref{eq:lm-BSMVC-1}. By \cref{asp:conti}, we have
    \begin{align*}
        \EE{\E_k(D_y^k)}=\EE{\nabla_yG^k-\nabla_yG_\star^k}\le L_g^2\EE{y^k-y_\star^k}.
    \end{align*}
    Consequently,
    \begin{align}
        \EE{\hat{D}_y^k}=&\EE{\E_k(D_y^k)}+\EE{\hat{D}_y^k-\E_k(D_y^k)}\nonumber\\
        \le&L_g^2\EE{y^k-y_\star^k}+\EE{\hat{D}_y^k-\nabla_yG^k}.\label{eq:pflm-BSMVC-5}
    \end{align}
    Summing \eqref{eq:pflm-BSMVC-5} from $k=0$ to $K-1$ and applying \cref{lm:GVCEL}, we obtain \eqref{eq:lm-BSMVC-2}. Similarly, we have
    \begin{align*}
        \EE{\E_k(D_z^k)}=&\EE{(\nabla_{yy}^2G^k-\nabla_{yy}^2G_\star^k)z^k+\nabla_{yy}^2G_\star^k(z^k-z_\star^k)+(\nabla_yF^k-\nabla_yF_\star^k)}\nonumber\\
        \le&3(L_f^2+L_{g_{yy}}^2\rho^2)\EE{y^k-y_\star^k}+3L_g^2\EE{z^k-z_\star^k}.
    \end{align*}
    Thus, 
    \begin{align}
        \EE{\hat{D}_z^k}=&\EE{\E_k(D_z^k)}+\EE{\hat{D}_z^k-\E_k(D_z^k)}\nonumber\\
        \le&3L_1^2\EE{y^k-y_\star^k}+3L_g^2\EE{z^k-z_\star^k}.\label{eq:pflm-BSMVC-6}
    \end{align}
    Summing \eqref{eq:pflm-BSMVC-6} from $k=0$ to $K-1$ and applying \cref{lm:GVCEL}, we obtain \eqref{eq:lm-BSMVC-3}.    
\end{proof}

\begin{lemma}[Lower Level Convergence]\label{lm:LLC-VCC}
    Under Assumptions \ref{asp:conti}, \ref{asp:stron-conve}, \ref{asp:stoch}, \ref{asp:unbia}, \ref{asp:bound-heter}, \ref{asp:bound-updat}, \ref{asp:stron-conti} and assume $\beta<\min\left\{\frac{2}{\mu_g+L_g},\frac{\mu_g}{8(1+4\omega_\ell /n) L_g^2}\right\}$, $\gamma\le\min\left\{\frac{1}{L_g},\frac{\mu_g}{12(1+4\omega_\ell /n) L_g^2}\right\}$, $\rho\ge\frac{C_f}{\mu_g}$, and 
    \[\alpha\le\sqrt{\frac{\mu_gL_{y^\star}^2\beta}{L_{yx}^2\left(B_x^2(1+\omega_u/n)+(1+\omega_u)\sigma_1^2/n+2\omega_ub_f^2/n+2\rho^2\omega_ub_g^2/n\right)}},\]
    the following inequalities hold for C-SOBA (Alg.~\ref{alg:C-SOBA}):
    \begin{align}
        \sum_{k=0}^K\cY^k\le&\frac{8\cY^0}{\mu_g\beta}+\frac{16(1+\omega_\ell )K\beta\sigma^2}{\mu_g n}+\frac{24L_{y^\star}^2(1+\omega_u)K\alpha^2\sigma_1^2}{\mu_g n\beta}\nonumber\\
        &+\left(\frac{24L_{y^\star}^2\alpha^2(1+\omega_u/n)}{\mu_g\beta}+\frac{16L_{y^\star}^2\alpha^2}{\mu_g^2\beta^2}\right)\sum_{k=0}^{K-1}\EE{\E_k(D_x^k)}\nonumber\\
        &+\frac{48L_{y^\star}^2\omega_uK\alpha^2b_f^2}{\mu_gn\beta}+\left(\frac{32\omega_\ell K\beta}{n\mu_g}+\frac{48L_{y^\star}^2\rho^2\omega_uK\alpha^2}{n\mu_g\beta}\right)b_g^2,\label{eq:lm-LLC-1}\\
        \sum_{k=0}^K\cZ^k\le&\frac{4\cZ^0}{\mu_g\gamma}+\frac{136L_1^2\cY^0}{\mu_g^3\beta}+\frac{272(1+\omega_\ell )L_1^2K\beta\sigma^2}{\mu_g^3n}+\nonumber\\
        &+\left(\frac{408L_1^2L_{y^\star}^2(1+\omega_u)\alpha^2}{\mu_g^3n\beta}+\frac{8(1+\omega_\ell )\gamma}{\mu_gn}+\frac{8(L_{z^\star}^2+L_{zx}\rho)(1+\omega_u)\alpha^2}{\mu_gn\gamma}\right)\cdot K\sigma_1^2\nonumber\\
        &+\left(\frac{408L_1^2L_{y^\star}^2(1+\omega_u/n)\alpha^2}{\mu_g^3\beta}+\frac{272L_1^2L_{y^\star}^2\alpha^2}{\mu_g^4\beta^2}+\frac{8(L_{z^\star}^2+L_{zx}\rho)(1+\omega_u/n)\alpha^2}{\mu_g\gamma}\right.\nonumber\\
        &+\left.\frac{8L_{z^\star}^2\alpha^2}{\mu_g^2\gamma^2}\right)\sum_{k=0}^{K-1}\EE{\E_k(D_x^k)}+\left(\frac{816L_1^2L_{y^\star}^2\omega_uK\alpha^2}{\mu_g^3n\beta}+\frac{32\omega_\ell K\gamma}{n\mu_g}+\frac{16(L_{z^\star}^2+L_{zx}\rho)\omega_uK\alpha^2}{n\mu_g\gamma}\right)b_f^2\nonumber\\
        &+\left(\frac{544L_1^2\omega_\ell K\beta}{n\mu_g^3}+\frac{816L_1^2L_{y^\star}^2\rho^2\omega_uK\alpha^2}{n\mu_g^3\beta}+\frac{16\rho^2\omega_\ell K\gamma}{n\mu_g}+\frac{8(L_{z^\star}^2+L_{zx}\rho)\rho^2\omega_uK\alpha^2}{n\mu_g\gamma}\right)b_g^2.\label{eq:lm-LLC-2}
    \end{align}
\end{lemma}
\begin{proof}
We first separate $\cY^{k+1}$ into five parts:
\begin{align}
    \cY^{k+1}=&\EE{y^{k+1}-y_\star^k}+\EE{y_\star^{k+1}-y_\star^k}-2\E\left[\langle y^{k+1}-y_\star^k, y_\star^{k+1}-y_\star^k\rangle\right]\nonumber\\
    =&\EE{y^{k+1}-y_\star^k}+\EE{y_\star^{k+1}-y_\star^k}-2\E\left[\langle y^k-y_\star^k,y_\star^{k+1}-y_\star^k\rangle\right]+2\beta\E\left[\left\langle\hat{D}_y^k,y_\star^{k+1}-y_\star^k\right\rangle\right]\nonumber\\
    =&\EE{y^{k+1}-y_\star^k}+\EE{y_\star^{k+1}-y_\star^k}-2\E\left[\left\langle y^k-y_\star^k,\nabla y^\star(x^k)(x^{k+1}-x^k)\right\rangle\right]\nonumber\\
    &-2\E\left[\left\langle y^k-y_\star^k,y_\star^{k+1}-y_\star^k-\nabla y^\star(x^k)(x^{k+1}-x^k)\right\rangle\right]+2\beta\E\left[\left\langle \hat{D}_y^k,y_\star^{k+1}-y_\star^k\right\rangle\right],\label{eq:pflm-LLC-VCC-1}
\end{align}
where the existence of $\nabla y^\star(x^k)$ is guaranteed by \cref{lm:smoot}.

    For the first part, by Assumptions \ref{asp:stoch} and  \ref{asp:unbia},  $\hat{D}_y^k$ is an unbiased estimator of $\nabla_yG^k$, thus
    \begin{align}
        \EE{y^{k+1}-y_\star^k}=&\EE{y^k-\beta\hat{D}_y^k-y_\star^k}=\EE{(y^k-y_\star^k-\beta\nabla_yG^k)-\beta(\hat{D}_y^k-\nabla_yG^k)}\nonumber\\
        =&\EE{y^k-y_\star^k-\beta\nabla_yG^k}+\EE{\beta(\hat{D}_y^k-\nabla_yG^k)}\nonumber\\
        \le&(1-\beta\mu_g)^2\cY^k+\beta^2\EE{\hat{D}_y^k-\nabla_yG^k}\nonumber\\
        \le&(1-\beta\mu_g)\cY^k+\beta^2\EE{\hat{D}_y^k-\nabla_yG^k},\label{eq:pflm-LLC-VCC-2}
    \end{align}
    where the first inequality uses \cref{lm:desce}. For the second part, \cref{lm:const} implies
    \begin{align}
        \EE{y_\star^{k+1}-y_\star^k}\le&L_{y^\star}^2\cX_+^k=L_{y^\star}^2\alpha^2\EE{\hat{D}_x^{k}}.\label{eq:pflm-LLC-VCC-3}
    \end{align}
    For the third part, we have
    \begin{align}
        -2\E\left[\left\langle y^k-y_\star^k,\nabla y^\star(x^k)(x^{k+1}-x^k)\right\rangle\right]=&2\alpha\E\left[\left\langle y^k-y_\star^k,\nabla y^\star(x^k)\E_k(D_x^k)\right\rangle\right]\nonumber\\
        \le&\frac{\beta\mu_g}{2}\cY^k+\frac{2\alpha^2}{\beta\mu_g}L_{y^\star}^2\EE{\E_k(D_x^k)},\label{eq:pflm-LLC-VCC-4}
    \end{align}
    where the equality uses the unbiasedness of $\hat{D}_x^k$ and the inequality uses Young's inequality and \cref{lm:const}.
    For the fourth part, we have
    \begin{align}
        &-2\E\left[\left\langle y^k-y_\star^k,y_\star^{k+1}-y_\star^k-\nabla y^\star(x^k)(x^{k+1}-x^k)\right\rangle\right]\le L_{yx}\E\left[\|y^k-y_\star^k\|\|x^{k+1}-x^k\|^2\right]\nonumber\\
        \le&\frac{L_{yx}^2}{4L_{y^\star}^2}\E\left[\E_k\left[\|y^k-y_\star^k\|^2\|x^{k+1}-x^k\|^2\right]\right]+L_{y^\star}^2\cX_+^k\nonumber\\
        \le&\frac{L_{yx}^2\alpha^2}{4L_{y^\star}^2}\E\left[\|y^k-y_\star^k\|^2\E_k\left[\left\|\hat{D}_x^k\right\|^2\right]\right]+L_{y^\star}^2\alpha^2\EE{\hat{D}_x^k}\nonumber\\
        \le&\frac{L_{yx}^2\left(B_x^2(1+\omega_u/n)+(1+\omega_u)\sigma_1^2/n+2\omega_ub_f^2/n+2\rho^2\omega_ub_g^2/n\right)\alpha^2}{4L_{y^\star}^2}\cY^k+L_{y^\star}^2\alpha^2\EE{\hat{D}_x^k}\nonumber\\
        \le&\frac{\beta\mu_g}{4}\cY^k+L_{y^\star}^2\alpha^2\EE{\hat{D}_x^k},\label{eq:pflm-LLC-VCC-5}
    \end{align}
    where the first inequality uses Taylor's expansion, Cauchy's inequality and \cref{lm:smoot}, the second inequality uses Young's inequality, the third inequality uses the definition of $x^{k+1}$, the fourth inequality is due to 
    \begin{align*}
        \E_k\left[\left\|\hat{D}_x^k\right\|^2\right]=&\left\|\E_k(D_x^k)\right\|^2+\E_k\left[\left\|D_x^k-\E_k(D_x^k)\right\|^2\right]+\E_k\left[\E\left[\left.\left\|\hat{D}_x^k-D_x^k\right\|^2\right|\mathcal{F}_x^k\right]\right]\\
        \le&B_x^2+\frac{\sigma_1^2}{n}+\frac{\omega_u}{n^2}\sum_{i=1}^n\E_k\left[\left\|D_{x,i}^k\right\|^2\right]\\
        \le&B_x^2+\frac{\sigma_1^2}{n}+\frac{\omega_u}{n}\left(B_x^2+\sigma_1^2+2b_f^2+2\rho^2b_g^2\right),
    \end{align*}
    and the last inequality is due to $\alpha\le\sqrt{\frac{\mu_gL_{y^\star}^2\beta}{L_{yx}^2\left(B_x^2(1+\omega_u/n)+(1+\omega_u)\sigma_1^2/n+2\omega_ub_f^2/n+2\rho^2\omega_ub_g^2/n\right)}}$. 
    For the fifth part, by Young's inequality and \cref{lm:const} we have
    \begin{align}
        2\beta\E\left[\left\langle\hat{D}_y^k,y_{\star}^{k+1}-y_\star^k\right\rangle\right]\le&\beta^2\EE{\hat{D}_y^k}+L_{y^\star}^2\alpha^2\EE{\hat{D}_x^k}.\label{eq:pflm-LLC-VCC-6}
    \end{align}
    Summing \eqref{eq:pflm-LLC-VCC-1}\eqref{eq:pflm-LLC-VCC-2}\eqref{eq:pflm-LLC-VCC-3}\eqref{eq:pflm-LLC-VCC-4}\eqref{eq:pflm-LLC-VCC-5}\eqref{eq:pflm-LLC-VCC-6} from $k=0$ to $K-1$ and applying \cref{lm:GVCEL} we obtain
    \begin{align}
        &\sum_{k=0}^{K}\cY^k\nonumber\\
        \le&\frac{4\cY^0}{\beta\mu_g}+\frac{4(1+\omega_\ell )K\beta\sigma^2}{n\mu_g}+\frac{12L_{y^\star}^2\alpha^2}{\beta\mu_g}\sum_{k=0}^{K-1}\EE{\hat{D}_x^k}+\frac{8L_{y^\star}^2\alpha^2}{\beta^2\mu_g^2}\sum_{k=0}^{K-1}\EE{\E_k(D_x^k)}\nonumber\\
        &+\frac{8\omega_\ell  L_g^2\beta}{n\mu_g}\sum_{k=0}^K\cY^k+\frac{4\beta}{\mu_g}\sum_{k=0}^{K-1}\EE{\hat{D}_y^k}+\frac{8\omega_\ell K\beta b_g^2}{n\mu_g}\nonumber\\
        \le&\frac{4\cY^0}{\beta\mu_g}+\frac{8(1+\omega_\ell )K\beta\sigma^2}{n\mu_g}+\frac{12L_{y^\star}^2(1+\omega_u)K\alpha^2\sigma_1^2}{\mu_g n\beta}+\left(\frac{12L_{y^\star}^2\alpha^2(1+\omega_u/n)}{\beta\mu_g}+\frac{8L_{y^\star}^2\alpha^2}{\beta^2\mu_g^2}\right)\sum_{k=0}^{K-1}\EE{\E_k(D_x^k)}\nonumber\\
        &+\frac{4(1+4\omega_\ell /n)L_g^2\beta}{\mu_g}\sum_{k=0}^{K}\cY^k+\frac{24L_{y^\star}^2\omega_uK\alpha^2b_f^2}{\mu_gn\beta}+\left(\frac{16\omega_\ell K\beta}{n\mu_g}+\frac{24L_{y^\star}^2\rho^2\omega_uK\alpha^2}{n\mu_g\beta}\right)b_g^,\label{eq:pflm-LLC-VCC-7}
    \end{align}
    where the second inequality uses \cref{lm:BSMVC}. Using $\beta\le\frac{\mu_g}{8(1+4\omega_\ell /n) L_g^2}$, \eqref{eq:pflm-LLC-VCC-7} implies \eqref{eq:lm-LLC-1}. Similarly, we separate $\cZ^{k+1}$ into five parts:
    \begin{align}
        \cZ^{k+1}\le&\EE{\tilde{z}^{k+1}-z_\star^{k+1}}\nonumber\\
        =&\EE{\tilde{z}^{k+1}-z_\star^k}+\EE{z_{\star}^{k+1}-z_\star^k}-2\E\left[\left\langle z^k-z_\star^k,\nabla z^\star(x^k)(x^{k+1}-x^k)\right\rangle\right]\nonumber\\
        &-2\E\left[\left\langle z^k-z_\star^k,z_\star^{k+1}-z_\star^k-\nabla z^\star(x^k)(x^{k+1}-x^k)\right\rangle\right]+2\gamma\E\left[\left\langle\hat{D}_z^k,z_\star^{k+1}-z_\star^k\right\rangle\right],\label{eq:pflm-LLC-VCC-8}
    \end{align}
    where the inequality is due to \cref{lm:const} and $\rho\ge C_f/\mu_g$.
    For the first part, we have
\begin{align}
    \EE{\tilde{z}^{k+1}-z_\star^k}=&\EE{z^k-\gamma\E_k(D_z^k)-z_\star^k}+\gamma^2\EE{\E_k(D_z^k)-\hat{D}_z^k}\nonumber\\
    =&\EE{(z^k-z_\star^k)-\gamma\nabla_{yy}^2G^k(z^k-z_\star^k)-\gamma(\nabla_{yy}^2G^kz_\star^k+\nabla_yF^k)}+\gamma^2\EE{\E_k(D_z^k)-\hat{D}_z^k}\nonumber\\
    \le&(1+\gamma\mu_g)(1-\gamma\mu_g)^2\cZ^k+\left(1+\frac{1}{\gamma\mu_g}\right)\gamma^2\cdot2L_1^2\cY^k+\gamma^2\EE{\E_k(D_z^k)-\hat{D}_z^k},\label{eq:pflm-LLC-VCC-9}
\end{align}
where the inequality uses Young's inequality, $\|I-\gamma\nabla_{yy}^2G^k\|_2\le1-\gamma\mu_g$ and \[
\EE{(\nabla_{yy}^2G^k-\nabla_{yy}^2G_\star^k)z_\star^k+(\nabla_yF^k-\nabla_yF_\star^k)}\le2\left(L_f^2+L_{g_{yy}}^2\rho^2\right)\cY^k.\]
For the second part, \cref{lm:const} implies
\begin{align}
    \EE{z_\star^{k+1}-z_\star^k}\le&L_{z^\star}^2\cX_+^k=L_{z^\star}^2\alpha^2\EE{\hat{D}_x^k}.\label{eq:pflm-LLC-VCC-10}
\end{align}
For the third part, applying Young's inequality along with \cref{lm:const} gives 
\begin{align}
    -2\E\left[\left\langle z^k-z_\star^k,\nabla z^\star(x^k)(x^{k+1}-x^k)\right\rangle\right]\le&\frac{\gamma\mu_g}{2}\cZ^k+\frac{2\alpha^2}{\gamma\mu_g}L_{z^\star}^2\EE{\E_k(D_x^k)}.\label{eq:pflm-LLC-VCC-11}
\end{align}
For the fourth part, we have
\begin{align}
    -2\E\left[\left\langle z^k-z_\star^k,z_\star^{k+1}-z_\star^k-\nabla z^\star(x^k)(x^{k+1}-x^k)\right\rangle\right]\le&L_{zx}\E\left[\|z^k-z_\star^k\|\|x^{k+1}-x^k\|^2\right]\le2L_{zx}\rho\alpha^2\EE{\hat{D}_x^k},\label{eq:pflm-LLC-VCC-12}
\end{align}
where the first inequality uses Taylor's expansion, Cauchy-Schwarz inequality and \cref{lm:smoot}, and the second inequality uses \cref{lm:const} and $\rho\ge C_f/\mu_g$. For the fifth part, by Young's inequality and \cref{lm:const} we have
\begin{align}
    2\gamma\E\left[\left\langle\hat{D}_z^k,z_\star^{k+1}-z_\star^k\right\rangle\right]\le&\gamma^2\EE{\hat{D}_z^k}+L_{z^\star}^2\alpha^2\EE{\hat{D}_x^k}.\label{eq:pflm-LLC-VCC-13}
\end{align}
Summing \eqref{eq:pflm-LLC-VCC-8}\eqref{eq:pflm-LLC-VCC-9}\eqref{eq:pflm-LLC-VCC-10}\eqref{eq:pflm-LLC-VCC-11}\eqref{eq:pflm-LLC-VCC-12}\eqref{eq:pflm-LLC-VCC-13} from $K=0$ to $K-1$ and applying \cref{lm:GVCEL} we achieve
\begin{align}
    \sum_{k=0}^K\cZ^k\le&\frac{2\cZ^0}{\mu_g\gamma}+\frac{2(1+\omega_\ell )K\gamma\sigma_1^2}{\mu_gn}+\left(\frac{4L_{z^\star}^2\alpha^2}{\mu_g\gamma}+\frac{4L_{zx}\rho\alpha^2}{\mu_g\gamma}\right)\sum_{k=0}^{K-1}\EE{\hat{D}_x^k}+\frac{4L_{z^\star}^2\alpha^2}{\mu_g^2\gamma^2}\sum_{k=0}^{K-1}\EE{\E_k(D_x^k)}\nonumber\\
    &+\frac{2\gamma}{\mu_g}\EE{\hat{D}_z^k}+\left(\frac{8L_1^2}{\mu_g^2}+\frac{12\omega
_\ell  L_1^2\gamma}{\mu_gn}\right)\sum_{k=0}^K\cY^k+\frac{12\omega_\ell  L_g^2\gamma}{\mu_gn}\sum_{k=0}^K\cZ^k+\frac{8\omega_\ell K\gamma b_f^2}{n\mu_g}+\frac{8C_f^2\omega_\ell K\gamma b_g^2}{n\mu_g^3}\nonumber\\
    \le&\frac{2\cZ^0}{\mu_g\gamma}+\left(\frac{4(1+\omega_\ell )K\gamma\sigma_1^2}{\mu_gn}+\frac{4(L_{z^\star}^2+L_{zx}\rho)(1+\omega_u)K\alpha^2\sigma_1^2}{\mu_gn\gamma}\right)\nonumber\\
    &+\left(\frac{4L_{z^\star}^2\alpha^2}{\mu_g^2\gamma^2}+\frac{4(L_{z^\star}^2+L_{zx}\rho)(1+\omega_u/n)\alpha^2}{\mu_g\gamma}\right)\sum_{k=0}^{K-1}\EE{\E_k(D_x^k)}\nonumber\\
    &+\left(\frac{8L_1^2}{\mu_g^2}+\frac{6(1+4\omega_\ell /n)L_1^2\gamma}{\mu_g}\right)\sum_{k=0}^K\cY^k+\frac{6(1+4\omega_\ell /n)L_g^2\gamma}{\mu_g}\sum_{k=0}^K\cZ^k\nonumber\\
    &+\left(\frac{16\omega_\ell K\gamma}{n\mu_g}+\frac{8(L_{z^\star}^2+L_{zx}\rho)\omega_uK\alpha^2}{n\mu_g\gamma}\right)b_f^2+\left(\frac{16\rho^2\omega_\ell K\gamma}{n\mu_g}+\frac{8(L_{z^\star}^2+L_{zx}\rho)\rho^2\omega_uK\alpha^2}{n\mu_g\gamma}\right)b_g^2,\label{eq:pflm-LLC-VCC-14}
\end{align}
where the second inequality uses \cref{lm:BSMVC} and $\rho\ge C_f/\mu_g$.
Using $\gamma\le\frac{\mu_g}{12(1+4\omega_\ell /n)L_g^2}$, we obtain
\begin{align}
\sum_{k=0}^K\cZ^k\le&\frac{4\cZ^0}{\mu_g\gamma}+\left(\frac{8(1+\omega_\ell )K\gamma\sigma_1^2}{\mu_gn}+\frac{8(L_{z^\star}^2+L_{zx}\rho)(1+\omega_u)K\alpha^2\sigma_1^2}{\mu_gn\gamma}\right)\nonumber\\
&+\left(\frac{8L_{z^\star}^2\alpha^2}{\mu_g^2\gamma^2}+\frac{8(L_{z^\star}^2+L_{zx}\rho)(1+\omega_u/n)\alpha^2}{\mu_g\gamma}\right)\sum_{k=0}^{K-1}\EE{\E_k(D_x^k)}\nonumber\\
&+\frac{17L_1^2}{\mu_g^2}\sum_{k=0}^K\cY^k+\left(\frac{32\omega_\ell K\gamma}{n\mu_g}+\frac{16(L_{z^\star}^2+L_{zx}C_f/\mu_g)\omega_uK\alpha^2}{n\mu_g\gamma}\right)b_f^2\nonumber\\
&+\left(\frac{32\rho^2\omega_\ell K\gamma}{n\mu_g}+\frac{16(L_{z^\star}^2+L_{zx}\rho)\rho^2\omega_uK\alpha^2}{n\mu_g\gamma}\right)b_g^2.\label{eq:pflm-LLC-VCC-15}
\end{align}
Applying \eqref{eq:lm-LLC-1} to \eqref{eq:pflm-LLC-VCC-15} achieves \eqref{eq:lm-LLC-2}.
\end{proof}

Now we are ready to prove \ref{thm:C-SOBA}. We first restate the theorem in a more detailed way.

\begin{theorem}[Convergence of C-SOBA]\label{thm:re-VCC}
    Under the conditions that Assumptions \ref{asp:conti}, \ref{asp:stron-conve}, \ref{asp:stoch}, \ref{asp:unbia}, \ref{asp:bound-heter}, \ref{asp:bound-updat}, \ref{asp:stron-conti} hold and $\beta<\min\left\{\frac{2}{\mu_g+L_g},\frac{\mu_gn}{8(1+4\omega_\ell /n) L_g^2}\right\}$, $\gamma\le\min\left\{\frac{1}{L_g},\frac{\mu_g}{12(1+4\omega_\ell /n) L_g^2}\right\}$, $\rho\ge C_f/\mu_g$,
    \begin{align}
        &\alpha\le\min\left\{\frac{1}{5L_{\nabla\Phi}(1+\omega_u/n)},\sqrt{\frac{\mu_g\beta}{360L_{y^\star}^2(L_2^2+17\kappa_g^2L_1^2)(1+\omega_u/n)}},\sqrt{\frac{\mu_g\gamma}{120L_g^2(L_{z^\star}^2+L_{zx}\rho)(1+\omega_u/n)}},\right.\nonumber\\
        &\left.\sqrt{\frac{\mu_g^2\beta^2}{240L_{y^\star}^2(L_2^2+17\kappa_g^2L_1^2)}},\sqrt{\frac{\mu_g^2\gamma^2}{120L_g^2L_{z^\star}^2}},\sqrt{\frac{\mu_gL_{y^\star}^2\beta}{L_{yx}^2\left(B_x^2(1+\omega_u/n)+(1+\omega_u)\sigma_1^2/n+2\omega_ub_f^2/n+2\rho^2\omega_ub_g^2/n\right)}}\right\},\label{eq:thm-reVCC-alpha}
    \end{align}
    C-SOBA (Alg.~\ref{alg:C-SOBA})  converges as
    \begin{align}
    &\frac{1}{K}\sum_{k=0}^{K-1}\EE{\nabla\Phi(x^k)}\nonumber\\
    \le&\frac{2\Delta_\Phi^0}{K\alpha}+\frac{24(L_2^2+17\kappa_g^2L_1^2)\Delta_y^0}{\mu_gK\beta}+\frac{12L_g^2\Delta_z^0}{\mu_gK\gamma}+\frac{48(L_2^2+17\kappa_g^2L_1^2)(1+\omega_\ell )\beta\sigma^2}{\mu_gn}+\left(\frac{L_{\nabla\Phi}(1+\omega_u)\alpha}{n}+\frac{24L_g^2(1+\omega_\ell )\gamma}{\mu_gn}\right.\nonumber\\
    &+\left.\frac{72L_{y^\star}^2(L_2^2+17\kappa_g^2L_1^2)(1+\omega_u)\alpha^2}{\mu_gn\beta}+\frac{24L_g^2(L_{z^\star}^2+L_{zx}\rho)(1+\omega_u)\alpha^2}{\mu_gn\gamma}\right)\sigma_1^2\nonumber\\
    &+\left(\frac{2L_{\nabla\Phi}\omega_u\alpha}{n}+\frac{96L_g^2\omega_\ell \gamma}{\mu_gn}+\frac{144L_{y^\star}^2(L_2^2+17\kappa_g^2L_1^2)\omega_u\alpha^2}{\mu_gn\beta}+\frac{48L_g^2(L_{z^\star}^2+L_{zx}\rho)\omega_u\alpha^2}{n\mu_g\gamma}\right)b_f^2\nonumber\\
    &+\left(\frac{2L_{\nabla\Phi}\rho^2\omega_u\alpha}{n}+\frac{96(L_2^2+17\kappa_g^2L_1^2)\omega_\ell \beta}{\mu_gn}+\frac{48L_g^2\rho^2\omega_\ell \gamma}{\mu_gn}+\frac{144L_{y^\star}^2\rho^2(L_2^2+16\kappa_g^2L_1^2)\omega_u\alpha^2}{\mu_gn\beta}\right.\nonumber\\
    &+\left.\frac{24L_g^2\rho^2(L_{z^\star}^2+L_{zx}\rho)\omega_u\alpha^2}{\mu_gn\gamma}\right)b_g^2.\label{eq:thm-reVCC-1}
\end{align}
If we further choose parameters as 
    \begin{align*}
        \beta=&\left(\frac{\mu_g+L_g}{2}+\frac{8(1+4\omega_\ell /n) L_g^2}{\mu_g}+\sqrt{\frac{2K\left((1+\omega_\ell )\sigma^2+2\omega_\ell b_g^2\right)}{n\Delta_y^0}}\right)^{-1},\\
        \gamma=&\left(L_g+\frac{12(1+4\omega_\ell /n) L_g^2}{\mu_g}+\sqrt{\frac{2K\left((1+\omega_\ell )\sigma_1^2+4\omega_\ell b_f^2+2(C_f^2/\mu_g^2)\omega_\ell b_g^2\right)}{n\Delta_z^0}}\right)^{-1},\\
        \rho=&\frac{C_f}{\mu_g},\\
        \alpha=&\left(5L_{\nabla\Phi}(1+\omega_u/n)+\sqrt{\frac{360L_{y^\star}^2(L_2^2+17\kappa_g^2L_1^2)(1+\omega_u/n)}{\mu_g\beta}}+\sqrt{\frac{120L_g^2(L_{z^\star}^2+L_{zx}C_f/\mu_g)(1+\omega_u/n)}{\mu_g\gamma}}\right.\nonumber\\
        &+\left.\sqrt{\frac{240L_{y^\star}^2(L_2^2+17\kappa_g^2L_1^2)}{\mu_g^2\beta^2}}+\sqrt{\frac{120L_g^2L_{z^\star}^2}{\mu_g^2\gamma^2}}\right.\nonumber\\
        &+\left. \sqrt{\frac{L_{yx}^2\left(B_x^2(1+\omega_u/n)+(1+\omega_u)\sigma_1^2/n+2\omega_ub_f^2/n+2C_f^2\omega_ub_g^2/(\mu_g^2n)\right)}{\mu_gL_{y^\star}^2\beta}}\right)^{-1},\nonumber\\
    \end{align*}
    C-SOBA (Alg.~\ref{alg:C-SOBA}) converges as order
    \begin{align*}
        \frac{1}{K}\sum_{k=0}^{K-1}\EE{\nabla\Phi(x^k)}=&\mathcal{O}\left(\frac{\sqrt{(1+\omega_\ell +\omega_u)\Delta}\sigma+\sqrt{(\omega_\ell +\omega_u)\Delta}(b_f+b_g)}{\sqrt{nK}}\right.\nonumber\\
        &+\left.\frac{\Delta^{3/4}(\sqrt[4]{1+\omega_\ell }\sqrt{\sigma}+\sqrt[4]{\omega_\ell }\sqrt{b_g})\left(\sqrt{1+\omega_u}\sigma+\sqrt{\omega_u}(b_f+b_g)+\sqrt{n+\omega_u}B_x\right)}{(nK)^{3/4}}\right.\nonumber\\
        &+\left.\frac{\sqrt{(1+\omega_u)(1+\omega_\ell /n)}\Delta\sigma+\sqrt{\omega_u(1+\omega_\ell /n)}(b_f+b_g)}{\sqrt{n}K}\right.\nonumber\\
        &+\left.\frac{\sqrt{(1+\omega_\ell /n)(1+\omega_u/n)}\Delta B_x}{K}+\frac{(1+\omega_\ell /n+\omega_u/n)\Delta}{K}\right),
    \end{align*}
    where $\Delta_\Phi^0\triangleq \Phi(x^0)$,   $\Delta_y^0\triangleq \|y^0-y_\star^0\|_2^2$, $\Delta_z^0\triangleq \|z^0-z_\star^0\|_2^2$, and $\Delta\triangleq \max\{\Delta_\Phi^0,\Delta_x^0, \Delta_y^0,\Delta_z^0\}$.
\end{theorem}
\begin{proof}
    By $L_{\nabla\Phi}$-smoothness of $\Phi$ (\cref{lm:const}), we have
\begin{align}
    &\E\left[\Phi(x^{k+1})\right]\nonumber\\
    \le&\E\left[\Phi(x^k)\right]+\E\left[\langle\nabla\Phi(x^k),x^{k+1}-x^k\rangle\right]+\frac{L_{\nabla\Phi}}{2}\cX_+^k\nonumber\\
    =&\E\left[\Phi(x^k)\right]-\alpha\E\left[\E_k\left[\left\langle\nabla\Phi(x^k),\hat{D}_x^k\right\rangle\right]\right]+\frac{L_{\nabla\Phi}\alpha^2}{2}\EE{\hat{D}_x^k}\nonumber\\
    =&\E\left[\Phi(x^k)\right]-\frac{\alpha}{2}\EE{\nabla\Phi(x^k)}-\frac{\alpha}{2}\EE{\E_k(D_x^k)}+\frac{\alpha}{2}\EE{\nabla\Phi(x^k)-\E_k(D_x^k)}+\frac{L_{\nabla\Phi}\alpha^2}{2}\EE{\hat{D}_x^k}.\label{eq:pfthm-reVCC-1}
\end{align}
Summing \eqref{eq:pfthm-reVCC-1} from $k=0$ to $K-1$, we obtain
\begin{align}
    \sum_{k=0}^{K-1}\EE{\nabla\Phi(x^k)}\le&\frac{2\Phi(x^0)}{{\alpha}}-\sum_{k=0}^{K-1}\EE{\E_k(D_x^k)}+\sum_{k=0}^{K-1}\EE{\nabla\Phi(x^k)-\E_k(D_x^k)}+L_{\nabla\Phi}\alpha\sum_{k=0}^{K-1}\EE{\hat{D}_x^k}.\label{eq:pfthm-reVCC-2}
\end{align}
Applying Lemma \ref{lm:gradi}, \ref{lm:BSMVC}, \ref{lm:LLC-VCC} to \eqref{eq:pfthm-reVCC-2}, we obtain
\begin{align}
    &\frac{1}{K}\sum_{k=0}^{K-1}\EE{\nabla\Phi(x^k)}\nonumber\\
    \le&\frac{2\Delta_\Phi^0}{K\alpha}+\frac{24(L_2^2+17\kappa_g^2L_1^2)\Delta_y^0}{\mu_gK\beta}+\frac{12L_g^2\Delta_z^0}{\mu_gK\gamma}+\frac{48(L_2^2+17\kappa_g^2L_1^2)(1+\omega_\ell )\beta\sigma^2}{\mu_gn}+\left(\frac{L_{\nabla\Phi}(1+\omega_u)\alpha}{n}+\frac{24L_g^2(1+\omega_\ell )\gamma}{\mu_gn}\right.\nonumber\\
    &+\left.\frac{72L_{y^\star}^2(L_2^2+17\kappa_g^2L_1^2)(1+\omega_u)\alpha^2}{\mu_gn\beta}+\frac{24L_g^2(L_{z^\star}^2+L_{zx}\rho)(1+\omega_u)\alpha^2}{\mu_gn\gamma}\right)\sigma_1^2-\frac{C_1}{K}\sum_{k=0}^{K-1}\EE{\E_k(D_x^k)}\nonumber\\
    &+\left(\frac{2L_{\nabla\Phi}\omega_u\alpha}{n}+\frac{96L_g^2\omega_\ell \gamma}{\mu_gn}+\frac{144L_{y^\star}^2(L_2^2+17\kappa_g^2L_1^2)\omega_u\alpha^2}{\mu_gn\beta}+\frac{48L_g^2(L_{z^\star}^2+L_{zx}\rho)\omega_u\alpha^2}{n\mu_g\gamma}\right)b_f^2\nonumber\\
    &+\left(\frac{2L_{\nabla\Phi}\rho^2\omega_u\alpha}{n}+\frac{96(L_2^2+17\kappa_g^2L_1^2)\omega_\ell \beta}{\mu_gn}+\frac{48L_g^2\rho^2\omega_\ell \gamma}{\mu_gn}+\frac{144L_{y^\star}^2\rho^2(L_2^2+16\kappa_g^2L_1^2)\omega_u\alpha^2}{\mu_gn\beta}\right.\nonumber\\
    &+\left.\frac{24L_g^2\rho^2(L_{z^\star}^2+L_{zx}\rho)\omega_u\alpha^2}{\mu_gn\gamma}\right)b_g^2,\label{eq:pfthm-reVCC-3}
\end{align}
where 
\begin{align}
    C_1\triangleq &1-\left(L_{\nabla\Phi}(1+\omega_u/n)\alpha+\frac{72L_{y^\star}^2(L_2^2+17\kappa_g^2L_1^2)(1+\omega_u/n)\alpha^2}{\mu_g\beta}+\frac{24L_g^2(L_{z^\star}^2+L_{zx}\rho)(1+\omega_u/n)\alpha^2}{\mu_g\gamma}\right.\nonumber\\
    &+\left.\frac{48L_{y^\star}^2(L_2^2+17\kappa_g^2L_1^2)\alpha^2}{\mu_g^2\beta^2}+\frac{24L_g^2L_{z^\star}^2\alpha^2}{\mu_g^2\gamma^2}\right).\label{eq:pfthm-reVCC-4}
\end{align}
Note that \eqref{eq:thm-reVCC-alpha} implies $C_1\ge0$, \eqref{eq:thm-reVCC-1} is a direct result of \eqref{eq:pfthm-reVCC-4}.
\end{proof}

\subsection{Proof of \cref{thm:CM-SOBA}}\label{app:CM-SOBA}

Before proving \cref{thm:CM-SOBA}, we need a few additional lemmas.

\begin{lemma}[Lower Level Convergence]\label{lm:LLC-VCM}
    Under Assumptions \ref{asp:conti}, \ref{asp:stron-conve}, \ref{asp:stoch}, \ref{asp:unbia}, \ref{asp:parti}, when 
    $\beta<\min\left\{\frac{2}{\mu_g+L_g},\frac{n\mu_g}{8\omega_\ell  L_g^2}\right\}$, $\gamma\le\min\left\{\frac{1}{L_g},\frac{n\mu_g}{36\omega_\ell  L_g^2}\right\}$, $\rho\ge C_f/\mu_g$, the following inequalities hold for CM-SOBA (Alg.~\ref{alg:C-SOBA}): 
    \begin{align}
        \sum_{k=0}^K\cY^k\le&\frac{2\cY^0}{\beta\mu_g}+\frac{4\beta K(1+\omega_\ell )\sigma^2}{n\mu_g}+\frac{8\beta K\omega_\ell b_g^2}{n\mu_g}+\frac{4L_{y^\star}^2}{\beta^2\mu_g^2}\sum_{k=0}^{K-1}\cX_+^k,\label{eq:lm-LLC-VCM-1}\\
        \sum_{k=0}^K\cZ^k\le&\frac{4\cZ^0}{\gamma\mu_g}+\frac{50L_1^2\cY^0}{\beta\mu_g^3}+\frac{100\beta K(1+\omega_\ell )L_1^2\sigma^2}{n\mu_g^3}+\frac{6K(1+\omega_\ell )\gamma\sigma_1^2}{\mu_gn}\nonumber\\
        &+\left(\frac{100L_{y^\star}^2L_1^2}{\beta^2\mu_g^4}+\frac{12L_{z^\star}^2}{\gamma^2\mu_g^2}\right)\sum_{k=0}^{K-1}\cX_+^k+\frac{(200\beta L_1^2+24\gamma \rho^2\mu_g^2) K\omega_\ell b_g^2}{n\mu_g^3}+\frac{24K\omega_\ell \gamma b_f^2}{\mu_gn}.\label{eq:lm-LLC-VCM-2}
    \end{align}
\end{lemma}

\begin{proof}
    By Assumptions \ref{asp:stoch} and \ref{asp:unbia}, $\hat{D}_y^k$ is an unbiased estimator of $\nabla_yG^k$, thus
    \begin{align}
        \EE{y^{k+1}-y_\star^k}=&\EE{y^k-\beta\hat{D}_y^k-y_\star^k}=\EE{(y^k-y_\star^k-\beta\nabla_yG^k)-\beta(\hat{D}_y^k-\nabla_yG^k)}\nonumber\\
        =&\EE{y^k-y_\star^k-\beta\nabla_yG^k}+\EE{\beta(\hat{D}_y^k-\nabla_yG^k)}\nonumber\\
        \le&(1-\beta\mu_g)^2\cY^k+\beta^2\EE{\hat{D}_y^k-\nabla_yG^k},\label{eq:pflm-LLC-VCM-1}
    \end{align}
    where the inequality uses \cref{lm:desce}. Consequently, we have
\begin{align}
    \cY^{k+1}\le&(1+\beta\mu_g)\EE{y^{k+1}-y_\star^k}+\left(1+\frac{1}{\beta\mu_g}\right)\EE{y_\star^{k+1}-y_\star^k}\nonumber\\
    \le&(1-\beta\mu_g)\cY^k+2\beta^2\EE{\hat{D}_y^k-\nabla_yG^k}+\frac{2L_{y^\star}^2}{\beta\mu_g}\cX_+^k,\label{eq:pflm-LLC-VCM-2}
\end{align}
where the first inequality uses Young's inequality, and the second inequality uses \eqref{eq:pflm-LLC-VCM-1}, $\beta<1/\mu_g$ and \cref{lm:const}. Summing \eqref{eq:pflm-LLC-VCM-2} from $k=0$ to $K-1$ and applying \cref{lm:GVCEL}, we obtain
\begin{align}
    \sum_{k=0}^{K}\cY^k\le&\frac{\cY^0}{\beta\mu_g}+\frac{2\beta}{\mu_g}\sum_{k=0}^{K-1}\EE{\hat{D}_y^k-\nabla_yG^k}+\frac{2L_{y^\star}^2}{\beta^2\mu_g^2}\sum_{k=0}^{K-1}\cX_+^k\nonumber\\
    \le&\frac{\cY^0}{\beta\mu_g}+\frac{2\beta K(1+\omega_\ell )\sigma^2}{n\mu_g}+\frac{4\beta K\omega_\ell b_g^2}{n\mu_g}+\frac{2L_{y^\star}^2}{\beta^2\mu_g^2}\sum_{k=0}^{K-1}\cX_+^k+\frac{4\omega_\ell  L_g^2\beta}{n\mu_g}\sum_{k=0}^K\cY^k.\label{eq:pflm-LLC-VCM-3}
\end{align}
By $\beta\le\frac{n\mu_g}{8\omega_\ell  L_g^2}$, \eqref{eq:pflm-LLC-VCM-3} implies \eqref{eq:lm-LLC-VCM-1}. Similarly,
\begin{align}
    \EE{z^{k+1}-z_\star^k}\le&\EE{\tilde{z}^{k+1}-z_\star^k}=\EE{z^k-\gamma\E_k(D_z^k)-z_\star^k}+\gamma^2\EE{\E_k(D_z^k)-\hat{D}_z^k}\nonumber\\
    =&\EE{(z^k-z_\star^k)-\gamma\nabla_{yy}^2G^k(z^k-z_\star^k)-\gamma(\nabla_{yy}^2G^kz_\star^k+\nabla_yF^k)}+\gamma^2\EE{\E_k(D_z^k)-\hat{D}_z^k}\nonumber\\
    \le&(1+\gamma\mu_g)(1-\gamma\mu_g)^2\cZ^k+\left(1+\frac{1}{\gamma\mu_g}\right)\gamma^2\cdot2L_1^2\cY^k+\gamma^2\EE{\E_k(D_z^k)-\hat{D}_z^k},\label{eq:pflm-LLC-VCM-4}
\end{align}
where the first inequality uses $\rho\ge C_f/\mu_g$ and \cref{lm:const}, the second inequality uses Young's inequality, $\|I-\gamma\nabla_{yy}^2G^k\|_2\le1-\gamma\mu_g$ and \[
\EE{(\nabla_{yy}^2G^k-\nabla_{yy}^2G_\star^k)z_\star^k+(\nabla_yF^k-\nabla_yF_\star^k)}\le2\left(L_f^2+L_{g_{yy}}^2\frac{C_f^2}{\mu_g^2}\right)\cY^k.\]
Consequently,
\begin{align}
    \cZ^{k+1}\le&\left(1+\frac{\gamma\mu_g}{2}\right)\EE{z^{k+1}-z_\star^k}+\left(1+\frac{2}{\gamma\mu_g}\right)\EE{z_\star^{k+1}-z_\star^k}\nonumber\\
    \le&\left(1-\frac{\gamma\mu_g}{2}\right)\cZ^k+\frac{6L_1^2\gamma}{\mu_g}\cY^k+\frac{3\gamma^2}{2}\EE{\hat{D}_z^k-\E_k(D_z^k)}+\frac{3L_{z^\star}^2}{\gamma\mu_g}\cX_+^k,\label{eq:pflm-LLC-VCM-5}
\end{align}
where the first inequality uses Young's inequality and the second inequality uses \eqref{eq:pflm-LLC-VCM-4}, $\gamma\le1/\mu_g$ and \cref{lm:const}.
Summing \eqref{eq:pflm-LLC-VCM-5} from $K=0$ to $K-1$, we achieve
\begin{align}
    \sum_{k=0}^K\cZ^k\le&\frac{2\cZ^0}{\gamma\mu_g}+\frac{12L_1^2}{\mu_g^2}\sum_{k=0}^K\cY^k+\frac{3\gamma}{\mu_g}\sum_{k=0}^{K-1}\EE{\hat{D}_z^k-\E_k(D_z^k)}+\frac{6L_{z^\star}^2}{\gamma^2\mu_g^2}\sum_{k=0}^{K-1}\cX_+^k\nonumber\\
    \le&\frac{2\cZ^0}{\gamma\mu_g}+\frac{3K(1+\omega_\ell )\gamma\sigma_1^2}{\mu_gn}+\frac{6L_{z^\star}^2}{\gamma^2\mu_g^2}\sum_{k=0}^{K-1}\cX_+^k+\left(\frac{12L_1^2}{\mu_g^2}+\frac{18\omega_\ell  L_1^2\gamma}{n\mu_g}\right)\sum_{k=0}^K\cY^k+\frac{18\omega_\ell  L_g^2\gamma}{n\mu_g}\sum_{k=0}^K\cZ^k+\frac{12K\omega_\ell \gamma b_f^2}{\mu_gn}\nonumber\\
    &+\frac{12\rho^2K\omega_\ell \gamma b_g^2}{n\mu_g}\nonumber\\
    \le&\frac{2\cZ^0}{\gamma\mu_g}+\frac{3K(1+\omega_\ell )\gamma\sigma_1^2}{\mu_gn}+\frac{6L_{z^\star}^2}{\gamma^2\mu_g^2}\sum_{k=0}^{K-1}\cX_+^k+\frac{25L_1^2}{2\mu_g^2}\sum_{k=0}^K\cY^k+\frac{1}{2}\sum_{k=0}^K\cZ^k+\frac{12K\omega_\ell \gamma b_f^2}{\mu_gn}+\frac{12\rho^2K\omega_\ell \gamma b_g^2}{n\mu_g},\label{eq:pflm-LLC-VCM-6}
\end{align}
where the second inequality uses \cref{lm:GVCEL} and $\rho\ge C_f/\mu_g$, the third inequality uses $\gamma\le\frac{n\mu_g}{36\omega_\ell  L_g^2}$. \eqref{eq:pflm-LLC-VCM-6} further implies 
\begin{align}
    \sum_{k=0}^K\cZ^k\le&\frac{4\cZ^0}{\gamma\mu_g}+\frac{6K(1+\omega_\ell )\gamma\sigma_1^2}{\mu_gn}+\frac{12L_{z^\star}^2}{\gamma^2\mu_g^2}\sum_{k=0}^{K-1}\cX_+^k+\frac{25L_1^2}{\mu_g^2}\sum_{k=0}^K\cY^k+\frac{24K\omega_\ell \gamma b_f^2}{\mu_gn}+\frac{24\rho^2K\omega_\ell \gamma b_g^2}{n\mu_g}.\label{eq:pflm-LLC-VCM-7}
\end{align}
Applying \eqref{eq:lm-LLC-VCM-1} to \eqref{eq:pflm-LLC-VCM-7}, we achieve \eqref{eq:lm-LLC-VCM-2}.
\end{proof}

\begin{lemma}[Global Vanilla Compression Error in Upper Level]\label{lm:GVCEU}
    Under Assumptions \ref{asp:conti}, \ref{asp:stron-conve}, \ref{asp:stoch}, \ref{asp:unbia}, \ref{asp:parti}, the following inequality holds for CM-SOBA (Alg.~\ref{alg:C-SOBA}): 
    \begin{align}
        \sum_{k=0}^{K-1}\EE{\hat{D}_x^k-\E_k(D_x^k)}\le&\frac{(1+\omega_u)K\sigma_1^2}{n}+\frac{6\omega_u}{n}\sum_{k=0}^{K-1}\EE{\nabla\Phi(x^k)}+\frac{6\omega_u L_2^2}{n}\sum_{k=0}^{K}\cY^k+\frac{6\omega_u L_g^2}{n}\sum_{k=0}^{K}\cZ^k\nonumber\\
        &+\frac{6\omega_uKb_f^2}{n}+\frac{6\rho^2\omega_uKb_g^2}{n}.\label{eq:lm-GVCEU-1}
    \end{align}
\end{lemma}
\begin{proof}
By Assumptions \ref{asp:stoch} and \ref{asp:unbia}, we have
    \begin{align}
        \EE{\hat{D}_x^k-\E_k(D_x^k)}=&\EE{(\hat{D}_x^k-D_x^k)+(D_x^k-\E_k(D_x^k))}=\EE{\hat{D}_x^k-D_x^k}+\EE{D_x^k-\E_k(D_x^k)}\nonumber\\
        =&\frac{1}{n^2}\sum_{i=1}^n\EE{\cC_i^u(D_{x,i}^k)-D_{x,i}^k}+\frac{1}{n^2}\sum_{i=1}^n\EE{D_{x,i}^k-\E_k(D_{x,i}^k)}\nonumber\\
        \le&\frac{\omega_u}{n^2}\sum_{i=1}^n\EE{D_{x,i}^k}+\frac{\sigma_1^2}{n},\label{eq:pflm-GVCEU-1}
    \end{align}
    where the inequality uses \cref{asp:unbia} and \cref{lm:varia}.
    We next bound the second moment of $D_{x,i}^k$.
    \begin{align}
        \EE{D_{x,i}^k}=&\EE{(D_{x,i}^k-\E_k(D_{x,i}^k))+\E_k(D_{x,i}^k)}\nonumber\\
        =&\EE{D_{x,i}^k-\E_k(D_{x,i}^k)}+\EE{\E_k(D_{x,i}^k)}\nonumber\\
        \le&\sigma_1^2+6\EE{\nabla\Phi(x^k)}+6\EE{\nabla_xF_{\star,i}^k-\nabla_xF_\star^k}+6\EE{(\nabla_{xy}^2G_{\star,i}^k-\nabla_{xy}^2G_\star^k)z^k}\nonumber\\
        &+6\EE{(\nabla_{xy}^2G_i^k-\nabla_{xy}^2G_{\star,i}^k)z^k}+6\EE{\nabla_{xy}^2G_\star^k(z^k-z_\star^k)}+6\EE{\nabla_xF_i^k-\nabla_xF_{\star,i}^k}\nonumber\\
        \le&\sigma_1^2+6\EE{\nabla\Phi(x^k)}+6\EE{\nabla_xF_{\star,i}^k-\nabla_xF_\star^k}+6\rho^2\EE{\nabla_{xy}^2G_{\star,i}^k-\nabla_{xy}^2G_\star^k}+6L_2^2\cY^k\nonumber\\
        &+6L_g^2\cZ^k,\label{eq:pflm-GVCEU-2}
    \end{align}
    where the first inequality uses \cref{lm:varia} and Cauchy-Schwarz inequality, and the second inequality uses \cref{asp:conti} and \cref{lm:const}.
    Combining \eqref{eq:pflm-GVCEU-1} \eqref{eq:pflm-GVCEU-2} and applying \cref{asp:parti}, we obtain 
    \begin{align}
        \EE{\hat{D}_x^k-\E_k(D_x^k)}\le&\frac{(1+\omega_u)\sigma_1^2}{n}+\frac{6\omega_u}{n}\EE{\nabla\Phi(x^k)}+\frac{6\omega_uL_2^2}{n}\cY^k+\frac{6\omega_uL_g^2}{n}\cZ^k+\frac{6\omega_ub_f^2}{n}+\frac{6\rho^2\omega_ub_g^2}{n}.\label{eq:pflm-GVCEU-3}
    \end{align}
    Summing \eqref{eq:pflm-GVCEU-3} from $k=0$ to $K-1$ achieves \eqref{eq:lm-GVCEU-1}.
\end{proof}

\begin{lemma}[Momentum-Gradient Bias]\label{lm:momen}
    Under Assumptions \ref{asp:conti}, \ref{asp:stron-conve}, \ref{asp:stoch}, \ref{asp:unbia}, \ref{asp:parti}, assuming $\rho\ge C_f/\mu_g$, the following inequality holds for CM-SOBA (Alg.~\ref{alg:C-SOBA}):
    \begin{align}
        \sum_{k=0}^{K-1}\EE{h_x^k-\nabla\Phi(x^k)}\le&\frac{\|h_x^0-\nabla\Phi(x^0)\|^2}{\theta}+\frac{(1+\omega_u)K\theta\sigma_1^2}{n}+\frac{6\omega_u\theta}{n}\sum_{k=0}^{K-1}\EE{\nabla\Phi(x^k)}\nonumber\\
        &+6L_2^2\left(1+\frac{\omega_u\theta}{n}\right)\sum_{k=0}^K\cY^k+6L_g^2\left(1+\frac{\omega_u\theta}{n}\right)\sum_{k=0}^K\cZ^k+\frac{2L_{\nabla\Phi}^2}{\theta^2}\sum_{k=0}^{K-1}\cX_+^k\nonumber\\
        &+\frac{6\omega_uK\theta b_f^2}{n}+\frac{6\rho^2\omega_uK\theta b_g^2}{n}.\label{eq:lm-momen-1}
    \end{align}
    Further assuming $\beta<\min\left\{\frac{2}{\mu_g+L_g},\frac{n\mu_g}{8\omega_\ell  L_g^2}\right\}$, $\gamma\le\min\left\{\frac{1}{L_g},\frac{n\mu_g}{36\omega_\ell  L_g^2}\right\}$, $\theta\le\min\left\{1,\frac{n}{12\omega_u}\right\}$ and applying \cref{lm:LLC-VCM}, we have
    \begin{align*}
        \sum_{k=0}^{K-1}\EE{h_x^k-\nabla\Phi(x^k)}\le&\frac{\|h_x^0-\nabla\Phi(x^0)\|^2}{\theta}+\frac{13(L_2^2+25\kappa_g^2L_1^2)\cY^0}{\beta\mu_g}+\frac{26L_g^2\cZ^0}{\gamma\mu_g}\nonumber\\
        &+\frac{26(L_2^2+25\kappa_g^2L_1^2)(1+\omega_\ell )K\beta\sigma^2}{n\mu_g}+\left(\frac{(1+\omega_u)\theta}{n}+\frac{39L_g^2(1+\omega_\ell )\gamma}{\mu_gn}\right)\cdot K\sigma_1^2\nonumber\\
        &+\left(\frac{2L_{\nabla\Phi}^2}{\theta^2}+\frac{26L_{y^\star}^2(L_2^2+25\kappa_g^2L_1^2)}{\beta^2\mu_g^2}+\frac{78L_{z^\star}^2L_g^2}{\gamma^2\mu_g^2}\right)\sum_{k=0}^{K-1}\cX_+^k\nonumber\\
        &+\left(\frac{6\rho^2\omega_u\theta}{n}+\frac{52(L_2^2+25\kappa_g^2L_1^2)\omega_\ell \beta}{n\mu_g}+\frac{156L_g^2\rho^2\omega_\ell \gamma}{n\mu_g}\right)\cdot Kb_g^2\nonumber\\
        &+\left(\frac{6\omega_u\theta}{n}+\frac{156L_g^2\omega_\ell \gamma}{n\mu_g}\right)\cdot Kb_f^2+\frac{1}{2}\sum_{k=0}^{K-1}\EE{\nabla\Phi(x^k)}.\nonumber
    \end{align*}
\end{lemma}
\begin{proof}
Note that $\hat{D}_x^k$ and $x^{k+1}$ are mutually independent conditioned on $\cF^k$, we have
\begin{align}
    &\EE{h_x^{k+1}-\nabla\Phi(x^{k+1})}\nonumber\\
    =&\EE{(1-\theta)(h_x^k-\nabla\Phi(x^k))+\theta(\hat{D}_x^k-\E_k(D_x^k))+\theta(\E_k(D_x^k)-\nabla\Phi(x^k))+(\nabla\Phi(x^k)-\nabla\Phi(x^{k+1})}\nonumber\\
    =&\EE{(1-\theta)(h_x^k-\nabla\Phi(x^k))+\theta(\E_k(D_x^k)-\nabla\Phi(x^k))+(\nabla\Phi(x^k)-\nabla\Phi(x^{k+1})}+\theta^2\EE{\hat{D}_x^k-\E_k(D_x^k)}.\label{eq:pflm-momen-1}
\end{align}
By Jensen's inequality, 
\begin{align}
    &\EE{(1-\theta)(h_x^k-\nabla\Phi(x^k))+\theta(\E_k(D_x^k)-\nabla\Phi(x^k))+(\nabla\Phi(x^k)-\nabla\Phi(x^{k+1})}\nonumber\\
    \le&(1-\theta)\EE{h_x^k-\nabla\Phi(x^k)}+\theta\EE{(\E_k(D_x^k)-\nabla\Phi(x^k))+\frac{1}{\theta}\cdot(\nabla\Phi(x^k)-\nabla\Phi(x^{k+1}))}\nonumber\\
    \le&(1-\theta)\EE{h_x^k-\nabla\Phi(x^k)}+2\theta\EE{\E_k(D_x^k)-\nabla\Phi(x^k)}+\frac{2}{\theta}\EE{\nabla\Phi(x^k)-\nabla\Phi(x^{k+1})}\nonumber\\
    \le&(1-\theta)\EE{h_x^k-\nabla\Phi(x^k)}+2\theta\EE{\E_k(D_x^k)-\nabla\Phi(x^k)}+\frac{2L_{\nabla\Phi}^2}{\theta}\cX_+^k,\label{eq:pflm-momen-2}
\end{align}
where the second inequality uses Cauchy-Schwarz inequality, and the third inequality uses \cref{lm:const}.
    Summing \eqref{eq:pflm-momen-1}\eqref{eq:pflm-momen-2} from $k=0$ to $K-1$ and applying Lemma \ref{lm:gradi} and \ref{lm:GVCEU}, we obtain \eqref{eq:lm-momen-1}.
\end{proof}

Now we are ready to prove \cref{thm:CM-SOBA}. We first restate the theorem in a more detailed way.

\begin{theorem}[Convergence of CM-SOBA]\label{thm:re-VCM}
    Under Assumptions \ref{asp:conti}, \ref{asp:stron-conve}, \ref{asp:stoch}, \ref{asp:unbia}, \ref{asp:parti} and assuming $\beta<\min\left\{\frac{2}{\mu_g+L_g},\frac{\mu_gn}{8\omega_\ell  L_g^2}\right\}$, $\gamma\le\min\left\{\frac{1}{L_g},\frac{\mu_g n}{36\omega_\ell  L_g^2}\right\}$, $\theta\le\min\left\{1,\frac{n}{12\omega_u}\right\}$, $\rho\ge C_f\mu_g$, $\alpha\le\min\{\frac{1}{2L_{\nabla\Phi}},C_2\}$ with
    \begin{align*}        C_2^{-2}\triangleq &2\cdot\left(\frac{2L_{\nabla\Phi}^2}{\theta^2}+\frac{26(L_2^2+25\kappa_g^2L_1^2)L_{y^\star}^2}{\beta^2\mu_g^2}+\frac{78L_g^2L_{z^\star}^2}{\gamma^2\mu_g^2}\right),
    \end{align*}
    CM-SOBA (Alg.~\ref{alg:C-SOBA}) converges as
    \begin{align}
    &\frac{1}{K}\sum_{k=0}^{K-1}\EE{\nabla\Phi(x^k)}\nonumber\\
    \le&\frac{4\Delta_\Phi^0}{K\alpha}+\frac{2\Delta_x^0}{K\theta}+\frac{26(L_2^2+25\kappa_g^2L_1^2)\Delta_y^0}{\mu_gK\beta}+\frac{52L_g^2\Delta_z^0}{\mu_gK\gamma}+\frac{52(1+\omega_\ell )(L_2^2+25\kappa_g^2L_1^2)\beta}{\mu_gn}\cdot\sigma^2\nonumber\\
    &+\left(\frac{2(1+\omega_u)\theta}{n}+\frac{78L_g^2(1+\omega_\ell )\gamma}{\mu_gn}\right)\cdot\sigma_1^2+\left(\frac{12\rho^2\omega_u\theta}{n}+\frac{104(L_2^2+25\kappa_g^2L_1^2)\omega_\ell \beta}{n\mu_g}+\frac{312L_g^2\rho^2\omega_\ell \gamma}{n\mu_g}\right)\cdot b_g^2\nonumber\\
    &+\left(\frac{12\omega_u\theta}{n}+\frac{312L_g^2\omega_\ell \gamma}{n\mu_g}\right)\cdot b_f^2.\label{eq:thm-reVCM-1}
\end{align}
If we further choose parameters as 
    \begin{align*}
        \alpha=&\frac{1}{2L_{\nabla\Phi}+C_2^{-1}},\\
        \beta=&\left(\frac{\mu_g+L_g}{2}+\frac{8\omega_\ell  L_g^2}{\mu_gn}+\sqrt{\frac{2K\left((1+\omega_\ell )\sigma^2+2\omega_\ell b_g^2\right)}{n\Delta_y^0}}\right)^{-1},\\
        \gamma=&\left(L_g+\frac{36\omega_\ell  L_g^2}{\mu_gn}+\sqrt{\frac{3K\left((1+\omega_\ell )\sigma_1^2+4\omega_\ell b_f^2+4\omega_\ell (C_f^2/\mu_g^2)b_g^2\right)}{2n\Delta_z^0}}\right)^{-1},\\
        \theta=&\left(1+\frac{12\omega_u}{n}+\sqrt{\frac{K\left((1+\omega_u)\sigma_1^2+6\omega_ub_f^2+6\omega_u(C_f^2/\mu_g^2)b_g^2\right)}{n\Delta_x^0}}\right)^{-1},\\
        \rho=&\frac{C_f}{\mu_g},
    \end{align*}
    CM-SOBA (Alg.~\ref{alg:C-SOBA}) converges as order
    \begin{align*}  \frac{1}{K}\sum_{k=0}^{K-1}\EE{\nabla\Phi(x^K)}=\mathcal{O}\left(\frac{\sqrt{(1+\omega_u+\omega_\ell )\Delta}\sigma+\sqrt{(\omega_u+\omega_\ell )\Delta}(b_f+b_g)}{\sqrt{nK}}+\frac{(1+\omega_u/n+\omega_\ell /n)\Delta}{K}\right)
\end{align*}
    where we define $\Delta_\Phi^0\triangleq \Phi(x^0)$, $\Delta_x^0\triangleq \|h_x^0-\nabla\Phi(x^0)\|_2^2$,  $\Delta_y^0\triangleq \|y^0-y_\star^0\|_2^2$, $\Delta_z^0\triangleq \|z^0-z_\star^0\|_2^2$, and $\Delta\triangleq \max\{\Delta_\Phi^0,\Delta_x^0, \Delta_y^0,\Delta_z^0\}$.
\end{theorem}
\begin{proof}
    By $L_{\nabla\Phi}$-smoothness of $\Phi$ (\cref{lm:const}), we have
\begin{align}
    &\E\left[\Phi(x^{k+1})\right]\nonumber\\
    \le&\E\left[\Phi(x^k)\right]+\E\left[\langle\nabla\Phi(x^k),x^{k+1}-x^k\rangle\right]+\frac{L_{\nabla\Phi}}{2}\EE{x^{k+1}-x^k}\nonumber\\
    =&\E\left[\Phi(x^k)\right]+\E\left[\left\langle\frac{{h}_x^k}{2},x^{k+1}-x^k\right\rangle\right]+\E\left[\left\langle\nabla\Phi(x^k)-\frac{{h}_x^k}{2},x^{k+1}-x^k\right\rangle\right]+\frac{L_{\nabla\Phi}}{2}\EE{x^{k+1}-x^k}\nonumber\\
    =&\E\left[\Phi(x^k)\right]-\left(\frac{1}{2\alpha}-\frac{L_{\nabla\Phi}}{2}\right)\cX_+^k+\frac{\alpha}{2}\EE{\nabla\Phi(x^k)-{h}_x^k}-\frac{\alpha}{2}\EE{\nabla\Phi(x^k)}.\label{eq:pfthm-reVCM-1}
\end{align}
Summing \eqref{eq:pfthm-reVCM-1} from $k=0$ to $K-1$, we have
\begin{align}
    \sum_{k=0}^{K-1}\EE{\nabla\Phi(x^k)}\le&\frac{2\Phi(x^0)}{{\alpha}}-\left(\frac{1}{\alpha^2}-\frac{L_{\nabla\Phi}}{\alpha}\right)\sum_{k=0}^{K-1}\cX_+^k+\sum_{k=0}^{K-1}\EE{{h}_x^k-\nabla\Phi(x^k)}.\label{eq:pfthm-reVCM-2}
\end{align}
By the choice of $\alpha$, we have 
\begin{align*}
    \frac{1}{\alpha^2}-\frac{L_{\nabla\Phi}}{\alpha}\ge\frac{1}{2\alpha^2}\ge\frac{1}{2}C_2^{-2},
\end{align*}
thus by applying \cref{lm:momen} to \eqref{eq:pfthm-reVCM-2} we obtain \eqref{eq:thm-reVCM-1}.
\end{proof}

\subsection{Proof of \cref{thm:EF-SOBA}}\label{app:EF-SOBA}

\begin{lemma}[Bounded Lower Level Updates]\label{lm:BLLU}
    Under Assumptions \ref{asp:conti}, \ref{asp:stron-conve}, the following inequalities hold for Alg.~\ref{alg:EF-SOBA}:
    \begin{align}
        \sum_{k=0}^{K-1}\cY_+^k\le&6\sum_{k=0}^K\cY^k+3L_{y^\star}^2\sum_{k=0}^{K-1}\cX_+^k,\label{eq:lm-BLLU-1}\\
        \sum_{k=0}^{K-1}\cZ_+^k\le&6\sum_{k=0}^K\cZ^k+3L_{z^\star}^2\sum_{k=0}^{K-1}\cX_+^k.\label{eq:lm-BLLU-2}
    \end{align}
\end{lemma}
\begin{proof}
 
    By Cauchy-Schwarz inequality and $L_{y^\star}$-Lipschitz continuity of $y^\star(x)$ (\cref{lm:const}), we have
    \begin{align}
        \cY_+^k=&\EE{(y^{k+1}-y_\star^{k+1})+(y_\star^{k+1}-y_\star^k)+(y_\star^k-y^k)}\nonumber\\
        \le&3\cY^{k+1}+3\EE{y_\star^{k+1}-y_\star^k}+3\cY^k\nonumber\\
        \le&3\cY^{k+1}+3\cY^k+3L_{y_\star}^2\cX_+^k.\label{eq:pflm-BLLU-1}
    \end{align}
    Summing \eqref{eq:pflm-BLLU-1} from $k=0$ to $K-1$ obtains \eqref{eq:lm-BLLU-1}.
    Similary, \eqref{eq:lm-BLLU-2} is achieved by applying Cauchy-Schwarz inequality and $L_{z^\star}$-Lipschitz continuity of $z^\star(x)$.
\end{proof}

The following Lemma describes the contractive property when multiplying $(1+\omega)^{-1}$ to an $\omega$-unbiased compressor.

\begin{lemma}[\cite{he2023lower}, Lemma 1]\label{lm:contr}
    Assume $\cC:\R^{d_{\cC}}\rightarrow\R^{d_{\cC}}$ is an $\omega$-unbiased compressor, then for any $x\in\R^{d_{\cC}}$, it holds that
    \begin{align*}
        \EE{\frac{1}{1+\omega}\cC(x)-x}\le\left(1-\frac{1}{1+\omega}\right)\|x\|_2^2.
    \end{align*}
\end{lemma}

\begin{lemma}[Bounded Difference of Local Update Directions in Lower Level]\label{lm:BDLUD}
    Under Assumptions \ref{asp:conti} and \ref{asp:stoch}, the following inequalities hold for EF-SOBA (Alg.~\ref{alg:EF-SOBA}):
    \begin{align}
        \EE{D_{y,i}^{k+1}-D_{y,i}^{k}}\le&2L_g^2\cX_+^k+2L_g^2\cY_+^k+3\sigma^2,\label{eq:lm-BDLUD-1}\\
        \EE{D_{z,i}^{k+1}-D_{z,i}^k}\le&6L_1^2\cX_+^k+6L_1^2\cY_+^k+6L_g^2\cZ_+^k+3\sigma_1^2.\label{eq:lm-BDLUD-2}
    \end{align}
\end{lemma}
\begin{proof}
For the first inequality, we have
\begin{align}
    \EE{D_{y,i}^k-D_{y,i}^{k-1}}=&\E\left[\E_k\left[\|D_{y,i}^k-D_{y,i}^{k-1}\|^2\right]\right]\le\EE{\E_k(D_{y,i}^k)-D_{y,i}^{k-1}}+\sigma^2\nonumber\\
    \le&2\EE{\E_k(D_{y,i}^k)-\E_{k-1}(D_{y,i}^{k-1})}+2\EE{D_{y,i}^{k-1}-\E_{k-1}(D_{y,i}^{k-1})}+\sigma^2\nonumber\\
    \le&2L_g^2\cX_+^{k-1}+2L_g^2\cY_+^{k-1}+3\sigma^2,\label{eq:pflm-BDLUD-1}
\end{align}
which is exactly \eqref{eq:lm-BDLUD-1}, where the first inequality uses \cref{lm:varia}, the second inequality uses Cauchy-Schwarz inequality, the third inequality uses  \cref{asp:conti} and \cref{lm:varia}. Similarly, we have
\begin{align}
    \EE{D_{z,i}^k-D_{z,i}^{k-1}}\le2\EE{\E_k(D_{z,i}^k)-\E_{k-1}(D_{z,i}^{k-1})}+3\sigma_1^2.\label{eq:pflm-BDLUD-2}
\end{align}
Since
\begin{align}
    \E_k(D_{z,i}^k)-\E_{k-1}(D_{z,i}^{k-1})=&
    \nabla_{yy}^2G_i^k(z^k-z^{k-1})+(\nabla_{yy}^2G_i^k-\nabla_{yy}^2G_i^{k-1})z^{k-1}+(\nabla_yF_i^k-\nabla_yF_i^{k-1}),\nonumber
\end{align}
by Cauchy-Schwarz inequality and \cref{asp:conti} we have
\begin{align}
    &\EE{\E_k(D_{z,i}^k)-\E_{k-1}(D_{z,i}^{k-1})}
    \le3L_g^2\cZ_+^{k-1}+\left(3L_{g_{yy}}^2\rho^2+3L_f^2\right)\left(\cX_+^{k-1}+\cY_+^{k-1}\right),\nonumber
\end{align}
which together with \eqref{eq:pflm-BDLUD-2} leads to \eqref{eq:lm-BDLUD-2}.
\end{proof}

\begin{lemma}[Memory Bias in Lower Level]\label{lm:MBLL}
Under Assumptions \ref{asp:conti}, \ref{asp:stron-conve}, \ref{asp:stoch}, \ref{asp:unbia}, and assuming $\delta_\ell =(1+\omega_\ell )^{-1}$, the following inequalities hold for EF-SOBA (Alg.~\ref{alg:EF-SOBA}):
\begin{align}
     \sum_{k=0}^{K-2}\frac{1}{n}\sum_{i=1}^n\EE{m_{y,i}^{k+1}-D_{y,i}^k}\le&\frac{2\omega_\ell }{n}\sum_{i=1}^n\left\|\E(D_{y,i}^0)-m_{y,i}^0\right\|_2^2+12\omega_\ell (1+\omega_\ell )L_3^2\sum_{k=0}^{K-2}\cX_+^k+72\omega_\ell (1+\omega_\ell )L_g^2\sum_{k=0}^{K-1}\cY^k\nonumber\\
     &+18\omega_\ell (1+\omega_\ell )(K-1)\sigma^2,\label{eq:lm-MBLL-1}\\
     \sum_{k=0}^{K-2}\frac{1}{n}\sum_{i=1}^n\EE{m_{z,i}^{k+1}-D_{z,i}^k}\le&\frac{2\omega_\ell }{n}\sum_{i=1}^n\left\|\E(D_{z,i}^0)-m_{z,i}^0\right\|_2^2+36\omega_\ell (1+\omega_\ell )L_4^2\sum_{k=0}^{K-2}\cX_+^k+216\omega_\ell (1+\omega_\ell )L_1^2\sum_{k=0}^{K-1}\cY^k\nonumber\\
     &+216\omega_\ell (1+\omega_\ell )L_g^2\sum_{k=0}^{K-1}\cZ^k+18\omega_\ell (1+\omega_\ell )(K-1)\sigma_1^2.\label{eq:lm-MBLL-2}
\end{align}
\end{lemma}
\begin{proof}
By the choice of $\delta_\ell $, we have
\begin{align}
    \EE{m_{y,i}^{k+1}-D_{y,i}^k}=&\EE{m_{y,i}^k+\frac{1}{1+\omega_\ell }\cC_i^\ell (D_{y,i}^k-m_{y,i}^k)-D_{y,i}^k}\nonumber\\
    \le&\left(1-\frac{1}{1+\omega_\ell }\right)\EE{D_{y,i}^k-m_{y,i}^k}\nonumber\\
    \le&\left(1-\frac{1}{2(1+\omega_\ell )}\right)\EE{m_{y,i}^k-D_{y,i}^{k-1}}+3\omega_\ell \EE{D_{y,i}^k-D_{y,i}^{k-1}},\label{eq:pflm-MBLL-1}
\end{align}
where the first inequality uses \cref{lm:contr}, and the second inequality uses Young's inequality. 
Applying \cref{lm:BDLUD} to \eqref{eq:pflm-MBLL-1}, we obtain
\begin{align}
    \EE{m_{y,i}^{k+1}-D_{y,i}^k}\le&\left(1-\frac{1}{2(1+\omega_\ell )}\right)\EE{m_{y,i}^k-D_{y,i}^{k-1}}+6\omega_\ell  L_g^2\cX_+^{k-1}+6\omega_\ell  L_g^2\cY_+^{k-1}+9\omega_\ell \sigma^2.\label{eq:pflm-MBLL-3}
\end{align}
For the initial term, we have
\begin{align}
    \EE{m_{y,i}^1-D_{y,i}^0}\le&\left(1-\frac{1}{1+\omega_\ell }\right)\EE{D_{y,i}^0-m_{y,i}^0}
    \le\frac{\omega_\ell }{1+\omega_\ell }\EE{\E(D_{y,i}^0)-m_{y,i}^0}+\frac{\omega_\ell \sigma^2}{1+\omega_\ell },\label{eq:pflm-MBLL-4}
\end{align}
where the first inequality uses \cref{lm:contr}, and the second inequality uses \cref{lm:varia}.
Averaging \eqref{eq:pflm-MBLL-3} from $i=1$ to $n$ and summing from $k=1$ to $K-2$ and applying \eqref{eq:pflm-MBLL-4}, we reach 
\begin{align}
    \sum_{k=0}^{K-1}\frac{1}{n}\sum_{i=1}^n\EE{m_{y,i}^{k+1}-D_{y,i}^k}\le&\frac{2\omega_\ell }{n}\sum_{i=1}^n\left\|\E(D_{y,i}^0)-m_{y,i}^0\right\|_2^2+12\omega_\ell (1+\omega_\ell )L_g^2\sum_{k=0}^{K-1}\cX_+^k\nonumber\\
    &+12\omega_\ell (1+\omega_\ell )L_g^2\sum_{k=0}^{K-1}\cY_+^k+18\omega_\ell (1+\omega_\ell )K\sigma^2.\label{eq:pflm-MBLL-5}
\end{align}
Applying \cref{lm:BLLU} to \eqref{eq:pflm-MBLL-5}, we reach \eqref{eq:lm-MBLL-1}. Similarly, we have 
\begin{align}
    \EE{m_{z,i}^{k+1}-D_{z,i}^k}\le&\left(1-\frac{1}{2(1+\omega_\ell )}\right)\EE{m_{z,i}^k-D_{z,i}^{k-1}}+3\omega_\ell \EE{D_{z,i}^{k}-D_{z,i}^{k-1}}.\label{eq:pflm-MBLL-6}
\end{align}
Applying \cref{lm:BDLUD} to \eqref{eq:pflm-MBLL-6} leads to 
\begin{align}
    \EE{m_{z,i}^{k+1}-D_{z,i}^k}\le&\left(1-\frac{1}{2(1+\omega_\ell )}\right)\EE{m_{z,i}^k-D_{z,i}^{k-1}}+18\omega_\ell  L_g^2\cZ_+^{k-1}+18\omega_\ell  L_1^2\left(\cX_+^{k-1}+\cY_+^{k-1}\right)+9\omega_\ell \sigma_1^2.\label{eq:pflm-MBLL-7}
\end{align}
For the initial term, we have
\begin{align}
    \EE{m_{z,i}^1-D_{y,i}^0}\le\left(1-\frac{1}{1+\omega_\ell }\right)\EE{D_{z,i}^0-m_{z,i}^0}\le\frac{\omega_\ell }{1+\omega_\ell }\EE{\E(D_{z,i}^0)-m_{z,i}^0}+\frac{\omega_\ell \sigma_1^2}{1+\omega_\ell },\label{eq:pflm-MBLL-8}
\end{align}
where the first inequality uses \cref{lm:contr}, and the second inequality uses \cref{lm:varia}. Averaging from $i=1$ to $n$, summing from $k=1$ to $K-2$, applying \eqref{eq:pflm-MBLL-8} and \cref{lm:BLLU}, we reach \eqref{eq:lm-MBLL-2}. 
\end{proof}

\begin{lemma}[Local EF21 Compression Error in Lower Level]\label{lm:LECEL}
Under Assumptions \ref{asp:conti}, \ref{asp:stron-conve}, \ref{asp:stoch}, \ref{asp:unbia}, assuming $\delta_\ell =(1+\omega_\ell )^{-1}$, the following inequalities hold for EF-SOBA (Alg.~\ref{alg:EF-SOBA}):
\begin{align}
    \sum_{k=0}^{K-1}\frac{1}{n}\sum_{i=1}^n\EE{\hat{D}_{y,i}^k-D_{y,i}^k}\le&\frac{\omega_\ell (1+4\omega_\ell )}{n}\sum_{i=1}^n\left\|\E(D_{y,i}^0)-m_{y,i}^0\right\|_2^2+4\omega_\ell \omega_1 L_3^2\sum_{k=0}^{K-1}\cX_+^k+24\omega_\ell \omega_1 L_g^2\sum_{k=0}^K\cY^k\nonumber\\
    &+6\omega_\ell \omega_1K\sigma^2.\label{eq:lm-LECEL-1}\\
    \sum_{k=0}^{K-1}\frac{1}{n}\sum_{i=1}^n\EE{\hat{D}_{z,i}^k-D_{z,i}^k}\le&\frac{\omega_\ell (1+4\omega_\ell )}{n}\sum_{i=1}^n\left\|\E(D_{z,i}^0)-m_{z,i}^0\right\|_2^2+12\omega_\ell \omega_1L_4^2\sum_{k=0}^{K-1}\cX_+^k+72\omega_\ell \omega_1L_1^2\sum_{k=0}^K\cY^k\nonumber\\
    &+72\omega_\ell \omega_1L_g^2\sum_{k=0}^K\cZ^k+6\omega_\ell \omega_1K\sigma_1^2.\label{eq:lm-LECEL-2}
\end{align}
\end{lemma}
\begin{proof}By the definition of $\hat{D}_{y,i}^k$, we have
\begin{align}
    \EE{\hat{D}_{y,i}^k-D_{y,i}^k}=&\EE{m_{y,i}^k+\cC(D_{y,i}^k-m_{y,i}^k)-D_{y,i}^k}\le\omega_\ell \EE{D_{y,i}^k-m_{y,i}^k}\nonumber\\
    \le&2\omega_\ell \EE{m_{y,i}^k-D_{y,i}^{k-1}}+2\omega_\ell \EE{D_{y,i}^k-D_{y,i}^{k-1}}\nonumber\\
    \le&2\omega_\ell \EE{m_{y,i}^k-D_{y,i}^{k-1}}+4\omega_\ell  L_g^2\cX_+^{k-1}+4\omega_\ell  L_g^2\cY_+^{k-1}+6\omega_\ell \sigma^2,\label{eq:pflm-LECEL-1}
\end{align}
where the first inequality uses \cref{asp:unbia}, the second inequality uses Cauchy-Schwarz inequality, the third inequality uses Cauchy-Schwarz inequality and \cref{lm:BDLUD}.
For the initial term, \cref{asp:stoch} implies
\begin{align}
    \EE{\hat{D}_{y,i}^0-D_{y,i}^0}\le(1+\omega_\ell )\EE{m_{y,i}^1-D_{y,i}^0}\le\omega_\ell \left\|\E(D_{y,i}^0)-m_{y,i}^0\right\|_2^2+\omega_\ell \sigma^2.\label{eq:pflm-LECEL-2}
\end{align}
Averaging \eqref{eq:pflm-LECEL-1} from $i=1$ to $n$, summing from $k=1$ to $K-1$, applying \eqref{eq:pflm-LECEL-2}, Lemma \ref{lm:BLLU} and \ref{lm:MBLL}, we obtain \eqref{eq:lm-LECEL-1}. Similarly, we have
\begin{align}
    \EE{\hat{D}_{z,i}^k-D_{z,i}^k}\le&2\omega_\ell \EE{m_{z,i}^k-D_{z,i}^{k-1}}+12\omega_\ell  L_1^2\cX_+^{k-1}+12\omega_\ell  L_1^2\cY_+^{k-1}+12\omega_\ell  L_g^2\cZ_+^{k-1}+6\omega_\ell \sigma_1^2,\label{eq:pflm-LECEL-3}
\end{align}
and 
\begin{align}
    \EE{\hat{D}_{z,i}^0-D_{z,i}^0}\le\omega_\ell \left\|\E(D_{z,i}^0)-m_{z,i}^0\right\|_2^2+\omega_\ell \sigma_1^2.\label{eq:pflm-LECEL-4}
\end{align}
Averaging \eqref{eq:pflm-LECEL-3} from $i=1$ to $n$, summing from $k=1$ to $K-1$, applying \eqref{eq:pflm-LECEL-4}, Lemma \ref{lm:BLLU} and \ref{lm:MBLL}, we obtain \eqref{eq:lm-LECEL-2}.
\end{proof}

\begin{lemma}[Global EF21 Compression Error in Lower Level]\label{lm:GECEL}
    Under Assumptions \ref{asp:conti}, \ref{asp:stron-conve}, \ref{asp:stoch}, \ref{asp:unbia}, assuming $\delta_\ell =(1+\omega_\ell )^{-1}$, the following inequalities hold for EF-SOBA (Alg.~\ref{alg:EF-SOBA}):
    \begin{align}
        \sum_{k=0}^{K-1}\EE{\nabla_yG^k-\hat{D_y^k}}\le&\frac{\omega_\ell (1+4\omega_\ell )}{n^2}\sum_{i=1}^n\left\|\E(D_{y,i}^0)-m_{y,i}^0\right\|_2^2+\frac{(1+6\omega_\ell \omega_1)K\sigma^2}{n}+\frac{4\omega_\ell  \omega_1L_3^2}{n}\sum_{k=0}^{K-1}\cX_+^k\nonumber\\
        &+\frac{24\omega_\ell \omega_1L_g^2}{n}\sum_{k=0}^K\cY^k,\label{eq:lm-GECEL-1}\\
        \sum_{k=0}^{K-1}\EE{\E_k(D_z^k)-\hat{D}_z^k}\le&\frac{\omega_\ell (1+4\omega_\ell )}{n^2}\sum_{i=1}^n\left\|\E(D_{z,i}^0)-m_{z,i}^0\right\|_2^2+\frac{(1+6\omega_\ell \omega_1)K\sigma_1^2}{n}+\frac{12\omega_\ell \omega_1L_4^2}{n}\sum_{k=0}^{K-1}\cX_+^k\nonumber\\
        &+\frac{72\omega_\ell \omega_1L_1^2}{n}\sum_{k=0}^K\cY^k+\frac{72\omega_\ell \omega_1L_g^2}{n}\sum_{k=0}^K\cZ^k.\label{eq:lm-GECEL-2}
    \end{align}
\end{lemma}
\begin{proof}
By \cref{asp:stoch}, we have
\begin{align}
    \EE{\nabla_yG^k-\hat{D}_y^k}=&\EE{(\hat{D}_y^k-D_y^k)+(D_y^k-\nabla_yG^k)}=\EE{\hat{D}_y^k-D_y^k}+\EE{D_y^k-\nabla_yG^k}\nonumber\\
    \le&\frac{1}{n^2}\sum_{i=1}^n\EE{\hat{D}_{y,i}^k-D_{y,i}^k}+\frac{1}{n}\sigma^2.\label{eq:pflm-GECEL-1}
\end{align}
Summing \eqref{eq:pflm-GECEL-1} from $k=0$ to $K-1$ and applying \cref{lm:LECEL}, we obtain \eqref{eq:lm-GECEL-1}. Similarly, we have
\begin{align}
    \EE{\E_k(D_z^k)-\hat{D}_z^k}\le\frac{1}{n^2}\sum_{i=1}^n\EE{\hat{D}_{z,i}^k-D_{z,i}^k}+\frac{1}{n}\sigma_1^2.\label{eq:pflm-GECEL-2}
\end{align}
Summing \eqref{eq:pflm-GECEL-2} from $k=0$ to $K-1$ and applying \cref{lm:LECEL}, we obtain \eqref{eq:lm-GECEL-2}.
\end{proof}

\begin{lemma}[Lower Level Convergence]\label{lm:LLC-MED}
Under Assumptions \ref{asp:conti}, \ref{asp:stron-conve}, \ref{asp:stoch}, \ref{asp:unbia}, assuming $\beta<\min\left\{\frac{2}{\mu_g+L_g},\frac{\mu_gn}{96\omega_\ell \omega_1L_g^2}\right\}$, $\gamma\le\min\left\{\frac{1}{L_g},\frac{\mu_g n}{432\omega_\ell \omega_1L_g^2}\right\}$, $\rho\ge C_f/\mu_g$, $\delta_\ell =(1+\omega_\ell )^{-1}$, the following inequalities hold for EF-SOBA (Alg.~\ref{alg:EF-SOBA}): 
\begin{align}
    \sum_{k=0}^K\cY^k\le&\frac{2\cY^0}{\beta\mu_g}+\frac{4\beta\omega_\ell (1+4\omega_\ell )}{\mu_gn^2}\sum_{i=1}^n\left\|\E(D_{y,i}^0)-m_{y,i}^0\right\|_2^2+\left(\frac{16\omega_\ell  L_3^2\omega_1\beta}{\mu_gn}+\frac{4L_{y_\star^2}}{\beta^2\mu_g^2}\right)\sum_{k=0}^{K-1}\cX_+^k\nonumber\\
    &+\frac{4\beta K(1+6\omega_\ell \omega_1)\sigma^2}{\mu_gn},\label{eq:lm-LLC-MED-1}\\
    \sum_{k=0}^K\cZ^k\le&\frac{4\cZ^0}{\gamma\mu_g}+\frac{50L_1^2\cY^0}{\beta\mu_g^3}+\frac{6\gamma\omega_\ell (1+4\omega_\ell )}{\mu_gn^2}\sum_{i=1}^n\left\|\E(D_{z,i}^0)-m_{z,i}^0\right\|_2^2+\frac{100L_1^2\beta\omega_\ell (1+4\omega_\ell )}{\mu_g^2n^2}\sum_{i=1}^n\left\|\E(D_{y,i}^0)-m_{y,i}^0\right\|_2^2\nonumber\\
    &+\frac{100K(1+6\omega_\ell \omega_1)L_1^2\beta\sigma^2}{\mu_g^3n}+\frac{6K(1+6\omega_\ell \omega_1)\gamma\sigma_1^2}{\mu_g n}\nonumber\\
    &+\left(\frac{400\omega_\ell \omega_1L_1^2L_3^2\beta}{\mu_g^3n}+\frac{100L_1^2L_{y^\star}^2}{\beta^2\mu_g^4}+\frac{72\omega_\ell \omega_1L_4^2\gamma}{\mu_g n}+\frac{12L_{z^\star}^2}{\gamma^2\mu_g^2}\right)\sum_{k=0}^{K-1}\cX_+^k.\label{eq:lm-LLC-MED-2}
\end{align}
\end{lemma}
\begin{proof}
By Assumptions \ref{asp:stoch} and \ref{asp:unbia}, $\hat{D}_y^k$ is an unbiased estimator of $\nabla_yG^k$, thus we have
\begin{align}
    \EE{y^{k+1}-y_\star^k}=&\EE{y^k-\beta\hat{D}_y^k-y_\star^k}=\EE{(y^k-y_\star^k-\beta\nabla_yG^k)-\beta(\hat{D}_y^k-\nabla_yG^k)}\nonumber\\
    =&\EE{y^k-y_\star^k-\beta\nabla_yG^k}+\EE{\beta(\hat{D}_y^k-\nabla_yG^k)}\nonumber\\
    \le&(1-\beta\mu_g)^2\cY^k+\beta^2\EE{\hat{D}_y^k-\nabla_yG^k},\label{eq:pflm-LLC-MED-1}
\end{align}
where the inequality uses \cref{lm:desce}.
Consequently, 
\begin{align}
    \cY^{k+1}\le&(1+\beta\mu_g)\EE{y^{k+1}-y_\star^k}+\left(1+\frac{1}{\beta\mu_g}\right)\EE{y_\star^{k+1}-y_\star^k}\nonumber\\
    \le&(1-\beta\mu_g)\cY^k+2\beta^2\EE{\hat{D}_y^k-\nabla_yG^k}+\frac{2L_{y^\star}^2}{\beta\mu_g}\cX_+^k,\label{eq:pflm-LLC-MED-2}
\end{align}
where the first inequality uses Young's inequality, the second inequality uses \eqref{eq:pflm-LLC-MED-2}, \cref{lm:const} and $\beta<1/\mu_g$. Summing \eqref{eq:pflm-LLC-MED-2} from $k=0$ to $K-1$ and applying \cref{lm:GECEL}, we obtain
\begin{align}
    \sum_{k=0}^{K}\cY^k\le&\frac{\cY^0}{\beta\mu_g}+\frac{2\beta}{\mu_g}\sum_{k=0}^{K-1}\EE{\hat{D}_y^k-\nabla_yG^k}+\frac{2L_{y^\star}^2}{\beta^2\mu_g^2}\sum_{k=0}^{K-1}\cX_+^k\nonumber\\
    \le&\frac{\cY^0}{\beta\mu_g}+\frac{2\beta\omega_\ell (1+4\omega_\ell )}{n^2\mu_g}\sum_{i=1}^n\left\|\E(D_{y,i}^0)-m_{y,i}^0\right\|_2^2+\frac{2\beta K(1+6\omega_\ell \omega_1)\sigma^2}{n\mu_g}+\left(\frac{8\omega_\ell \omega_1L_3^2\beta}{n\mu_g}+\frac{2L_{y^\star}^2}{\beta^2\mu_g^2}\right)\sum_{k=0}^{K-1}\cX_+^k\nonumber\\
    &+\frac{48\omega_\ell \omega_1L_g^2\beta}{n\mu_g}\sum_{k=0}^K\cY^k.\label{eq:pflm-LLC-MED-3}
\end{align}
By $\beta\le\frac{n\mu_g}{96\omega_\ell \omega_1L_g^2}$, \eqref{eq:pflm-LLC-MED-3} implies
\eqref{eq:lm-LLC-MED-1}. Similarly, by $\rho\ge C_f/\mu_g$ and \cref{lm:const}, we have
\begin{align}
    \EE{z^{k+1}-z_\star^k}\le&\EE{\tilde{z}^{k+1}-z_\star^k}=\EE{z^k-\gamma\E_k(D_z^k)-z_\star^k}+\gamma^2\EE{\E_k(D_z^k)-\hat{D}_z^k}\nonumber\\
    =&\EE{(z^k-z_\star^k)-\gamma\nabla_{yy}^2G^k(z^k-z_\star^k)-\gamma(\nabla_{yy}^2G^kz_\star^k+\nabla_yF^k)}+\gamma^2\EE{\E_k(D_z^k)-\hat{D}_z^k}\nonumber\\
    \le&(1+\gamma\mu_g)(1-\gamma\mu_g)^2\cZ^k+\left(1+\frac{1}{\gamma\mu_g}\right)\gamma^2\cdot2L_1^2\cY^k+\gamma^2\EE{\E_k(D_z^k)-\hat{D}_z^k},\label{eq:pflm-LLC-MED-4}
\end{align}
where the last inequality uses Young's inequality, $\|I-\gamma\nabla_{yy}^2G^k\|_2\le1-\gamma\mu_g$ and \[
\EE{(\nabla_{yy}^2G^k-\nabla_{yy}^2G_\star^k)z_\star^k+(\nabla_yF^k-\nabla_yF_\star^k)}\le2\left(L_f^2+L_{g_{yy}}^2\frac{C_f^2}{\mu_g^2}\right)\cY^k.\]
Consequently,
\begin{align}
    \cZ^{k+1}\le&\left(1+\frac{\gamma\mu_g}{2}\right)\EE{z^{k+1}-z_\star^k}+\left(1+\frac{2}{\gamma\mu_g}\right)\EE{z_\star^{k+1}-z_\star^k}\nonumber\\
    \le&\left(1-\frac{\gamma\mu_g}{2}\right)\cZ^k+\frac{6L_1^2\gamma}{\mu_g}\cY^k+\frac{3\gamma^2}{2}\EE{\hat{D}_z^k-\E_k(D_z^k)}+\frac{3L_{z^\star}^2}{\gamma\mu_g}\cX_+^k,\label{eq:pflm-LLC-MED-5}
\end{align}
where the first inequality uses Young's inequality, the second inequality uses \eqref{eq:pflm-LLC-MED-4}, \cref{lm:const} and $\gamma\le1/\mu_g$. Summing \eqref{eq:pflm-LLC-MED-5} from $K=0$ to $K-1$ and applying \cref{lm:GECEL} we achieve
\begin{align}
    &\sum_{k=0}^K\cZ^k\nonumber\\
    \le&\frac{2\cZ^0}{\gamma\mu_g}+\frac{12L_1^2}{\mu_g^2}\sum_{k=0}^K\cY^k+\frac{3\gamma}{\mu_g}\sum_{k=0}^{K-1}\EE{\hat{D}_z^k-\E_k(D_z^k)}+\frac{6L_{z^\star}^2}{\gamma^2\mu_g^2}\sum_{k=0}^{K-1}\cX_+^k\nonumber\\
    \le&\frac{2\cZ^0}{\gamma\mu_g}+\frac{3\gamma\omega_\ell (1+4\omega_\ell )}{\mu_gn^2}\sum_{i=1}^n\left\|\E(D_{z,i}^0)-m_{z,i}^0\right\|_2^2+\frac{3K(1+6\omega_\ell \omega_1)\gamma\sigma_1^2}{\mu_gn}+\left(\frac{36\omega_\ell \omega_1L_4^2\gamma}{\mu_gn}+\frac{6L_{z^\star}^2}{\gamma^2\mu_g^2}\right)\sum_{k=0}^{K-1}\cX_+^k\nonumber\\
    &+\left(\frac{12L_1^2}{\mu_g^2}+\frac{216\omega_\ell \omega_1L_1^2\gamma}{n\mu_g}\right)\sum_{k=0}^K\cY^k+\frac{216\omega_\ell \omega_1L_g^2\gamma}{n\mu_g}\sum_{k=0}^K\cZ^k\nonumber\\
    \le&\frac{2\cZ^0}{\gamma\mu_g}+\frac{3\gamma\omega_\ell (1+4\omega_\ell )}{\mu_gn^2}\sum_{i=1}^n\left\|\E(D_{z,i}^0)-m_{z,i}^0\right\|_2^2+\frac{3K(1+6\omega_\ell \omega_1)\gamma\sigma_1^2}{\mu_gn}+\left(\frac{36\omega_\ell \omega_1L_4^2\gamma}{\mu_gn}+\frac{6L_{z^\star}^2}{\gamma^2\mu_g^2}\right)\sum_{k=0}^{K-1}\cX_+^k\nonumber\\
    &+\frac{25L_1^2}{2\mu_g^2}\sum_{k=0}^K\cY^k+\frac{1}{2}\sum_{k=0}^K\cZ^k,\label{eq:pflm-LLC-MED-6}
\end{align}
where the last inequality uses $\gamma\le\frac{\mu_gn}{432\omega_\ell \omega_1L_g^2}$. This further implies 
\begin{align}
    \sum_{k=0}^K\cZ^k\le&\frac{4\cZ^0}{\gamma\mu_g}+\frac{6\gamma\omega_\ell (1+4\omega_\ell )}{\mu_gn^2}\sum_{i=1}^n\left\|\E(D_{z,i}^0)-m_{z,i}^0\right\|_2^2+\frac{6K(1+6\omega_\ell \omega_1)\gamma\sigma_1^2}{\mu_gn}+\left(\frac{72\omega_\ell \omega_1L_4^2\gamma}{\mu_gn}+\frac{12L_{z^\star}^2}{\gamma^2\mu_g^2}\right)\sum_{k=0}^{K-1}\cX_+^k\nonumber\\
    &+\frac{25L_1^2}{\mu_g^2}\sum_{k=0}^K\cY^k.\label{eq:pflm-LLC-MED-7}
\end{align}
Applying \eqref{eq:lm-LLC-MED-1} to \eqref{eq:pflm-LLC-MED-7}, we achieve \eqref{eq:lm-LLC-MED-2}.
\end{proof}

\begin{lemma}[Momentum-Gradient Bias]\label{lm:MGB-MED}
    Under Assumptions \ref{asp:conti}, \ref{asp:stron-conve}, \ref{asp:stoch}, assuming $\rho\ge C_f/\mu_g$, the following inequality holds for EF-SOBA (Alg.~\ref{alg:EF-SOBA}):
    \begin{align}
        \sum_{k=0}^{K}\EE{h_x^{k}-\nabla\Phi(x^k)}\le&\frac{\|h_x^0-\nabla\Phi(x^0)\|^2}{\theta}+\frac{K\theta\sigma_1^2}{n}+6L_2^2\sum_{k=0}^K\cY^k+6L_g^2\sum_{k=0}^K\cZ^k+\frac{2L_{\nabla\Phi}^2}{\theta^2}\sum_{k=0}^{K-1}\cX_+^k.\label{eq:lm-MGB-MED-1}
    \end{align}
\end{lemma}
\begin{proof}
Note that $x^{k+1}$ is conditionally independent of $D_x^k$ given filtration $\cF^k$, we have
\begin{align}
    &\EE{h_x^{k+1}-\nabla\Phi(x^{k+1})}\nonumber\\
    =&\EE{(1-\theta)h_x^k+\theta D_x^k-\nabla\Phi(x^{k+1})}\nonumber\\
    =&\E\left[\E_k\left[(1-\theta)(h_x^k-\nabla\Phi(x^k))+\theta(D_x^k-\E_k(D_x^k))+\theta(\E_k(D_x^k)-\nabla\Phi(x^k))+(\nabla\Phi(x^k)-\nabla\Phi(x^{k+1}))\right]\right]\nonumber\\
    =&\EE{(1-\theta)(h_x^k-\nabla\Phi(x^k))+\theta(\E_k(D_x^k)-\nabla\Phi(x^k))+(\nabla\Phi(x^k)-\nabla\Phi(x^{k+1}))}+\theta^2\EE{D_x^k-\E_k(D_x^k)}\nonumber\\
    \le&(1-\theta)\EE{h_x^k-\nabla\Phi(x^k)}+\theta\EE{(\E_k(D_x^k)-\nabla\Phi(x^k))+\frac{1}{\theta}(\nabla\Phi(x^k)-\nabla\Phi(x^{k+1}))}+\frac{\theta^2\sigma_1^2}{n}\nonumber\\
    \le&(1-\theta)\EE{h_x^k-\nabla\Phi(x^k)}+2\theta\EE{\E_k(D_x^k)-\nabla\Phi(x^k)}+\frac{2L_{\nabla\Phi}^2}{\theta}\cX_+^k+\frac{\theta^2\sigma_1^2}{n},\label{eq:pflm-MGB-MED-1}
\end{align}
where the first inequality uses Jensen's inequality and \cref{lm:varia}, the second inequality uses Cauchy-Schwarz inequality and \cref{lm:const}. Summing \eqref{eq:pflm-MGB-MED-1} from $k=0$ to $K-1$ and applying \cref{lm:gradi}, we achieve \eqref{eq:lm-MGB-MED-1}.
\end{proof}

\begin{lemma}[Local Momentum Bias]\label{lm:LMB}
    Under Assumptions \ref{asp:conti}, \ref{asp:stron-conve}, \ref{asp:stoch}, the following inequality holds for EF-SOBA (Alg.~\ref{alg:EF-SOBA}):
    \begin{align}
        \sum_{k=0}^{K-1}\frac{1}{n}\sum_{i=1}^n\EE{h_{x,i}^{k}-\E_k(D_{x,i}^k)}\le&\frac{1}{\theta n}\sum_{i=1}^n\|h_{x,i}^0-\E(D_{x,i}^0)\|^2+\frac{6L_5^2}{\theta^2}\sum_{k=0}^{K-1}\cX_+^k+\frac{36L_2^2}{\theta^2}\sum_{k=0}^K\cY^k+\frac{36L_g^2}{\theta^2}\sum_{k=0}^K\cZ^k\nonumber\\
        &+2K\sigma_1^2.\label{eq:lm-LMB-1}
    \end{align}
\end{lemma}
\begin{proof}
By the update rule of $h_{x,i}^k$, we have
\begin{align}
    &\EE{h_{x,i}^k-\E_k(D_{x,i}^k)}\nonumber\\
    =&\EE{(1-\theta)(h_{x,i}^{k-1}-\E_{k-1}(D_{x,i}^{k-1}))+\theta (D_{x,i}^{k-1}-\E_{k-1}(D_{x,i}^{k-1}))+(\E_{k-1}(D_{x,i}^{k-1})-\E_k(D_{x,i}^k))}\nonumber\\
    \le&(1-\theta)\EE{h_{x,i}^{k-1}-\E(D_{x,i}^{k-1})}+2\theta\sigma_1^2+\frac{2}{\theta}\EE{\E_k(D_{x,i}^k)-\E_{k-1}(D_{x,i}^{k-1})},\label{eq:pflm-LMB-1}
\end{align}
where the inequality uses Jensen's inequality, Cauchy-Schwarz inequality and \cref{lm:varia}. For the last term, we have
\begin{align}
    &\EE{\E_k(D_{x,i}^k)-\E_{k-1}(D_{x,i}^{k-1})}\nonumber\\
    =&\EE{(\nabla_{xy}^2G_i^k-\nabla_{xy}^2G_i^{k-1})z^k+\nabla_{xy}^2G_i^{k-1}(z^k-z^{k-1})+(\nabla F_i^k-\nabla F_i^{k-1})}\nonumber\\
    \le&\left(3L_{g_{xy}}^2\rho^2+3L_f^2\right)\left(\cX_+^{k-1}+\cY_+^{k-1}\right)+3L_g^2\cZ_+^{k-1},\label{eq:pflm-LMB-2}
\end{align}
where the inequality uses Cauchy-Schwarz inequality and \cref{asp:conti}.
Averaging \eqref{eq:pflm-LMB-1} from $i=1$ to $n$, summing from $k=1$ to $K-1$ and applying \eqref{eq:pflm-LMB-2}, we obtain
\begin{align}
    \sum_{k=0}^{K-1}\frac{1}{n}\sum_{i=1}^n\EE{h_{x,i}^k-\E_k(D_{x,i}^k)}\le&\frac{1}{\theta n}\sum_{i=1}^n\|h_{x,i}^0-\E(D_{x,i}^0)\|^2+2K\sigma_1^2+\frac{6L_2^2}{\theta^2}\sum_{k=0}^{K-1}\cX_+^k+\frac{6L_2^2}{\theta^2}\sum_{k=0}^{K-1}\cY_+^k\nonumber\\
    &+\frac{6L_g^2}{\theta^2}\sum_{k=0}^{K-1}\cZ_+^k.\label{eq:pflm-LMB-3}
\end{align}
By applying \cref{lm:BLLU} to \eqref{eq:pflm-LMB-3}, we obtain \eqref{eq:lm-LMB-1}.
\end{proof}

\begin{lemma}[Global EF21 Compression Error in Upper Level]\label{lm:GECEU}
    Under Assumptions \ref{asp:conti}, \ref{asp:stron-conve}, \ref{asp:stoch}, \ref{asp:unbia}, assuming $\delta_u=(1+\omega_u)^{-1}$, and $m_{x,i}^0=h_{x,i}^0$  for $i=1,\cdots,n$, the following inequality holds for EF-SOBA (Alg.~\ref{alg:EF-SOBA}):
    \begin{align}
        \sum_{k=0}^K\EE{\hat{h}_x^k-h_x^k}\le&\frac{6\omega_u(1+\omega_u)\theta}{n}\sum_{i=1}^n\|h_{x,i}^0-\E(D_{x,i}^0)\|^2+36\omega_u(1+\omega_u)L_5^2\sum_{k=0}^{K-1}\cX_+^k+216\omega_u(1+\omega_u)L_2^2\sum_{k=0}^K\cY^k\nonumber\\
        &+216\omega_u(1+\omega_u)L_g^2\sum_{k=0}^K\cZ^k+18\omega_u(1+\omega_u)K\theta^2\sigma_1^2,\label{eq:lm-GECEU-1}
    \end{align}
    where we define $h_x^k\triangleq \frac{1}{n}\sum_{i=1}^nh_{x,i}^k$.
\end{lemma}
\begin{proof}
By \cref{lm:contr}, we have
\begin{align}
    \EE{m_{x,i}^{k+1}-h_{x,i}^{k+1}}=\EE{m_{x,i}^k+\frac{\cC_i^u(h_{x,i}^{k+1}-m_{x,i}^k)}{1+\omega_u}-h_{x,i}^{k+1}}\le\left(1-\frac{1}{1+\omega_u}\right)\EE{h_{x,i}^{k+1}-m_{x,i}^k}.\label{eq:pflm-GECEU-1}
\end{align}
For $\EE{h_{x,i}^{k+1}-m_{x,i}^k}$, we have
\begin{align}
    &\EE{h_{x,i}^{k+1}-m_{x,i}^k}\nonumber\\
    =&\EE{(h_{x,i}^{k}-m_{x,i}^k)+\theta (D_{x,i}^k-\E(D_{x,i}^k))+\theta(\E(D_{x,i}^k)-h_{x,i}^k)}\nonumber\\
    \le&\left(1+\frac{1}{2(1+\omega_u)}\right)\EE{m_{x,i}^k-h_{x,i}^k}+(1+2(1+\omega_u))\theta^2\sigma_1^2+(1+2(1+\omega_u))\theta^2\EE{h_{x,i}^k-\E(D_{x,i}^k)},\label{eq:pflm-GECEU-2}
\end{align}
where the inequality uses Young's inequality and \cref{asp:stoch}. Combining \eqref{eq:pflm-GECEU-1}\eqref{eq:pflm-GECEU-2} achieves
\begin{align}
    \EE{m_{x,i}^{k+1}-h_{x,i}^{k+1}}\le\left(1-\frac{1}{2(1+\omega_u)}\right)\EE{m_{x,i}^k-h_{x,i}^k}+3\omega_u\theta^2\sigma_1^2+3\omega_u\theta^2\EE{h_{x,i}^k-\E(D_{x,i}^k)}.\label{eq:pflm-GECEU-3}
\end{align}
Averaging \eqref{eq:pflm-GECEU-3} from $i=1$ to $n$ and summing from $k=0$ to $K-1$ gives 
\begin{align}
    \sum_{k=0}^K\frac{1}{n}\sum_{i=1}^n\EE{m_{x,i}^k-h_{x,i}^k}\le&\frac{2(1+\omega_u)}{n}\sum_{i=1}^n\|m_{x,i}^0-h_{x,i}^0\|^2+6\omega_u(1+\omega_u)K\theta^2\sigma_1^2\nonumber\\
    &+6\omega_u(1+\omega_u)\theta^2\sum_{k=0}^{K-1}\frac{1}{n}\sum_{i=1}^n\EE{h_{x,i}^k-\E(D_{x,i}^k)}.\label{eq:pflm-GECEU-4}
\end{align}
Applying \cref{lm:LMB} to \eqref{eq:pflm-GECEU-4} and noting that $m_{x,i}^0=h_{x,i}^0$, we obtain
\begin{align}
    &\sum_{k=0}^K\frac{1}{n}\sum_{i=1}^n\EE{m_{x,i}^k-h_{x,i}^k}\nonumber\\
    \le&18\omega_u(1+\omega_u)K\theta^2\sigma_1^2+\frac{6\omega_u(1+\omega_u)\theta}{n}\sum_{i=1}^n\|h_{x,i}^0-\E(D_{x,i}^0)\|^2+36\omega_u(1+\omega_u)L_5^2\sum_{k=0}^{K-1}\cX_+^k\nonumber\\
    &+216\omega_u(1+\omega_u)L_2^2\sum_{k=0}^K\cY^k+216\omega_u(1+\omega_u)L_g^2\sum_{k=0}^K\cZ^k.\label{eq:pflm-GECEU-5}
\end{align}
By the update rules we have $\hat{h}_x^k=\frac{1}{n}\sum_{i=1}^nm_{x,i}^k$, thus \eqref{eq:lm-GECEU-1} is a direct result of \eqref{eq:pflm-GECEU-5} by applying the following inequality:
\begin{align*}
    \sum_{k=0}^K\EE{\hat{h}_x^k-h_x^k}\le\sum_{k=0}^K\frac{1}{n}\sum_{i=1}^n\EE{m_{x,i}^k-h_{x,i}^k}.
\end{align*}
\end{proof}

\begin{lemma}[Update Direction Bias in Upper Level]\label{lm:UDBUL}
    Under Assumptions \ref{asp:conti}, \ref{asp:stron-conve}, \ref{asp:stoch}, \ref{asp:unbia}, assuming $\rho\ge C_f/\mu_g$, $\delta_u=(1+\omega_u)^{-1}$, and  $m_{x,i}^0=h_{x,i}^0$ for $i=1,\cdots,n$, the following inequality holds for EF-SOBA (Alg.~\ref{alg:EF-SOBA}):
    \begin{align}
        \sum_{k=0}^{K}\EE{\hat{h}_x^k-\nabla\Phi(x^k)}\le&\frac{2\|h_x^0-\nabla\Phi(x^0)\|^2}{\theta}+\frac{12\omega_u(1+\omega_u)\theta}{n}\sum_{i=1}^n\|h_{x,i}^0-\E(D_{x,i}^0)\|^2\nonumber\\
        &+\left(\frac{4L_{\nabla\Phi}^2}{\theta^2}+72\omega_u(1+\omega_u)L_5^2\right)\sum_{k=0}^{K-1}\cX_+^k+12\omega_2L_2^2\sum_{k=0}^K\cY^k\nonumber\\
        &+12\omega_2L_g^2\sum_{k=0}^K\cZ^k+\left(\frac{2K\theta}{n}+36\omega_u(1+\omega_u)K\theta^2\right)\sigma_1^2.\label{eq:lm-UDBUL-1}
    \end{align}
    Further assume $\beta<\min\left\{\frac{2}{\mu_g+L_g},\frac{\mu_gn}{96\omega_\ell  \omega_1L_g^2}\right\}$, $\gamma\le\min\left\{\frac{1}{L_g},\frac{\mu_g n}{432\omega_\ell \omega_1L_g^2}\right\}$ and applying \cref{lm:LLC-MED}, we have
    \begin{align}
        &\sum_{k=0}^K\EE{\hat{h}_x^k-\nabla\Phi(x^k)}\nonumber\\
        \le&\frac{2}{\theta}\cdot\|h_x^0-\nabla\Phi(x^0)\|^2+\frac{12\omega_u(1+\omega_u)\theta}{n}\sum_{i=1}^n\|h_{x,i}^0-\E(D_{x,i}^0)\|^2+\frac{24\omega_2(L_2^2+25\kappa_g^2L_1^2)}{\mu_g\beta}\cdot\cY^0+\frac{48\omega_2L_g^2}{\mu_g\gamma}\cdot\cZ^0\nonumber\\
        &+\frac{48\omega_2(1+6\omega_\ell \omega_1)(L_2^2+25\kappa_g^2L_1^2)K\beta}{\mu_gn}\cdot\sigma^2+\left(\frac{2\theta}{n}+36\omega_u(1+\omega_u)\theta^2+\frac{72\omega_2(1+6\omega_\ell \omega_1)L_g^2\gamma}{\mu_gn}\right)K\cdot\sigma_1^2\nonumber\\
        &+\frac{48\omega_2\omega_\ell (1+4\omega_\ell )(L_2^2+25\kappa_g^2L_1^2)\beta}{\mu_gn^2}\sum_{i=1}^n\left\|\E(D_{y,i}^0)-m_{y,i}^0\right\|_2^2+\frac{72\omega_2\omega_\ell (1+4\omega_\ell )L_g^2\gamma}{\mu_gn^2}\sum_{i=1}^n\left\|\E(D_{z,i}^0)-m_{z,i}^0\right\|_2^2\nonumber\\
        &+\left[\frac{4L_{\nabla\Phi}^2}{\theta^2}+72\omega_u(1+\omega_u)L_5^2+12\omega_2(L_2^2+25\kappa_g^2L_1^2)\left(\frac{4L_{y^\star}^2}{\beta^2\mu_g^2}+\frac{16\omega_\ell \omega_1L_3^2\beta}{\mu_gn}\right)\right.\nonumber\\       &+\left.12\omega_2L_g^2\left(\frac{12L_{z^\star}^2}{\gamma^2\mu_g^2}+\frac{72\omega_\ell \omega_1L_4^2\gamma}{\mu_gn}\right)\right]\cdot\sum_{k=0}^{K-1}\cX_+^k.\label{eq:lm-UDBUL-2}
    \end{align}
\end{lemma}
\begin{proof}
By Cauchy-Schwarz inequality, we have
    \begin{align}
        \EE{\hat{h}_x^k-\nabla \Phi(x^k)}\le2\EE{\hat{h}_x^k-h_x^k}+2\EE{h_x^k-\nabla\Phi(x^k)}.\label{eq:pflm-UDBUL-1}
    \end{align}
    Summing \eqref{eq:pflm-UDBUL-1} from $k=0$ to $K$ and applying Lemma \ref{lm:MGB-MED} and \ref{lm:GECEU}, we obtain \eqref{eq:lm-UDBUL-1}.
\end{proof}

Now we are ready to prove \cref{thm:EF-SOBA}. We first restate the theorem in a more detailed way.

\begin{theorem}[Convergence of EF-SOBA]\label{thm:re-MED}
    Under Assumptions \ref{asp:conti}, \ref{asp:stron-conve}, \ref{asp:stoch}, \ref{asp:unbia}, assuming $\beta<\min\left\{\frac{2}{\mu_g+L_g},\frac{\mu_gn}{96\omega_\ell  \omega_1L_g^2}\right\}$, $\gamma\le\min\left\{\frac{1}{L_g},\frac{\mu_g n}{432\omega_\ell \omega_1L_g^2}\right\}$, $\rho\ge C_f/\mu_g$, $\delta_\ell =(1+\omega_\ell )^{-1}$, $\delta_u=(1+\omega_u)^{-1}$, $m_{x,i}^0=h_{x,i}^0$ for $i=1,\cdots,n$, $\alpha\le\min\{\frac{1}{2L_{\nabla\Phi}},C_3\}$ with
    \begin{align*}
        C_3^{-2}\triangleq &2\cdot\left[\frac{4L_{\nabla\Phi}^2}{\theta^2}+72\omega_u(1+\omega_u)L_5^2+12\omega_2(L_2^2+25\kappa_g^2L_1^2)\left(\frac{4L_{y^\star}^2}{\beta^2\mu_g^2}+\frac{16\omega_\ell \omega_1L_3^2\beta}{\mu_gn}\right)\right.\nonumber\\
        &+\left.12\omega_2L_g^2\left(\frac{12L_{z^\star}^2}{\gamma^2\mu_g^2}+\frac{72\omega_\ell \omega_1L_4^2\gamma}{\mu_gn}\right)\right],
    \end{align*}
    EF-SOBA (Alg.~\ref{alg:EF-SOBA}) converges as
    \begin{align}
    \frac{1}{K}\sum_{k=0}^{K-1}\EE{\nabla\Phi(x^k)}\le&\frac{2\Delta_\Phi^0}{K\alpha}+\frac{2\Delta_x^0}{K\theta}+\frac{12\omega_u(1+\omega_u)\theta\Delta_h^0}{K}+\frac{24\omega_2(L_2^2+25\kappa_g^2L_1^2)\Delta_y^0}{\mu_gK\beta}\nonumber\\
    &+\frac{48\omega_2L_g^2\Delta_z^0}{\mu_gK\gamma}+\frac{48\omega_2(1+6\omega_\ell \omega_1)(L_2^2+25\kappa_g^2L_1^2)\beta}{\mu_gn}\cdot\sigma^2\nonumber\\
    &+\left(\frac{2\theta}{n}+36\omega_u(1+\omega_u)\theta^2+\frac{72\omega_2(1+6\omega_\ell \omega_1)L_g^2\gamma}{\mu_gn}\right)\cdot\sigma_1^2\nonumber\\
    &+\frac{48\omega_2\omega_\ell (1+4\omega_\ell )(L_2^2+25\kappa_g^2L_1^2)\beta}{\mu_gKn^2}\sum_{i=1}^n\left\|\E(D_{y,i}^0)-m_{y,i}^0\right\|_2^2\nonumber\\
    &+\frac{72\omega_2\omega_\ell (1+4\omega_\ell )L_g^2\gamma}{\mu_gKn^2}\sum_{i=1}^n\left\|\E(D_{z,i}^0)-m_{z,i}^0\right\|_2^2.\label{eq:thm-reMED-1}
\end{align}
    If we further choose parameters as 
    \begin{align*}
        \alpha=&\frac{1}{2L_{\nabla\Phi}+C_3^{-1}},\quad \rho=C_f/\mu_g,\quad \delta_\ell =\frac{1}{1+\omega_\ell },\quad\delta_u=\frac{1}{1+\omega_u},\\
        \beta=&\left(\frac{\mu_g+L_g}{2}+\frac{96\omega_\ell \omega_1L_g^2}{\mu_gn}+\sqrt{\frac{2\omega_\ell (1+4\omega_\ell )\Delta_{m_y}^0}{n\Delta_y^0}}+\sqrt{\frac{2K\sigma^2(1+6\omega_\ell \omega_1)}{n\Delta_y^0}}\right)^{-1},\\
        \gamma=&\left(L_g+\frac{432\omega_\ell \omega_1L_g^2}{\mu_gn}+\sqrt{\frac{3\omega_\ell (1+4\omega_\ell )\Delta_{m_z}^0}{2n\Delta_z^0}}+\sqrt{\frac{K\sigma_1^2(1+6\omega_\ell \omega_1)}{n\Delta_z^0}}\right)^{-1},\\
        \theta=&\left(1+\sqrt{\frac{6\omega_u(1+\omega_u)\Delta_h^0}{\Delta_x^0}}+\sqrt{\frac{K\sigma_1^2}{n\Delta_x^0}}+\sqrt[3]{\frac{18\omega_u(1+\omega_u)K\sigma_1^2}{\Delta_x^0}}\right)^{-1},
    \end{align*}
    EF-SOBA (Alg.~\ref{alg:EF-SOBA}) converges as order
    \begin{align*}
        \frac{1}{K}\sum_{k=0}^{K-1}\EE{\nabla\Phi(x^k)}=&\mathcal{O}\left(\frac{(1+\omega_u)^{2}(1+\omega_\ell )^{3/2}\sqrt{\Delta}\sigma}{\sqrt{nK}}+\frac{\omega_u^{1/3}(1+\omega_u)^{1/3}\Delta^{2/3}\sigma^{2/3}}{K^{2/3}}+\frac{(1+\omega_u)\Delta}{K}\right.\nonumber\\
        &+\frac{(1+\omega_u)^2\sqrt{\omega_\ell (1+\omega_\ell )}\Delta}{\sqrt{n}K}+\left.\frac{(1+\omega_u)^2\omega_\ell ^3\Delta}{nK}\right),
    \end{align*}
    where we define $\Delta_\Phi^0\triangleq \Phi(x^0)$, $\Delta_x^0\triangleq \|h_x^0-\nabla\Phi(x^0)\|_2^2$, $\Delta_h^0\triangleq \frac{1}{n}\sum_{i=1}^n\|h_{x,i}^0-\E(D_{x,i}^0)\|_2^2$, $\Delta_y^0\triangleq \|y^0-y_\star^0\|_2^2$, $\Delta_z^0\triangleq \|z^0-z_\star^0\|_2^2$, $\Delta_{m_y}^0\triangleq \frac{1}{n}\sum_{i=1}^n\left\|\E(D_{y,i}^0)-m_{y,i}^0\right\|_2^2$, $\Delta_{m_{z}}^0\triangleq \frac{1}{n}\sum_{i=1}^n\left\|\E(D_{z,i}^0)-m_{z,i}^0\right\|_2^2$, and $\Delta\triangleq \max\{\Delta_\Phi^0,\Delta_x^0,\Delta_h^0,\Delta_y^0,\Delta_z^0, \Delta_{m_y}^0, \Delta_{m_z}^0\}$.
\end{theorem}
\begin{proof}
By $L_{\nabla\Phi}$-smoothness of $\Phi$ (\cref{lm:const}), we have
\begin{align}
    &\E\left[\Phi(x^{k+1})\right]\nonumber\\
    \le&\E\left[\Phi(x^k)\right]+\E\left[\langle\nabla\Phi(x^k),x^{k+1}-x^k\rangle\right]+\frac{L_{\nabla\Phi}}{2}\EE{x^{k+1}-x^k}\nonumber\\
    =&\E\left[\Phi(x^k)\right]+\E\left[\left\langle\frac{1}{2}\hat{h}_x^k,x^{k+1}-x^k\right\rangle\right]+\E\left[\left\langle\nabla\Phi(x^k)-\frac{1}{2}\hat{h}_x^k,x^{k+1}-x^k\right\rangle\right]+\frac{L_{\nabla\Phi}}{2}\EE{x^{k+1}-x^k}\nonumber\\
    =&\E\left[\Phi(x^k)\right]-\left(\frac{1}{2\alpha}-\frac{L_{\nabla\Phi}}{2}\right)\cX_+^k+\frac{\alpha}{2}\EE{\nabla\Phi(x^k)-\hat{h}_x^k}-\frac{\alpha}{2}\EE{\nabla\Phi(x^k)}.\label{eq:pfthm-reMED-1}
\end{align}
Summing \eqref{eq:pfthm-reMED-1} from $k=0$ to $K-1$, we have
\begin{align}
    \sum_{k=0}^{K-1}\EE{\nabla\Phi(x^k)}\le&\frac{2\Phi(x^0)}{{\alpha}}-\left(\frac{1}{\alpha^2}-\frac{L_{\nabla\Phi}}{\alpha}\right)\sum_{k=0}^{K-1}\cX_+^k+\sum_{k=0}^{K}\EE{\hat{h}_x^k-\nabla\Phi(x^k)}.\label{eq:pfthm-reMED-2}
\end{align}
By the choice of $\alpha$, we have 
\begin{align*}
    \frac{1}{\alpha^2}-\frac{L_{\nabla\Phi}}{\alpha}\ge\frac{1}{2\alpha^2}\ge\frac{1}{2}C_3^{-2},
\end{align*}
thus by applying \cref{lm:UDBUL} to \eqref{eq:pfthm-reMED-2} we obtain \eqref{eq:thm-reMED-1}.
\end{proof}

\section{Convergence Acceleration}\label{app:CA}
In this section, we present the algorithmic design as well as convergence results of the two variants mentioned in Sec.~\ref{sec:varia}, namely CM-SOBA-MSC and EF-SOBA-MSC.

\subsection{Algorithm Design}

In this subsection, we propose two variants of the proposed algorithms to help further improve the convergence rate. These variants are both based on the multi-step compression (MSC) technique \cite{he2023lower} as shown in Alg.~\ref{alg:MSC}.

\begin{algorithm}[htbp]
    \caption{MSC Module}
    \label{alg:MSC}
    \begin{algorithmic}
        \STATE {\bfseries Input:} vector $x$, $\omega$-unbiased compressor $\cC$, communication rounds $R$;
        \STATE Initialize $v^0=0$;
        \FOR{$r=1,\cdots,R$}   
        \STATE Send compressed vector $\cC(x-v^{r-1})$ to the receiver;
        \STATE Update $v^r=v^{r-1}+(1+\omega)^{-1}\cC(x-v^{r-1})$;
        \ENDFOR
        \STATE {\bfseries Output:} $\MSC{x}{\cC}{R}\triangleq v^R/\left(1-(\omega/(1+\omega))^R\right)$.
    \end{algorithmic}
\end{algorithm}

Note that MSC contains a loop of length $R$, we may also introduce an $R$-time sampling step to balance the computation and communication complexities. Specifically, we use the following steps instead of \eqref{eq:local-Dx-Dy-Dz}:

\begin{align}
    \begin{split}
        D_{x,i}^k\triangleq &\frac{1}{R}\sum_{r=1}^R\nabla_{xy}^2G(x^k,y^k;\xi_i^{k,r})z^k+\nabla_xF(x^k,y^k;\phi_i^{k,r}),\\
        D_{y,i}^k\triangleq &\frac{1}{R}\sum_{r=1}^R\nabla_yG(x^k,y^k;\xi_i^{k,r}),\\
        D_{z,i}^k\triangleq &\frac{1}{R}\sum_{r=1}^R\nabla_{yy}^2G(x^k,y^k;\xi_i^{k,r})z^k+\nabla_yF(x^k,y^k;\phi_i^{k,r}).
    \end{split}\label{eq:MSC-Dx-Dy-Dz}
\end{align}



Intuitively, by replacing $\cC_i^\ell (\cdot)$, $\cC_i^u(\cdot)$ with $\MSC{\cdot}{\cC_i^\ell }{R}$, $\MSC{\cdot}{\cC_i^u}{R}$, \eqref{eq:local-Dx-Dy-Dz} with \eqref{eq:MSC-Dx-Dy-Dz}, the outer loop of the double-loop variants are equivalent to the original algorithm except for reducing the gradient/Jacobian sampling variance $\sigma^2$ by $R$ times, as well as reducing the compression variance $\omega_u$, $\omega_\ell $ to $\omega_u\left(\omega_u/(1+\omega_u)\right)^R$ and $\omega_\ell \left(\omega_\ell /(1+\omega_\ell )\right)^R$, see \cref{lm:MSC}. For convenience, we name the so-generated variants of CM-SOBA and EF-SOBA as CM-SOBA-MSC and EF-SOBA-MSC, respectively. 

A detailed description of CM-SOBA-MSC and EF-SOBA-MSC is in Algorithms \ref{alg:CM-SOBA-MSC} and \ref{alg:EF-SOBA-MSC}, respectively.

\begin{algorithm}[htbp]
    \caption{CM-SOBA-MSC Algorithm}
    \label{alg:CM-SOBA-MSC}
    \begin{algorithmic}
        \STATE {\bfseries Input:} $\alpha$, $\beta$, $\gamma$, $\theta$, $\rho$, $R$, $x^0$, $y^0$, $z^0(\|z^0\|_2\le\rho)$, $h_x^0$;
        \FOR{$k=0,1,\cdots,K-1$}   
        \STATE \textbf{on worker:}
        \STATE Compute $D_{x,i}^k$, $D_{y,i}^k$, $D_{z,i}^k$ as in \eqref{eq:MSC-Dx-Dy-Dz};
        \STATE Send $\MSC{D_{x,i}^{k}}{\cC_i^u}{R}$, $\MSC{D_{y,i}^{k}}{\cC_i^\ell }{R}$, and $\MSC{D_{z,i}^{k}}{\cC_i^\ell }{R}$ to server;
        \STATE \textbf{on server:}
        \STATE \quad      $\hat{D}_x^{k}=\frac{1}{n}\sum_{i=1}^n\MSC{D_{x,i}^{k}}{\cC_i^u}{R}$, 
        \STATE \quad $\hat{D}_y^k=\frac{1}{n}\sum_{i=1}^n\MSC{D_{y,i}^{k}}{\cC_i^\ell }{R}$, 
        \STATE \quad $\hat{D}_z^k=\frac{1}{n}\sum_{i=1}^n\MSC{D_{z,i}^{k}}{\cC_i^\ell }{R}$;
        \STATE \quad $x^{k+1}=x^{k}-\alpha\cdot h_x^{k}$;
        \STATE \quad $y^{k+1}=y^{k}-\beta\cdot \hat{D}_y^k$;
        \STATE \quad $\tilde{z}^{k+1}=z^{k}-\gamma\cdot \hat{D}_z^k$;
        \STATE \quad $z^{k+1}=\mathrm{Clip}(\tilde{z}^{k+1},\rho)$;
        \STATE \quad $h_x^{k+1}=(1-\theta)h_x^k+\theta\cdot\hat{D}_x^k$;
        \STATE Broadcast $x^{k+1},y^{k+1},z^{k+1}$ to all workers;
        \ENDFOR
    \end{algorithmic}
\end{algorithm}

\begin{algorithm}[htbp]
    \caption{EF-SOBA-MSC Algorithm}
    \label{alg:EF-SOBA-MSC}
    \begin{algorithmic}
        \STATE {\bfseries Input:} $\alpha$, $\beta$, $\gamma$, $\theta$, $\rho$, $\delta_u$, $\delta_\ell $, $x^0$, $y^0$, $z^0(\|z^0\|_2\le\rho)$, $\{m_{x,i}^0\}$, $\{m_{y,i}^0\}$, $\{m_{z,i}^0\}$, $\{h_{x,i}^0\}$, $\hat{h}_x^0=\frac{1}{n}\sum_{i=1}^nm_{x,i}^0$, $m_y^0=\frac{1}{n}\sum_{i=1}^nm_{y,i}^0$, $m_z^0=\frac{1}{n}\sum_{i=1}^nm_{z,i}^0$;
        \FOR{$k=0,1,\cdots,K-1$}   
        \STATE \textbf{on worker:}
        \STATE Compute $D_{x,i}^k$, $D_{y,i}^k$, $D_{z,i}^k$ as in \eqref{eq:MSC-Dx-Dy-Dz};
        \STATE \quad$h_{x,i}^{k+1}=(1-\theta)h_{x,i}^k+\theta\cdot D_x^k$;
        \STATE \quad$m_{x,i}^{k+1}=m_{x,i}^k+\delta_u\cdot \MSC{h_{x,i}^{k+1}-m_{x,i}^k}{\cC_i^u}{R}$,
        \STATE \quad$m_{y,i}^{k+1}=m_{y,i}^k+\delta_\ell \cdot \MSC{D_{y,i}^k-m_{y,i}^k}{\cC_i^\ell }{R}$,
        \STATE \quad$m_{z,i}^{k+1}=m_{z,i}^k+\delta_\ell \cdot \MSC{D_{z,i}^k-m_{z,i}^k}{\cC_i^\ell }{R}$;
        \STATE Send $\MSC{h_{x,i}^{k+1}-m_{x,i}^k}{\cC_i^u}{R},\MSC{D_{y,i}^{k}-m_{y,i}^k}{\cC_i^\ell }{R},\MSC{D_{z,i}^{k}-m_{z,i}^k}{\cC_i^\ell }{R}$ to server;
        \STATE \textbf{on server:}
        \STATE \quad      $\hat{D}_y^k=m_y^k+\frac{1}{n}\sum_{i=1}^n\MSC{D_{y,i}^{k}-m_{y,i}^k}{\cC_i^\ell }{R}$, \STATE \quad $\hat{D}_z^k=m_z^k+\frac{1}{n}\sum_{i=1}^n\MSC{D_{z,i}^{k}-m_{z,i}^k}{\cC_i^\ell }{R}$;
        \STATE \quad $x^{k+1}=x^{k}-\alpha\cdot \hat{h}_x^{k}$;
        \STATE \quad $y^{k+1}=y^{k}-\beta\cdot \hat{D}_y^k$;
        \STATE \quad $\tilde{z}^{k+1}=z^{k}-\gamma\cdot \hat{D}_z^k$;
        \STATE \quad $z^{k+1}=\mathrm{Clip}(\tilde{z}^{k+1},\rho)$;
        \STATE \quad $\hat{h}_x^{k+1}=\hat{h}_x^k+\frac{\delta_u}{n}\sum_{i=1}^n\MSC{h_{x,i}^{k+1}-m_{x,i}^k}{\cC_i^u}{R}$;
        \STATE \quad$m_y^{k+1}=m_y^k+\frac{\delta_\ell }{n}\sum_{i=1}^n\MSC{D_{y,i}^k-m_{y,i}^k}{\cC_i^\ell }{R}$;
        \STATE \quad$m_z^{k+1}=m_z^k+\frac{\delta_\ell }{n}\sum_{i=1}^n\MSC{D_{z,i}^k-m_{z,i}^k}{\cC_i^\ell }{R}$;
        \STATE Broadcast $x^{k+1},y^{k+1},z^{k+1}$ to all workers;
        \ENDFOR
    \end{algorithmic}
\end{algorithm}

\subsection{Convergence of CM-SOBA-MSC and EF-SOBA-MSC}\label{app:MSC}

Similar to \cref{lm:varia}, we have the following lemma for the gradient accumulation mechanism.

\begin{lemma}[Reduced Variance]\label{lm:reduc}
    Under Assumption \ref{asp:stoch}, we have the following variance bounds for CM-SOBA-MSC (Alg.~\ref{alg:CM-SOBA-MSC}) and EF-SOBA-MSC (Alg.~\ref{alg:EF-SOBA-MSC}):
    \begin{align*}
        &\Var{D_{y,i}^k\mid\cF^k}\le\tilde{\sigma}^2,\quad \Var{D_y^k\mid\cF^k}\le\tilde{\sigma}^2/n;\\
        &\Var{D_{x,i}^k\mid\cF^k}\le\tilde{\sigma}_1^2,\quad \Var{D_x^k\mid\cF^k}\le\tilde{\sigma}_1^2/n;\\
        &\Var{D_{z,i}^k\mid\cF^k}\le\tilde{\sigma}_1^2,\quad \Var{D_z^k\mid\cF^k}\le\tilde{\sigma}_1^2/n.
    \end{align*}
\end{lemma}

The following lemma describes the property of the MSC module (Alg.~\ref{alg:MSC}).

\begin{lemma}[\cite{he2023lower}, Lemma 2]\label{lm:MSC}
    Assume $\cC$ is an $\omega$-unbiased compressor, and $R$ is any positive integer. $\MSC{\cdot}{C}{R}$ is then an $\tilde{\omega}$-unbiased compressor with
    \[\tilde{\omega}=\omega\left(\frac{\omega}{1+\omega}\right)^R.
    \]
\end{lemma}

Consequently, $\MSC{\cdot}{\cC_i^u}{R}$ is equivalent to an $\tilde{\omega}_u$-unbiased compressor. Similarly, $\MSC{\cdot}{\cC_i^\ell }{R}$ is equivalent to an $\tilde{\omega}_\ell $-unbiased compressor. The following is a technical lemma.

\begin{lemma}[\cite{he2023lower}, Lemma 11]\label{lm:RwR}
    For $R\ge4(1+\omega)\ln\left(4(1+\omega)\right)$, it holds that
    \[
    R\left(\frac{\omega}{1+\omega}\right)^R\le\left(\frac{\omega}{1+\omega}\right)^{R/2}.
    \]
\end{lemma}

We have the following convergence result for CM-SOBA-MSC (Alg.~\ref{alg:CM-SOBA-MSC}).

\begin{theorem}[Convergence of CM-SOBA-MSC]\label{thm:re-MCM}
    Under Assumptions \ref{asp:conti}, \ref{asp:stron-conve}, \ref{asp:stoch}, \ref{asp:unbia}, \ref{asp:parti} and assuming $\beta<\min\left\{\frac{2}{\mu_g+L_g},\frac{\mu_gn}{8\tilde{\omega}_\ell  L_g^2}\right\}$, $\gamma\le\min\left\{\frac{1}{L_g},\frac{\mu_g n}{36\tilde{\omega}_\ell  L_g^2}\right\}$, $\theta\le\min\left\{1,\frac{n}{12\tilde{\omega}_u}\right\}$, $\rho\ge C_f/\mu_g$, $\alpha\le\min\{\frac{1}{2L_{\nabla\Phi}},C_2\}$ with
    \begin{align*}        C_2^{-2}\triangleq &2\cdot\left(\frac{2L_{\nabla\Phi}^2}{\theta^2}+\frac{26(L_2^2+25\kappa_g^2L_1^2)L_{y^\star}^2}{\beta^2\mu_g^2}+\frac{78L_g^2L_{z^\star}^2}{\gamma^2\mu_g^2}\right),
    \end{align*}
    CM-SOBA-MSC (Alg.~\ref{alg:CM-SOBA-MSC}) converges as
    \begin{align}
    &\frac{1}{K}\sum_{k=0}^{K-1}\EE{\nabla\Phi(x^k)}\nonumber\\
    \le&\frac{4\Delta_\Phi^0}{K\alpha}+\frac{2\Delta_x^0}{K\theta}+\frac{26(L_2^2+25\kappa_g^2L_1^2)\Delta_y^0}{\mu_gK\beta}+\frac{52L_g^2\Delta_z^0}{\mu_gK\gamma}+\frac{52(1+\tilde{\omega}_\ell )(L_2^2+25\kappa_g^2L_1^2)\beta}{\mu_gn}\cdot\tilde{\sigma}^2\nonumber\\
    &+\left(\frac{2(1+\tilde{\omega}_u)\theta}{n}+\frac{78L_g^2(1+\tilde{\omega}_\ell )\gamma}{\mu_gn}\right)\cdot\tilde{\sigma}_1^2+\left(\frac{12\rho^2\tilde{\omega}_u\theta}{n}+\frac{104(L_2^2+25\kappa_g^2L_1^2)\tilde{\omega}_\ell \beta}{n\mu_g}+\frac{312L_g^2\rho^2\tilde{\omega}_\ell \gamma}{n\mu_g}\right)\cdot b_g^2\nonumber\\
    &+\left(\frac{12\tilde{\omega}_u\theta}{n}+\frac{312L_g^2\tilde{\omega}_\ell \gamma}{n\mu_g}\right)\cdot b_f^2.\label{eq:thm-reMCM-1}
\end{align}
If we further choose parameters as 
    \begin{align*} R=&\left\lceil4(1+\omega_u+\omega_\ell )\ln\left(4(1+\omega_u+\omega_\ell )+\frac{(1+\omega_u+\omega_\ell )^2(b_f^2+b_g^2)^2}{\sigma^4}\right)\right\rceil,\\
    \alpha=&\frac{1}{2L_{\nabla\Phi}+C_2^{-1}},\\
    \beta=&\left(\frac{\mu_g+L_g}{2}+\frac{8\tilde{\omega}_\ell  L_g^2}{\mu_gn}+\sqrt{\frac{2K\left((1+\tilde{\omega}_\ell )\tilde{\sigma}^2+2\tilde{\omega}_\ell b_g^2\right)}{n\Delta_y^0}}\right)^{-1},\\
    \gamma=&\left(L_g+\frac{36\tilde{\omega}_\ell  L_g^2}{\mu_gn}+\sqrt{\frac{3K\left((1+\tilde{\omega}_\ell )\tilde{\sigma}_1^2+4\tilde{\omega}_\ell b_f^2+4\tilde{\omega_\ell }(C_f^2/\mu_g^2)b_g^2\right)}{2n\Delta_z^0}}\right)^{-1},\\
    \theta=&\left(1+\frac{12\tilde{\omega}_u}{n}+\sqrt{\frac{K\left((1+\tilde{\omega}_u)\tilde{\sigma}_1^2+6\tilde{\omega}_ub_f^2+6\tilde{\omega}_u(C_f^2/\mu_g^2)b_g^2\right)}{n\Delta_x^0}}\right)^{-1},\\
    \rho=&\frac{C_f}{\mu_g},
\end{align*}
    CM-SOBA-MSC (Alg.~\ref{alg:CM-SOBA-MSC}) converges as order
    \begin{align*}  \frac{1}{K}\sum_{k=0}^{K-1}\EE{\nabla\Phi(x^k)}=\cO\left(\frac{\sqrt{\Delta}\sigma}{\sqrt{nT}}+\frac{(1+\omega_u+\omega_\ell )\Delta\tilde{\Theta}(1)}{T}\right),
\end{align*}
    where $T\triangleq KR$ is the total number of iterations of CM-SOBA-MSC (Alg.~\ref{alg:CM-SOBA-MSC}), $\tilde{\Theta}$ hides logarithmic terms independent of $T$, and $\Delta$ is as defined in \cref{thm:re-VCM}.
\end{theorem}
\begin{proof}
    By Lemma \ref{lm:reduc} and \ref{lm:MSC}, the outer loop of CM-SOBA-MSC (Alg.~\ref{alg:CM-SOBA-MSC}) is equivalent to CM-SOBA (Alg.~\ref{alg:C-SOBA}), except for using gradient/Jacobian oracles and unbiased compressors with variance reduced by a factor of $R$. Thus, \eqref{eq:thm-reMCM-1} is a direct corollary of \cref{thm:re-VCM}. By applying \cref{lm:RwR}, the choice of $R$ implies
    \begin{align*}
        \tilde{\omega}_u\le1,\quad\tilde{\omega}_\ell \le1,\quad R\tilde{\omega}_u\le\frac{\sigma^2}{b_f^2+b_g^2},\quad R\tilde{\omega}_\ell \le\frac{\sigma^2}{b_f^2+b_g^2}.
    \end{align*}
    Consequently, by the choice of $\alpha$, $\beta$, $\gamma$, $\theta$, $\rho$ and $R$, CM-SOBA-MSC (Alg.~\ref{alg:CM-SOBA-MSC}) converges as 
    \begin{align*}
        \frac{1}{K}\sum_{k=0}^{K-1}\EE{\nabla\Phi(x^k)}=&\cO\left(\frac{\sqrt{(1+\tilde{\omega}_u+\tilde{\omega}_\ell )\Delta}\sigma+\sqrt{R(\tilde{\omega}_u+\tilde{\omega}_\ell )\Delta}(b_f+b_g)}{\sqrt{nT}}+\frac{(1+\tilde{\omega}_u/n+\tilde{\omega}_\ell /n)\Delta R}{T}\right)\\
        =&\cO\left(\frac{\sqrt{\Delta}{\sigma}}{\sqrt{nT}}+\frac{(1+\omega_u+\omega_\ell )\Delta\tilde{\Theta}(1)}{T}\right).
    \end{align*}
\end{proof}
For EF-SOBA-MSC (Alg.~\ref{alg:EF-SOBA-MSC}), we have the following notations for convenience:
\begin{align*}
    \tilde{\omega}_1\triangleq 1+6\tilde{\omega}_\ell (1+\tilde{\omega}_\ell ),\quad \tilde{\omega}_2\triangleq 1+36\tilde{\omega}_u(1+\tilde{\omega}_u).
\end{align*}

Now we are ready to prove  \cref{thm:CM-SOBA-MSC}. The detailed result is as follows.

\begin{theorem}[Convergence of EF-SOBA-MSC]\label{thm:re-MEM}
    Under assumptions \ref{asp:conti}, \ref{asp:stron-conve}, \ref{asp:stoch}, \ref{asp:unbia} and assuming $\rho\ge C_f/\mu_g$, $\beta<\min\left\{\frac{2}{\mu_g+L_g},\frac{\mu_gn}{96\tilde{\omega}_\ell \tilde{\omega}_1}L_g^2\right\}$, $\gamma\le\min\left\{\frac{1}{L_g},\frac{\mu_gn}{432\tilde{\omega}_\ell \tilde{\omega}_1L_g^2}\right\}$,  $\delta_\ell =(1+\tilde{\omega}_\ell )^{-1}$, $\delta_u=(1+\tilde{\omega}_u)^{-1}$, $m_{x,i}^0=h_{x,i}^0$ for $i=1,\cdots,n$, $\alpha\le\min\left\{\frac{1}{2L_{\nabla\Phi}},C_4\right\}$ with
\begin{align*}
C_4^{-2}\triangleq &2\cdot\left[\frac{4L_{\nabla\Phi}^2}{\theta^2}+72\tilde{\omega}_u(1+\tilde{\omega}_u)L_5^2+12\tilde{\omega}_2(L_2^2+25\kappa_g^2L_1^2)\left(\frac{4L_{y^\star}^2}{\beta^2\mu_g^2}+\frac{16\tilde{\omega}_\ell \tilde{\omega}_1L_3^2\beta}{\mu_gn}\right)\right.\\
&+\left.12\tilde{\omega}_2L_g^2\left(\frac{12L_{z^\star}^2}{\gamma^2\mu_g^2}+\frac{72\tilde{\omega}_\ell \tilde{\omega}_1L_4^2\gamma}{\mu_gn}\right)\right],
\end{align*}
EF-SOBA-MSC (Alg.~\ref{alg:EF-SOBA-MSC}) converges as
\begin{align}
    \frac{1}{K}\sum_{k=0}^{K-1}\EE{\nabla\Phi(x^k)}\le&\frac{2\Delta_\Phi^0}{K\alpha}+\frac{2\Delta_x^0}{K\theta}+\frac{12\tilde{\omega}_u(1+\tilde{\omega}_u)\theta\Delta_h^0}{K}+\frac{24\tilde{\omega}_2(L_2^2+25\kappa_g^2L_1^2)\Delta_y^0}{\mu_gK\beta}\nonumber\\
    &+\frac{48\tilde{\omega}_2L_g^2\Delta_z^0}{\mu_gK\gamma}+\frac{48\tilde{\omega}_2(1+6\tilde{\omega}_\ell \tilde{\omega}_1)(L_2^2+25\kappa_g^2L_1^2)\beta}{\mu_gn}\cdot\tilde{\sigma}^2\nonumber\\
    &+\left(\frac{2\theta}{n}+36\tilde{\omega}_u(1+\tilde{\omega}_u)\theta^2+\frac{72\tilde{\omega}_2(1+6\tilde{\omega}_\ell \tilde{\omega}_1)L_g^2\gamma}{\mu_gn}\right)\cdot\tilde{\sigma}_1^2\nonumber\\
    &+\frac{48\tilde{\omega}_2\tilde{\omega}_\ell (1+4\tilde{\omega}_\ell )(L_2^2+25\kappa_g^2L_1^2)\beta}{\mu_gKn^2}\sum_{i=1}^n\left\|\E(D_{y,i}^0)-m_{y,i}^0\right\|_2^2\nonumber\\
    &+\frac{72\tilde{\omega}_2\tilde{\omega}_\ell (1+4\tilde{\omega}_\ell )L_g^2\gamma}{\mu_gKn^2}\sum_{i=1}^n\left\|\E(D_{z,i}^0)-m_{z,i}^0\right\|_2^2.\label{eq:thm-reMEM-1}
\end{align}
If we further choose parameters as
\begin{align*}
    &R=\left\lceil4(1+\omega_u+\omega_\ell )\ln\left(4(1+\omega_u+\omega_\ell )+\frac{(1+\omega_u)^2\Delta n^3}{\sigma^2}\right)\right\rceil,\\
    &\alpha=\frac{1}{2L_{\nabla\Phi}+C_4^{-1}},\quad\rho=C_f/\mu_g,\quad\delta_\ell =\frac{1}{1+\tilde{\omega}_\ell },\quad\delta_u=\frac{1}{1+\tilde{\omega}_u},\\
    &\beta=\left(\frac{\mu_g+L_g}{2}+\frac{96\tilde{\omega}_\ell \tilde{\omega}_1L_g^2}{\mu_gn}+\sqrt{\frac{2\tilde{\omega}_\ell (1+4\tilde{\omega}_\ell )\Delta_{m_y}^0}{n\Delta_y^0}}+\sqrt{\frac{2K\tilde{\sigma}^2(1+6\tilde{\omega}_\ell \tilde{\omega}_1)}{n\Delta_y^0}}\right)^{-1},\\
    &\gamma=\left(L_g+\frac{432\tilde{\omega}_\ell \tilde{\omega}_1L_g^2}{\mu_gn}+\sqrt{\frac{3\tilde{\omega}_\ell (1+4\tilde{\omega}_\ell )\Delta_{m_z}^0}{2n\Delta_z^0}}+\sqrt{\frac{K\tilde{\sigma}_1^2(1+6\tilde{\omega}_\ell \tilde{\omega}_1)}{n\Delta_y^0}}\right)^{-1},\\
    &\theta=\left(1+\sqrt{\frac{6\tilde{\omega}_u(1+\tilde{\omega}_u)\Delta_h^0}{\Delta_x^0}}+\sqrt{\frac{K\tilde{\sigma}_1^2}{n\Delta_x^0}}+\sqrt[3]{\frac{18\tilde{\omega}_u(1+\tilde{\omega}_u)K\tilde{\sigma}_1^2}{\Delta_x^0}}\right)^{-1},
\end{align*}
EF-SOBA-MSC (Alg.~\ref{alg:EF-SOBA-MSC}) converges as order
\begin{align*}
    \frac{1}{K}\sum_{k=0}^{K-1}\EE{\nabla\Phi(x^k)}=\cO\left(\frac{\sqrt{\Delta}\sigma}{\sqrt{nT}}+\frac{(1+\omega_u+\omega_\ell )\Delta\tilde{\Theta}(1)}{T}\right),
\end{align*}
where $T:KR$ is the total number of iterations of EF-SOBA-MSC (Alg.~\ref{alg:EF-SOBA-MSC}), $\tilde{\Theta}$ hides logarithmic terms independent of $T$, and $\Delta$ is as defined in \cref{thm:re-MED}.
\end{theorem}

\begin{proof}
    By Lemma \ref{lm:reduc} and \ref{lm:MSC}, the outer loop of EF-SOBA-MSC (Alg.~\ref{alg:EF-SOBA-MSC}) is equivalent to EF-SOBA (Alg.~\ref{alg:EF-SOBA}), except for using gradient/Jacobian oracles and unbiased compressors with variance reduced by a factor of $R$. Thus, \eqref{eq:thm-reMEM-1} is a direct corollary of \cref{thm:re-MED}. By applying \cref{lm:RwR}, the choice of $R$ implies
    \begin{align*}
        \tilde{\omega}_u\le1,\quad\tilde{\omega}_\ell \le1,\quad R\tilde{\omega}_u\le\frac{\sigma}{\sqrt{\Delta}n^{3/2}}.
    \end{align*}
    Consequently, by the choice of $\delta_u$, $\delta_\ell $, $\alpha$, $\beta$, $\gamma$, $\theta$, $\rho$ and $R$, EF-SOBA-MSC (Alg.~\ref{alg:EF-SOBA-MSC}) converges as 
    \begin{align*}
        \frac{1}{K}\sum_{k=0}^{K-1}\EE{\nabla\Phi(x^k)}=&\cO\left(\frac{(1+\tilde{\omega}_u)^2(1+\tilde{\omega}_\ell )^{3/2}\sqrt{\Delta}\sigma}{\sqrt{nT}}+\frac{(R\tilde{\omega}_u)^{1/3}(1+\tilde{\omega}_u)^{1/3}\Delta^{2/3}\sigma^{2/3}}{T^{2/3}}+\frac{(1+\tilde{\omega}_u)\Delta R}{T}\right.\nonumber\\
        &+\left.\frac{(1+\tilde{\omega}_u)^2\sqrt{\tilde{\omega}_\ell (1+\tilde{\omega}_\ell )}\Delta R}{\sqrt{n}T}+\frac{(1+\tilde{\omega}_u)^2\tilde{\omega}_\ell ^3\Delta R}{nT}\right).\\
        =&\cO\left(\frac{\sqrt{\Delta}{\sigma}}{\sqrt{nT}}+\frac{(1+\omega_u+\omega_\ell )\Delta\tilde{\Theta}(1)}{T}\right).
    \end{align*}
\end{proof}

\section{Experimental Specifications}\label{app:exp_spf}
\subsection{Hyper-Representation}\label{app:exp_form}
\textbf{Problem formulation.}
Following \cite{franceschi2018bilevel}, the hyper-representation problem can be formulated as:
\begin{align*}
    \min_\lambda\; L(\lambda)&=\frac{1}{|\mathcal{D}_v|}\sum_{\xi \in \mathcal{D}_v}L(\omega^*(\lambda), \lambda; \xi) \\
    s.t.\; \omega^*(\lambda)&=\arg\min_\omega \frac{1}{|\mathcal{D}_\tau|}\sum_{\eta\in \mathcal{D}_\tau}L(\omega, \lambda; \eta)
\end{align*}
where $L$ stands for the cross entropy loss here, $\mathcal{D}_v$ and $\mathcal{D}_\tau$ denote the validation set and training set, respectively. Hyper-representation consists of two nested problems, where the upper-level optimizes the intermediate representation parameter $\lambda$ to obtain better feature representation on validation data, and the lower-level optimizes the weights $\omega$ of the downstream tasks on training data.

\textbf{Datasets and model architecture.} For MNIST, we use a 2-layer multilayer perceptron (MLP) with 200 hidden units. Therefore, the upper problem optimizes the hidden layer with 157,000 parameters, and the lower problem optimizes the output layer with 2,010 parameters. For CIFAR-10, we train the 7-layer LeNet \cite{lecun1998lenet}, where we treat the last fully connected layer’s parameters as lower-level variables and the rest layers’ parameters as upper-level variables.

\textbf{Hyperparameter settings}. According to the optimal relation shown in (\ref{eq:optimal-omega}), we set the compression parameter $K=200$ for lower-level and $K=2000$ for upper-level. The dataset is partitioned to $10$ workers both under homogeneous and heterogeneous distributions. The batch size of workers' stochastic oracle is $512$ for MNIST and $1000$ for CIFAR-10. The moving average parameter $\theta$ of CM-SOBA and EF-SOBA is $0.1$. We optimize the stepsizes for all compared algorithms via grid search, each ranging from $[0.001,0.05,\dots,0.5]$, which is summarized in Table \ref{tab:stepsize1}.

\begin{table*}[ht]
\footnotesize
\centering
\vspace{-5mm}
\caption{Stepsize selection for experiments of hyper-representation}
    \begin{tabular}{ccc}
    \toprule
    Algorithm & Dataset & Stepsize $[\alpha,\beta,\gamma]$ \\
    \midrule
    NC-SOBA & MNIST &  [0.5, 0.1, 0.01]\\
    C-SOBA & MNIST &  [0.1, 0.1, 0.01]\\
    CM-SOBA & MNIST & [0.1, 0.1, 0.01]\\
    EF-SOBA & MNIST & [0.5, 0.1, 0.01]\\
    NC-SOBA & CIFAR-10 & [0.1, 0.001, 0.001]\\
    C-SOBA & CIFAR-10 & [0.05, 0.001, 0.001] \\
    \bottomrule
    \end{tabular}
\label{tab:stepsize1}
\end{table*}

\subsection{Hyperparameter Optimization}\label{app:exp_syn}
\textbf{Problem formulation.}
Hyperparameter optimization can be formulated as:
\begin{align*}
    \min_\lambda \;L(\lambda)&=\frac{1}{|\mathcal{D}_v|}\sum_{\xi\in \mathcal{D}_v}L(\omega^*(\lambda);\xi) \\
    s.t.\; \omega^*(\lambda)&=\arg\min_\omega \frac{1}{|\mathcal{D}_\tau|}\sum_{\eta\in \mathcal{D}_\tau}(L(\omega;\eta)+R(\omega,\lambda))
\end{align*}
where $L$ is the loss function, $R(\omega,\lambda)$ is a regularizer, $\mathcal{D}_v$ and $\mathcal{D}_\tau$ denote the validation set and training set. 
To perform logistic regression with regularization, following \citep{pedregosa2016hyperparameter,grazzi2020iteration,chen2022decentralized}, we define $L= \log(1+e^{-yx^T\omega})$ and $R=\frac{1}{2}\sum_{i=1}^pe^{\lambda_i}\omega_i^2$ on synthetic dataset.
For MNIST, we have the model parameter $\omega\in\mathbb{R}^{p\times c}$ with $p=784$ and $c=10$. Following \cite{grazzi2020iteration}, we set $L$ as the cross entropy loss and $R=\frac{1}{cp}\sum_{j=1}^c\sum_{i=1}^pe^{\lambda_i}\omega_{ij}^2$.

\textbf{Datasets.} We construct synthetic heterogeneous data by a linear model $y=\text{sign}(x^T\omega+\epsilon\cdot z)$, where $\epsilon=0.1$ is the noise rate and $z\in\mathbb{R}$ is the noise vector sampled from standard normal distribution. The distribution of $x\in\mathbb{R}^{100}$ on worker $i$ is $N(0,i^2)$ if $i\%2=0$ otherwise $\chi^2(i)$. 
Additionally, we assume there are 5 workers with 500 training data and 500 validation data respectively. For MNIST, We partition it to $10$ workers under both homogeneous and heterogeneous data distributions.

\textbf{Hyperparameter settings.} For the experiments in this study, where the upper and lower levels share the same compressed dimension, we use a uniform compression parameter for both levels: $K=10$ for the MNIST dataset and $K=20$ for the synthetic dataset. The batch size for the synthetic dataset is $50$ and for MNIST is $512$. We optimize the stepsizes for all compared algorithms via grid search, each ranging from $[0.001,0.05,\dots,0.5]$, which is summarized in Table \ref{tab:stepsize2}.
\begin{table*}[ht]
\footnotesize
\centering
\vspace{-5mm}
\caption{Stepsize selection for experiments of hyperparameter optimization}
    \begin{tabular}{ccc}
    \toprule
    Algorithm & Dataset & Stepsize $[\alpha,\beta,\gamma]$ \\
    \midrule
    NC-SOBA & Synthetic &  [0.1, 0.01, 0.001]\\
    C-SOBA & Synthetic &  [0.05, 0.01, 0.001]\\
    CM-SOBA & Synthetic & [0.05, 0.01, 0.001]\\
    EF-SOBA & Synthetic & [0.1, 0.01, 0.001]\\
    NC-SOBA & MNIST &  [0.1, 0.1, 0.1]\\
    C-SOBA & MNIST &  [0.1, 0.1, 0.1]\\
    CM-SOBA & MNIST & [0.1, 0.1, 0.1]\\
    EF-SOBA & MNIST & [0.1, 0.1, 0.1]\\
    \bottomrule
    \end{tabular}
\label{tab:stepsize2}
\end{table*}

\textbf{Additional results on synthetic data.} It can be seen from Figure \ref{fig:ho_syn} that under the heterogeneous data distribution, our proposed EF-SOBA outperforms with a similar convergence rate and much fewer communication bits. These results are also consistent with those on MNIST in Figure \ref{fig:ho_mnist_id}, which implies the broad application of our algorithms on various datasets and problem setups.
\begin{figure}
    \centering
    \begin{minipage}[t]{0.4\textwidth}
        \centering
        \includegraphics[width=6 cm]{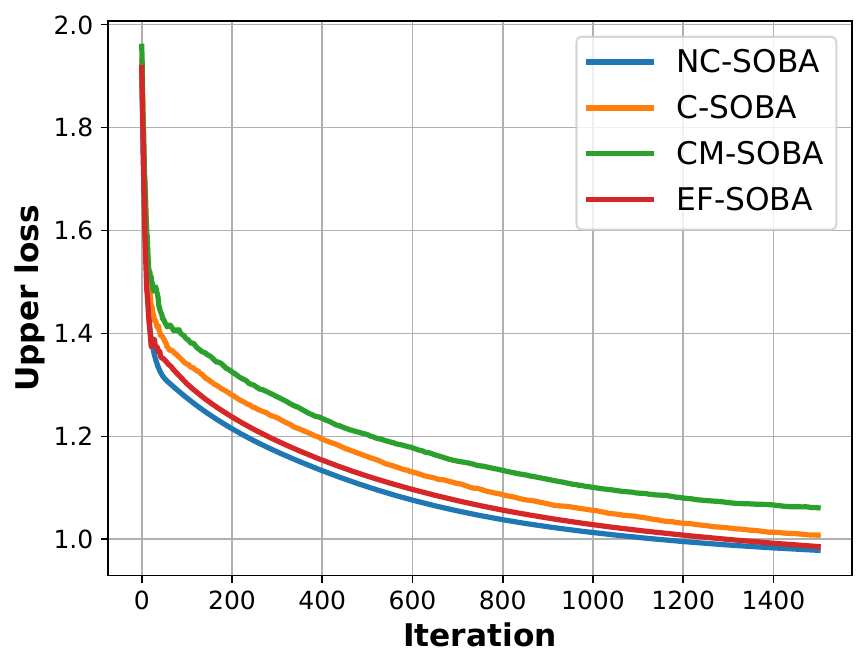}
    \end{minipage}
    \begin{minipage}[t]{0.4\textwidth}
        \centering
        \includegraphics[width=6 cm]{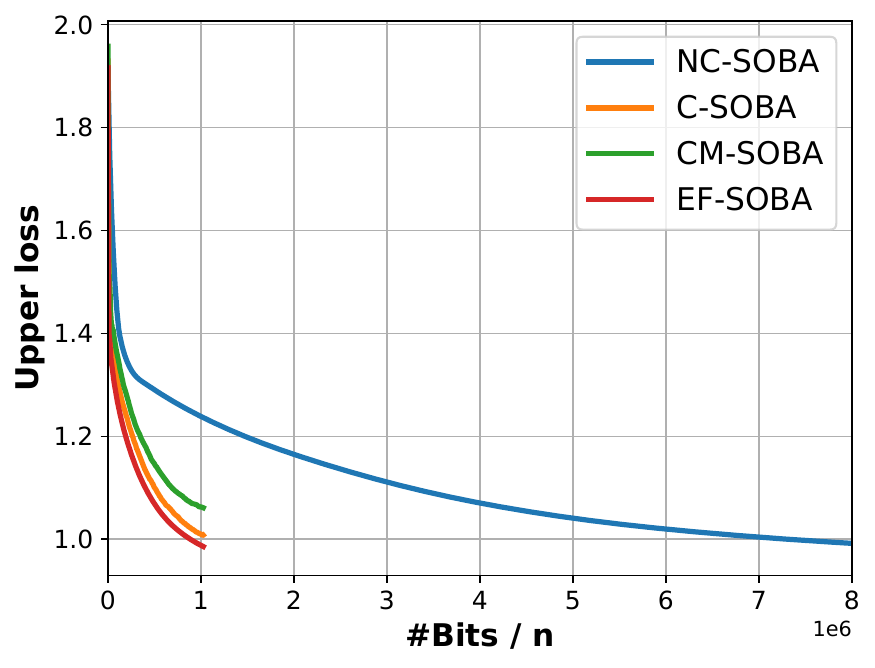}
    \end{minipage}
    \caption{Hyperparameter optimization on synthetic heterogeneous data.}
    \label{fig:ho_syn}
\end{figure}

\subsection{Additional Results}\label{app:exp_cifar}

\begin{figure}
    \centering
    \begin{minipage}[t]{0.4\textwidth}
        \centering
        \includegraphics[width=6 cm]{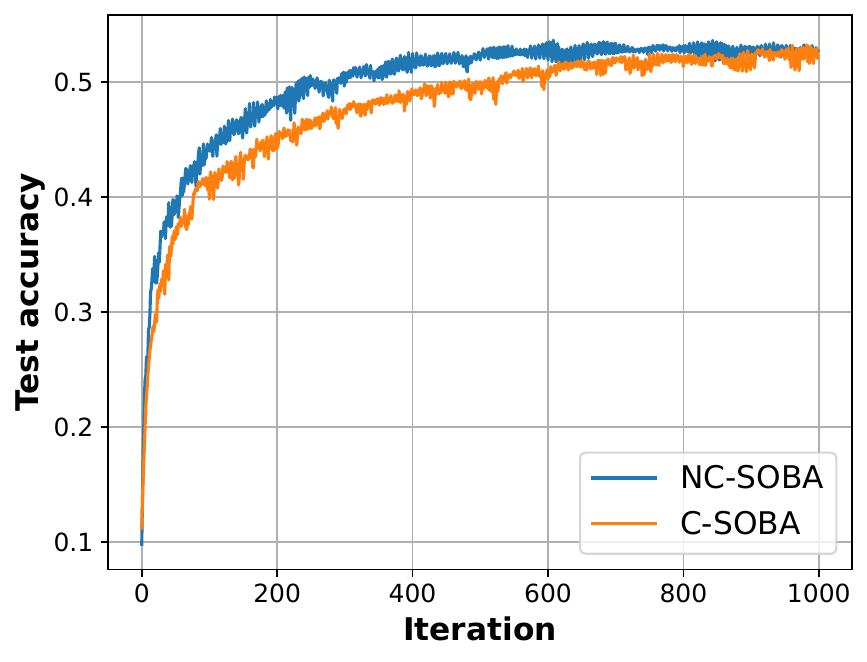}
        \label{fig:cifar_acc}
    \end{minipage}
    \begin{minipage}[t]{0.4\textwidth}
        \centering
        \includegraphics[width=6 cm]{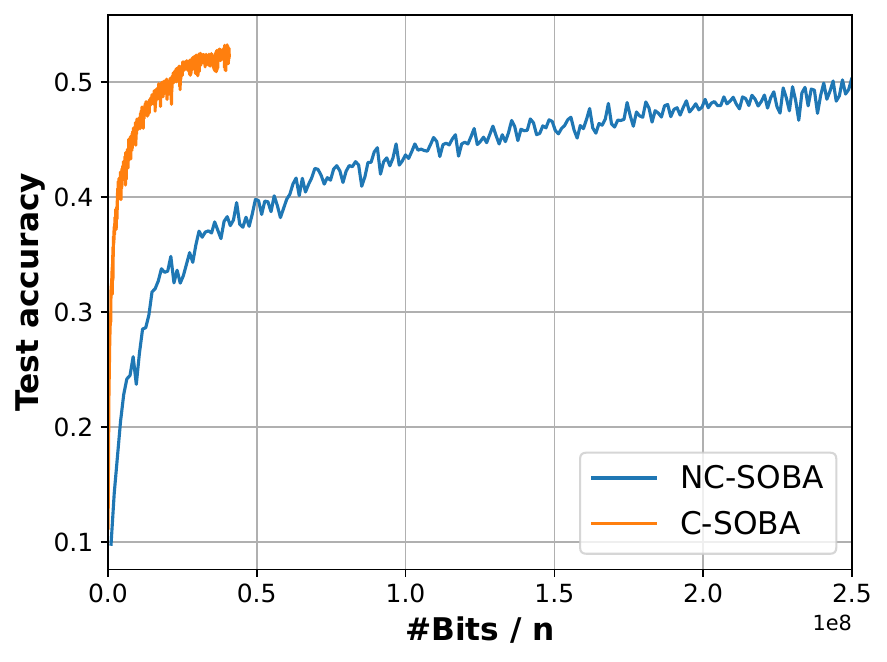}
        \label{fig:cifar_acc_bt}
    \end{minipage}
    \caption{Hyper-representation on CIFAR-10 under homogeneous data distributions.}
    \label{fig:hr_cifar}
\end{figure}

\textbf{Results on CIFAR-10. }Under the experimental setup in Appendix \ref{app:exp_form}, we evaluate the performance of compressed algorithms on CIFAR-10 under homogeneous data distributions. We only draw the result of C-SOBA here to compare with NC-SOBA  because from results on MNIST we can see other algorithms perform worse than C-SOBA under the homogeneous data distributions. From Figure \ref{fig:hr_cifar}, it can be seen that with nearly $10\times$ communication bits savings, our compressed algorithm converges to the same test accuracy as non-compressed algorithm. It validates the effectiveness of our proposed algorithms even under complicated model architecture and large dataset. Notice that the backbone test accuracy is not satisfactory here, we suspect that it's because the bilevel structure of the hyper-representation problem brings challenges to the training.


{\textbf{More comparison baselines. }In our study, we evaluate our compression algorithms for distributed stochastic bilevel optimization against the SOBA algorithm, which serves as our non-compression baseline (referred to as NC-SOBA). Furthermore, we include FedNest\cite{tarzanagh2022fednest} as another non-compression baseline for comparison with SOBA. We conduct hyper-representation experiments on the MNIST dataset, employing an MLP backbone and utilizing homogeneous data distributions. We implement FedNest based on its publicly available source code. As illustrated in Fig.~\ref{fig:fednest}, it is evident that NC-SOBA achieves faster convergence in terms of communication bits compared to FedNest.}

\begin{figure}[ht]
    \centering
    \includegraphics[width=.4\textwidth]{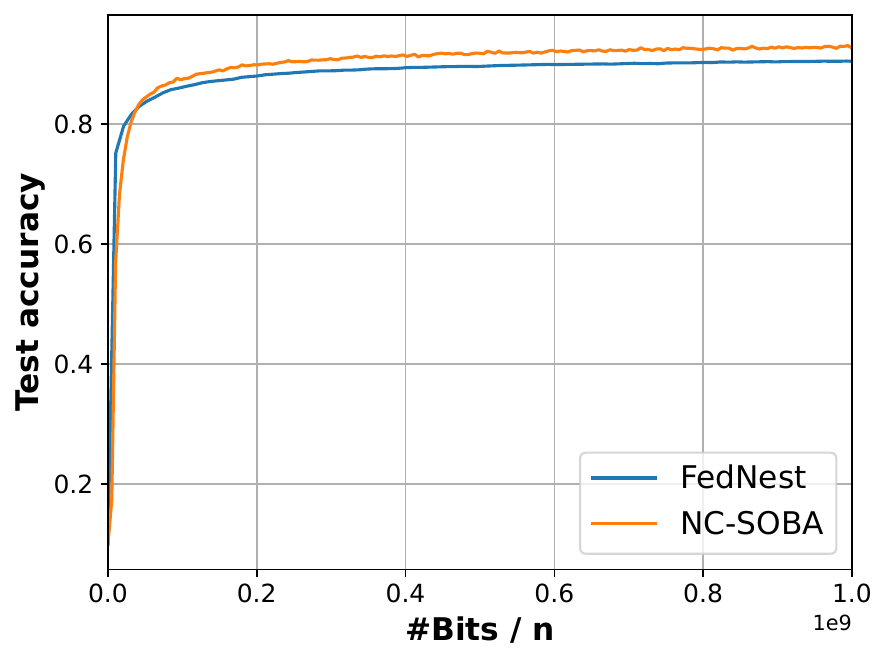}
    \caption{Hyper-representation on MNIST under homogeneous data distributions.}
    \label{fig:fednest}
\end{figure}


{\textbf{Tuning the momentum parameter in CM-SOBA. }In Fig.~\ref{fig:hr_mnist_nd}, CM-SOBA performs inferiorly to C-SOBA due to the momentum parameter $\theta$ being set to a fixed value of 0.1, without further optimization, which may lead to sub-optimal results. To mitigate this limitation, we conducted additional experiments employing a refined approach for selecting the momentum parameter. The results are depicted in Fig.~\ref{fig:nd_acc}, illustrating that with an appropriately tuned momentum parameter, CM-SOBA indeed outperforms C-SOBA. This underscores the significance of momentum parameter optimization for the effective implementation of CM-SOBA.}

\begin{figure}[ht]
    \centering
    \includegraphics[width=.4\textwidth]{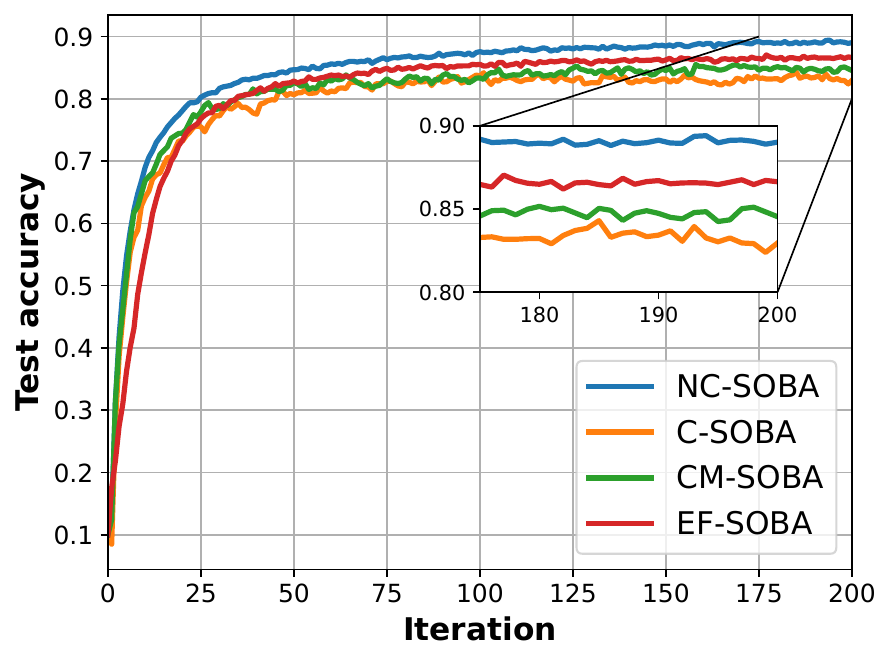}
    \caption{Hyper-representation on MNIST under heterogeneous data distributions.}
    \label{fig:nd_acc}
\end{figure}



\subsection{Ablation Studies}\label{subapp:ablation}
\textbf{Ablation on MSC rounds. }We propose algorithm variants in Sec.~\ref{sec:varia} to enhance theoretical convergence rate, by utilizing the multi-step compression and gradient accumulation mechanism. It's worth noting that when CM-SOBA (Alg.~\ref{alg:C-SOBA}) is a special case of CM-SOBA-MSC (Alg.~\ref{alg:CM-SOBA-MSC}) with $R=1$. Same thing happens to EF-SOBA (Alg.~\ref{alg:EF-SOBA}) and EF-SOBA-MSC (Alg.~\ref{alg:EF-SOBA-MSC}). Thus a natural question is, how should we select $R$ in practice, and whether $R>1$ can be more effective than $R=1$? To address this issue, we conduct ablation experiments on the hyperparameter optimization task
on MNIST dataset. The problem formulation, data and stepsizes are set consistent with Appendix \ref{app:exp_syn}.

\begin{figure}[tbp]
\centering
    \begin{minipage}[t]{0.4\textwidth}
        \centering
        \includegraphics[width=6 cm]{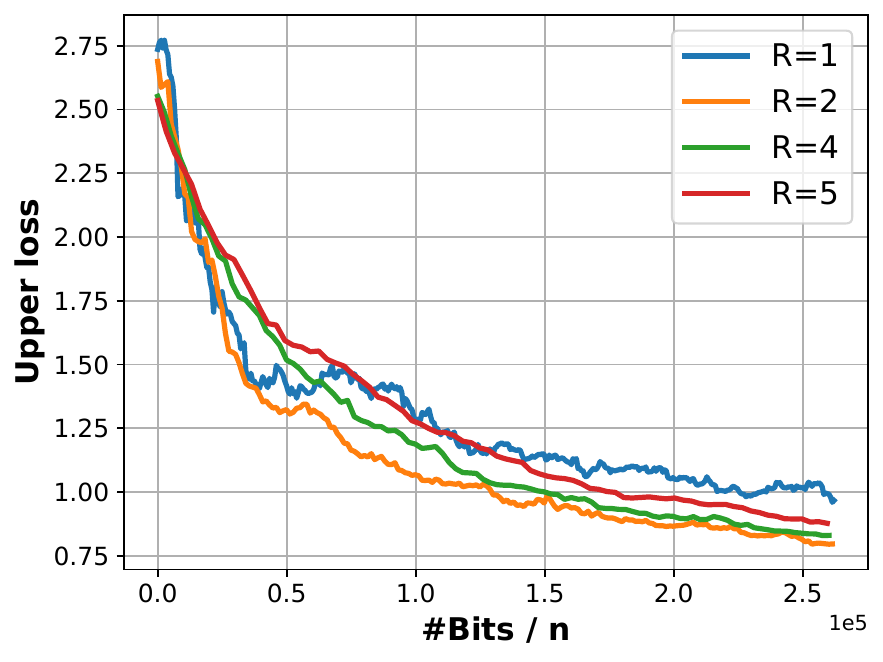}
    \end{minipage}
    \begin{minipage}[t]{0.4\textwidth}
        \centering
        \includegraphics[width=6 cm]{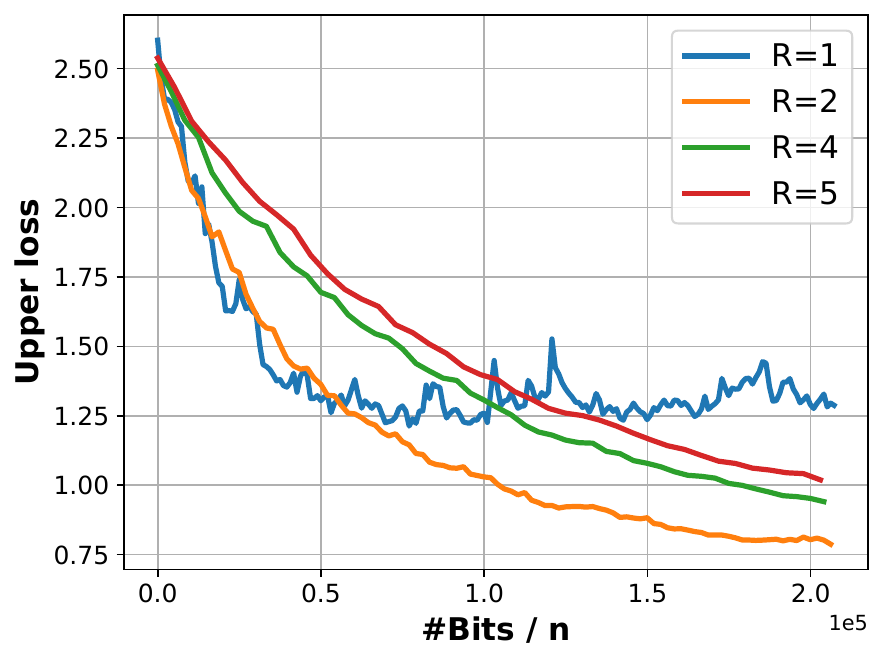}
    \end{minipage}
    \caption{Ablation on MSC rounds $R$ for CM-SOBA-MSC (left) and EF-SOBA-MSC (right), conducted on hyperparameter optimization task on MNIST heterogeneous data.}
    \label{fig:msc_ablation}
\end{figure}

\cref{fig:msc_ablation} displays the loss curve of CM-SOBA-MSC (left) and EF-SOBA-MSC (right) with different $R$'s in the hyperparameter optimization task on MNIST under heterogeneous data distributions. With $R=2$, both algorithms perform better than those with $R=1,4,5$. We demonstrate that when $R$ is too small, the gradient bias induced by compression error and sampling randomness slows down the convergence, while a much larger $R$ trades communication/computation savings to update directions with little improvement, making it less effective. Generally speaking, there is a trade-off in the selection of $R$, and we recommend choosing suitable $R$'s by cross validation.

One can also observe from \cref{fig:msc_ablation} that EF-SOBA-MSC with $R=1$ (which is exactly EF-SOBA) has a worse performance than CM-SOBA-MSC with $R=1$ (which is exactly CM-SOBA), even if the data is constructed heterogeneously. This phenomenon is consistent with our convergence results, that EF-SOBA is more susceptible to large $\omega$'s than CM-SOBA. Consequently, we recommend using EF-SOBA-MSC with $R>1$ when severely aggressive compressors are applied.

\begin{table}[htbp]
    \centering
    \caption{Compressor choices under different strategies.}
    \label{tab:omega_ablation}
    \begin{tabular}{cccc}
    \hline
         Strategy & $K$ for lower-level rand-$K$ & $K$ for upper-level rand-$K$ & Communicated entries per iter\\\hline
         $\sqrt{d_x}:\sqrt{d_y}$ & 6 & 68 & 80\\
         $1:1$ & 1 & 78 & 80\\\hline
    \end{tabular}
    \end{table}
\vspace{-5pt}
\begin{figure}[htbp]
    \centering
    \includegraphics[width=.4\textwidth]{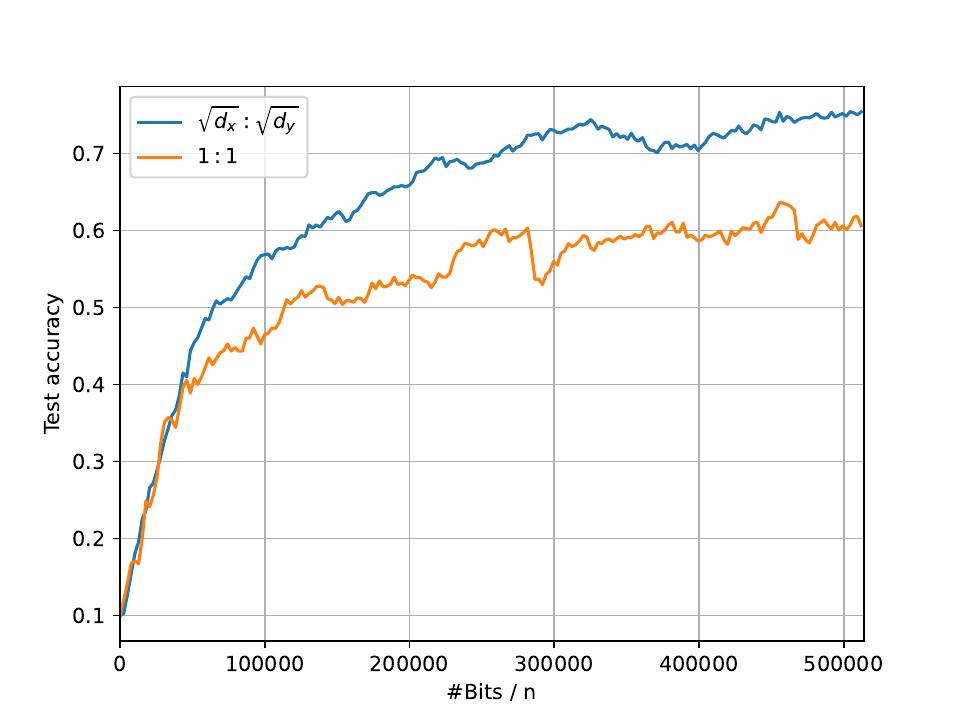}
    \caption{Ablation on compressor choices conducted on hyper-presentation optimization task on MNIST heterogeneous data.}
    \label{fig:omega_ablation}
\end{figure}



{\textbf{Ablation on compressor choices. } We evaluate various compressor choices while maintaining consistent per-round communication cost constraints. In Fig.~\ref{fig:omega_ablation}, we present the performance comparison of C-SOBA on the hyper-presentation problem using the MNIST dataset, where $d_x=157000$ and $d_y=2010$. All experiments employ identical learning rates: $\alpha=$ 2e-2, $\beta=$ 8e-3, and $\gamma=$ 8e-4. The Rand-$K$ compressors, outlined in Table \ref{tab:omega_ablation}, are selected based on two strategies:
\begin{itemize}
    \item Strategy 1: $\frac{\omega_u}{\omega_l}\approx\sqrt{\frac{d_x}{2d_y}}=\Theta\left(\sqrt{\frac{d_x}{d_y}}\right)$;
    \item Strategy 2: $\frac{\omega_u}{\omega_l}\approx\frac{1}{2}=\Theta\left(\frac{1}{1}\right)$.
\end{itemize} 
It is evident that Strategy 1 (as recommended in \eqref{eq:optimal-omega}) outperforms Strategy 2.}
\end{document}